\documentclass{amsart}
\usepackage{cases}
\usepackage{amssymb,amsmath}
\usepackage{pifont}
\usepackage{amsfonts}
\usepackage{txfonts}
\usepackage{mathrsfs}
\usepackage{bbm}
\usepackage{enumerate}
\usepackage{color}
\usepackage{ulem}

\theoremstyle{plain}
\newtheorem{theorem}{Theorem}[section]
\newtheorem{remark}[theorem]{Remark}
\newtheorem{lemma}[theorem]{Lemma}
\newtheorem{proposition}[theorem]{Proposition}
\newtheorem{corollary}[theorem]{Corollary}

\numberwithin{equation}{section}

\begin{document}

\title[Non-cutoff Boltzmann equation]
{Sharp regularity properties for the non-cutoff\\
spatially homogeneous Boltzmann equation}

\author[L. Glang\'etas, H.-G. Li \& C.-J. Xu]
{L. Glangetas, H.-G. LI and C.-J. Xu}

\address{L\'eo Glangetas,
\newline\indent
Universit\'e de Rouen, CNRS UMR 6085, Math\'ematiques
\newline\indent
76801 Saint-Etienne du Rouvray, France}
\email{leo.glangetas@univ-rouen.fr}

\address{Hao-Guang Li,
\newline\indent
School of Mathematics and statistics, Wuhan University 430072,
Wuhan, P. R. China
\newline\indent
School of mathematics and statistics, South-Central University for Nationalities 430074,
\newline\indent
Wuhan, P. R. China}
\email{lihaoguang@mail.scuec.edu.cn}

\address{Chao-Jiang Xu,
\newline\indent
School of Mathematics and statistics, Wuhan University 430072,
Wuhan, P. R. China
\newline\indent
Universit\'e de Rouen, CNRS UMR 6085, Math\'ematiques
\newline\indent
76801 Saint-Etienne du Rouvray, France}
\email{chao-jiang.xu@univ-rouen.fr}

\date{\today}

\subjclass[2000]{35Q20, 36B65}

\keywords{Boltzmann equation, spectral decomposition, Gelfand-Shilov class}

\begin{abstract}
In this work, we study the Cauchy problem for the spatially homogeneous non-cutoff Boltzamnn equation with Maxwellian molecules.
We prove that this Cauchy problem
{\color{black} enjoys Gelfand-Shilov's regularizing effect},
{\color{black}meaning that} the smoothing properties
{\color{black} are the} same as the Cauchy problem defined by the evolution equation associated to a fractional harmonic oscillator.
The power of {\color{black} the fractional exponent}
is exactly the same as the singular index of {\color{black} the} non-cutoff collisional
kernel of {\color{black} the} Boltzmann equation.
{\color{black} Therefore}, we get the sharp regularity of solutions in the {\color{black} Gevrey} class
and {\color{black} also the sharp decay of solutions with an exponential {\color{black} weight}}.
We also give a method to construct the solution of the Boltzmann equation
by solving an infinite {\color{black} system} of ordinary differential equations.
The key {\color{black} tool} is the spectral decomposition of linear and non-linear Boltzmann operators.
\end{abstract}

\maketitle
\tableofcontents


\section{Introduction}\label{S1}
In this work, we consider the spatially homogeneous Boltzmann equation
\begin{equation}\label{eq1.10}
\left\{
\begin{array}{ll}
   \partial_t f= Q(f,f),\\
  f|_{t=0}=f_0,
\end{array}
\right.
\end{equation}
where $f=f(t,v)$ is the density distribution function {\color{black}depending} on the variables $t\geq0$~and $v\in\mathbb{R}^{3}$. The Boltzmann bilinear collision operator is given by
\begin{equation*}
Q(g,f)(v)
=\int_{\mathbb{R}^{3}}
\int_{\mathbb{S}^{2}}
B(v-v_{\ast},\sigma)
(g(v_{\ast}^{\prime})f(v^{\prime})-g(v_{\ast})f(v))
dv_{\ast}d\sigma,
\end{equation*}
where for $\sigma\in \mathbb{S}^{2}$,~the symbols~$v_{\ast}^{\prime}$~and~$v^{\prime}$~are abbreviations for the expressions,
$$
v^{\prime}=\frac{v+v_{\ast}}{2}+\frac{|v-v_{\ast}|}{2}\sigma,\,\,\,\,\, v_{\ast}^{\prime}
=\frac{v+v_{\ast}}{2}-\frac{|v-v_{\ast}|}{2}\sigma,
$$
which are obtained in such a way that collision preserves momentum and kinetic energy,~namely
$$
v_{\ast}^{\prime}+v^{\prime}=v+v_{\ast},\,\,\,\,\, |v_{\ast}^{\prime}|^{2}+|v^{\prime}|^{2}=|v|^{2}+|v_{\ast}|^{2}.
$$
The non-negative cross section $B(z,\sigma)$~depends only on $|z|$~and the scalar product $\frac{z}{|z|}\cdot\sigma$.~For physical models,~it usually takes the form
$$
B(v-v_{\ast},\sigma)=\Phi(|v-v_{\ast}|)b(\cos\theta),~~~~
\cos\theta=\frac{v-v_{\ast}}{|v-v_{\ast}|}\cdot\sigma,~~ 0\leq\theta\leq\frac{\pi}{2}.
$$

In this paper, we consider only the Maxwellian molecules case which {\color{black}corresponds} to {\color{black}the case} $\Phi\equiv 1$,
and we focus our attention on the angular part $b$ satisfying
\begin{equation}\label{beta}
\beta(\theta)=
2\pi b(\cos2\theta) |\sin2\theta| \approx |\theta|^{-1-2s},\,\,\mbox{when}~\theta\rightarrow0^{+},
\end{equation}
for some $0<s<1${\color{black}. Without} loss of generality, we may assume that $b(\cos\theta)$ is supported on the set $\cos\theta\geq0$.
{\color{black} See for instance \cite{NYKC1} for more details on $\beta(\,\cdot\,)$ and \cite{Villani} for a general collision kernel}.

We linearize the Boltzmann equation near the absolute Maxwellian distribution
$$
\mu(v)=(2\pi)^{-\frac 32}e^{-\frac{|v|^{2}}{2}}.
$$
{\color{black} We introduce $g(t,v)$ such that $f(t,v)=\mu(v)+\sqrt{\mu}(v)g(t,v)$ and
obtain}
$$
\frac{\partial g}{\partial t}+\mathcal{L}[g]={\bf \Gamma}(g, g)
$$
with
$$
{\bf \Gamma}(g, h)=\frac{1}{\sqrt{\mu}}Q(\sqrt{\mu}g,\sqrt{\mu}h),\,\,
\mathcal{L}(g)=-\frac{1}{\sqrt{\mu}}[Q(\sqrt{\mu}g,\mu)+Q(\mu,\sqrt{\mu}g)].
$$
{\color{black} Therefore} the Cauchy problem \eqref{eq1.10} can be re-writed in the form
\begin{equation} \label{eq-1}
\left\{ \begin{aligned}
         &\partial_t g+\mathcal{L}(g)={\bf \Gamma}(g, g),\,\\
         &g|_{t=0}=g_{0}.
\end{aligned} \right.
\end{equation}
The linear operator $\mathcal{L}$ is nonnegative (\cite{NYKC1,NYKC2,NYKC3}),\,with the null space
$$
\mathcal{N}=\text{span}\left\{\sqrt{\mu},\,\sqrt{\mu}v_1,\,\sqrt{\mu}v_2,\,
\sqrt{\mu}v_3,\,\sqrt{\mu}|v|^2\right\}.
$$

In the present work, we study the smoothing effect for the Cauchy problem asso\-cia\-ted
{\color{black} with} the  spatially homogeneous non-cutoff Boltzmann equation
{\color{black} in the case of} Maxwellian molecules.
It is well known that the non-cutoff spatially homogeneous Boltzmann equation
{\color{black} enjoys} an $\mathscr{S}(\mathbb{R}^3)$-regularizing effect
for the weak solutions to the Cauchy problem (see \cite{DW,YSCT}).
Regarding the Gevrey regularity, Ukai showed in \cite{Ukai}
that the Cauchy problem for the Boltzmann equation has a unique local solution in Gevrey classes.
{\color{black} Then} Desvillettes, Furioli and Terraneo proved in \cite{DFT}
the propagation of Gevrey regularity for solutions of the Boltzmann equation with Maxwellian mole\-cules.
For mild singularities, Morimoto and Ukai proved in \cite{MU} the Gevrey regularity
of smooth Maxwellian decay solutions to the Cauchy problem of the spatially homogeneous
Boltzmann equation with a modified kinetic factor
{\color{black} (see also \cite{TZ} for the non-modified case)}.
On the other hand, Lekrine and Xu proved in \cite{L-X} the property of
Gevrey smoothing effect for the weak solutions to the Cauchy problem associated to the radially symmetric spatially
homogeneous Boltzmann equation with Maxwellian molecules for $0 < s < 1/2$. This result was then completed by Glangetas
and Najeme who established in \cite{G-N} the analytic smoothing effect in the case when $1/2 \leq s <1$.
{\color{black} In \cite{ NYKC1}, for the radially symmetric case,
the linearized Boltzmann operator was shown to {\color{black} behave,}
essentially as a fractional harmonic oscillator $\mathcal{H}^s$,
with $\mathcal{H}=-\triangle +\frac{|v|^2}{4}$ and $0 < s <1$.
For the non-radial case, the linearized non-cutoff Boltzmann operator
{\color{black} behaves} as
$$\mathcal{L}=a(\mathcal{H},\Delta_{\mathbb{S}^2})\mathcal{L}^s_{L}$$
where $\Delta_{\mathbb{S}^2}$ is the Laplace-Beltrami operator on the sphere,
$\mathcal{L}_{L}$ is the linearized Landau operator and
$a(\mathcal{H},\Delta_{\mathbb{S}^2})$ is an isomorphism on $L^2(\mathbb{R}^3)$.
Here, the fractional power $\mathcal{L}_{L}^s$ is defined
through functional calculus.
We can refer to \cite{NYKC2}.
}
The solutions of the following Cauchy problem
$$
\left\{ \begin{aligned}
         &\partial_t g+\mathcal{L}(g)=0,\,\\
         &g|_{t=0}=g_{0}\in L^2,
\end{aligned} \right.
$$
belong to the symmetric Gelfand-Shilov {\color{black} space} $S^{1/2s}_{1/2s}(\mathbb{R}^3)$ for any positive time and
$$
{\color{black} \forall\, t>0},\quad \|e^{ct \mathcal{H}^s}g(t)\|_{L^2}\leq\,C\|g_0\|_{L^2},
$$
where the Gelfand-Shilov {\color{black} space} $S^{\mu}_{\nu}(\mathbb{R}^3)$, with $\mu,\,\nu>0,$\,$\mu+\nu\geq1,$\, is the space of smooth functions $f\in\,C^{+\infty}(\mathbb{R}^3)$ satisfying:
$$
\exists\, A>0,\, C>0,\,
\sup_{v\in\mathbb{R}^3}|v^{\beta}\partial^{\alpha}_vf(v)|\leq\,CA^{|\alpha|+
|\beta|}(\alpha!)^{\mu}(\beta!)^{\nu},\,\,\forall\,\alpha,\,\beta\in\mathbb{N}^3.
$$
{\color{black} This} Gelfand-Shilov space can {\color{black} also} be characterized as the
sub-space of Schwartz functions $f\in\,\mathscr{S}(\mathbb{R}^3)$ such that,
$$
\exists\, C>0,\,\epsilon>0,\,|f(v)|\leq Ce^{-\epsilon|v|^{\frac{1}{\nu}}},\,\,v\in\mathbb{R}^3\,\,\text{and}\,\,|\hat{f}(\xi)|\leq Ce^{-\epsilon|\xi|^{\frac{1}{\mu}}},\,\,\xi\in\mathbb{R}^3.
$$
The symmetric Gelfand-Shilov space $S^{\nu}_{\nu}(\mathbb{R}^3)$ with $\nu\geq\frac{1}{2}$ can {\color{black} also} be {\color{black} identified} with
$$
S^{\nu}_{\nu}(\mathbb{R}^{3})=\left\{f\in C^\infty (\mathbb{R}^3);  \exists \tau>0,
\|e^{\tau \, \mathcal{H}^{\frac{1}{2\nu}}}f\|_{L^2}<+\infty\right\}.
$$
See Appendix \ref{S6} for more properties of Gelfand-Shilov spaces.

From a historical point of view,~the spectral analysis is
{\color{black}an} {\color{black}important} method {\color{black} to study} the linear Boltzmann operator ({\color{black}see \cite{Cerci}}). In \cite{NYKC1}, the linearized non-cutoff radially symmetric Boltzmann operator is shown to be diagonal in the Hermite basis.
{\color{black} This property has been used to prove
in the continued work \cite{NYKC3} that the Cauchy problem
{\color{black} of} the non-cutoff spatially homogeneous Boltzmann equation
with a radial initial datum $g_0\in L^2(\mathbb{R}^3)$,
has a unique global radial solution which belongs to the Gelfand-Shilov class $S^{1/2s}_{1/2s}(\mathbb{R}^3)$.
}

{\color{black}The main theorem of this paper is given in the following.}
\begin{theorem}\label{trick}
Assume that the Maxwellian collision cross-section $b(\,\cdot\,)$
is given in~$\eqref{beta}$ with $0<s<1$,
then there exists $\varepsilon_0>0$ such that for any initial datum
$g_0\in\,L^2(\mathbb{R}^3)\bigcap\mathcal{N}^{\perp}$
with $\|g_0\|^2_{L^2(\mathbb{R}^3)}\le \varepsilon_0$,
the Cauchy problem \eqref{eq-1} {admits} a {\color{black} weak} solution,
{\color{black} which} belongs to the Gelfand-Shilov space $S^{1/2s}_{1/2s}(\mathbb{R}^3)$ for any $t>0$.
Moreover, there exists  $c_0>0$, such that, for any $t\ge 0$,
\begin{equation}\label{rate}
\|e^{{c_0t} \, \mathcal{H}^s}g(t)\|_{L^2(\mathbb{R}^3)}
\leq\,
Ce^{-\frac{\lambda_{2,0}}{4} t} \, \|g_0\|_{L^2(\mathbb{R}^3)},
\end{equation}
where
$$
\lambda_{2,0}=\int^{\pi/4}_{-\pi/4}\beta( \theta)(1-\sin^4\theta-\cos^4\theta)d\theta>0.
$$
\end{theorem}

\begin{remark}
We have proved that for the Cauchy problem \eqref{eq1.10}, if the initial data is a small perturbation of Maxwellian in $L^2$, then the global
solution {\color{black}returns} to the equilibrium with an exponential rate with respect to Gelfand-Shilov norm, {\color{black} which is}
{\color{black}an} {exponentially} weighted norm
{\color{black} of both {\color{black} the} solution and the Fourier transformation of the solution.}
\end{remark}

The rest of the paper is arranged as follows:
In Section \ref{S2}, we introduce the spectral ana\-ly\-sis
of the linear and nonlinear Boltzmann {\color{black}operators},
and transform the nonlinear Cauchy problem of Boltzmann equation
to an infinite {\color{black} system} of ordinary differential
{\color{black}equations} which can be solved explicitly.
{\color{black} Then} we {\color{black} derive} the formal solution
of the Cauchy problem for Boltzmann equation.
In Section \ref{S3}, we establish an upper bounded estimates
of {\color{black} some} nonlinear operators with an exponential weighted {\color{black} norm},
which is crucial to {\color{black} obtain} the convergence
of {\color{black} the} formal solution in Gelfand-Shilov {\color{black} spaces}.
The proof of the main Theorem \ref{trick} will be presented in Section \ref{S4}.
Finally, {\color{black} Section \ref{S5-1} and Section \ref{S5-2} are}
devoted to the proof of {\color{black} some propositions} used in Section \ref{S4}.
In {\color{black} Section} \ref{S5-1}, we study the spectral re\-presentation
of {\color{black} the non linear Boltzmann operator},
and prove that it can be represented
by an ``inferior triangular matrix'' of infinite dimension
with three {\color{black}indices},
so that the presentation and the computations are very
{\color{black} complicated}.
This inferior triangular property is essential
for the construction of the formal solution
by solving an infinite system of ordinary differential equations.
In Section \ref{S5-2}, we
 prove some estimates on the entries
of the triangular matrix obtained in
 Section \ref{S5-1},
 and {\color{black}this is} a key point in proving the convergence
of the formal solution with respect to Gelfand-Shilov {\color{black} norms}.

\section{The spectral analysis of the Boltzmann operators }\label{S2}

\subsection{Diagonalization of {\color{black} the} linear operators}

We first recall the spectral decomposition of {\color{black} the} linear Boltzmann operator.
In the cutoff case, that is, when $b(\cos\theta)\sin\theta\in\,L^1([0,\frac{\pi}{2}])$, it was shown in \cite{WU} that
\begin{equation*}
\mathcal{L}(\varphi_{n, l, m})=\lambda_{n,l}\, \varphi_{n, l, m}, \,\,n,l\in\mathbb{N},\,m\in\mathbb{Z}, |m|\leq l.
\end{equation*}
This diagonalization of the linearized Boltzmann operator with Maxwellian molecules holds
as well in the non-cutoff case, (see \cite{Boby,Cerci,Dole,NYKC1,NYKC2}).\,\,
The eigenvalues are 
\begin{equation*}
\lambda_{n,l}=\int_{|\theta|\le \frac{\pi}{4}}\beta(\theta)\Big(1+\delta_{n, 0}\delta_{l, 0}
-(\sin\theta)^{2n+l}P_{l}(\sin\theta)-(\cos\theta)^{2n+l}P_{l}(\cos\theta)\Big)d\theta,
\end{equation*}
the eigenfunctions are
\begin{equation}\label{v}
\varphi_{n,l,m}(v)=\left(\frac{n!}{\sqrt{2}\Gamma(n+l+3/2)}\right)^{1/2}
\left(\frac{|v|}{\sqrt{2}}\right)^{l}e^{-\frac{|v|^{2}}{4}}
L^{(l+1/2)}_{n}\left(\frac{|v|^{2}}{2}\right)Y^{m}_{l}\left(\frac{v}{|v|}\right),
\end{equation}
where $\Gamma(\,\cdot\,)$ is the standard Gamma function, for any $x>0$,
$$\Gamma(x)=\int^{+\infty}_0t^{x-1}e^{-x}dx.$$
The $l^{th}$-Legendre polynomial~$P_{l}$ and the Laguerre polynomial $L^{(\alpha)}_{n}$~of order $\alpha$,~degree $n$ read,
\begin{align*}
&P_{l}(x)=\frac{1}{2^ll!}\frac{d^l}{dx^l}(x^2-1)^l,\,\,\text{where}\,|x|\leq1;\\
&L^{(\alpha)}_{n}(x)=\sum^{n}_{r=0}(-1)^{n-r}\frac{\Gamma(\alpha+n+1)}{r!(n-r)!
\Gamma(\alpha+n-r+1)}x^{n-r}.
\end{align*}
{\color{black} We refer the properties of these special functions to the classical book
\cite{San} (see $(7)$ of Sec.1 in Chap. III and $(3_1)$ of Sec. 1 in Chap. IV)}.
For any unit vector
$$\sigma=(\cos\theta,\sin\theta\cos\phi,\sin\theta\sin\phi)$$
with $\theta\in[0,\pi]$~and~$\phi\in[0,2\pi]$,~the orthonormal basis of spherical harmonics~$Y^{m}_{l}(\sigma)$ is
\begin{equation*}
Y^{m}_{l}(\sigma)=N_{l,m}P^{|m|}_{l}(\cos\theta)e^{im\phi},\,\,|m|\leq l,
\end{equation*}
where the normalisation factor is given by
$$
N_{l,m}=\sqrt{\frac{2l+1}{4\pi}\cdot\frac{(l-|m|)!}{(l+|m|)!}}
$$
and $P^{|m|}_{l}$~is the associated Legendre functions of the first kind of order $l$ and degree $|m|$ with
\begin{equation}\label{Plm}
P^{|m|}_{l}(x)= (1-x^2)^\frac{|m|}{2}
\left(\frac{\mathrm{d}}{\mathrm{d}x}\right)^{|m|} P_{l}(x).
\end{equation}
The family $\Big(Y^m_l(\sigma)\Big)_{l\geq0,|m|\leq\,l}$ constitutes an orthonormal basis of the space $L^2(\mathbb{S}^2,\,d\sigma)$ with $d\sigma$ being the surface measure on $\mathbb{S}^2$ (see {\color{black} $(16)$ of Chap.1 in the book} \cite{Jones}
).  Noting that $\left\{\varphi_{n,l,m}(v)\right\}$ consist an orthonormal basis of $L^2(\mathbb{R}^3)$ composed of eigenvectors of the harmonic oscillator
(see\cite{Boby}, \cite{NYKC2})
\begin{equation*}
\mathcal{H}(\varphi_{n, l, m})=(2n+l+\frac 32)\, \varphi_{n, l, m}.
\end{equation*}
As a special case, $\left\{\varphi_{n, 0, 0}(v)\right\}$ {\color{black}is} an orthonormal basis of $L^2_{rad}(\mathbb{R}^3)$, the radially sym\-me\-tric function space (see \cite{NYKC3}).
We have that, for suitable functions
$g$,
\begin{equation*}
\mathcal{L}(g)=\sum^{\infty}_{n=0}\sum^{\infty}_{l=0}\sum^{l}_{m=-l}
\lambda_{n,l}\, g_{n,l,m}\, \varphi_{n, l, m},
\end{equation*}
where $g_{n,l,m}=(g, \varphi_{n,l,m})_{L^2(\mathbb{R}^3)}$, and
\begin{equation*}
\mathcal{H}(g)=\sum^{\infty}_{n=0}\sum^{\infty}_{l=0}\sum^{l}_{m=-l}(2n+l+\frac 32)\, g_{n,l,m}\, \varphi_{n, l, m}\, .
\end{equation*}
Using this spectral decomposition, the definition of {\color{black}$\mathcal{H}^s,\, e^{c\mathcal{H}^s},\, e^{c\mathcal{L}}$} is then classical.

\smallskip
\subsection{Triangular effect of the non linear operators}
We {\color{black} now} study the algebra property of the nonlinear terms
$$
{\color{black}{\bf \Gamma}(\varphi_{n,l,m},
\varphi_{\tilde{n},\tilde{l},\tilde{m}}).}
$$
We have the following triangular effect for the nonlinear Boltzmann operators on the basis $\{\varphi_{n,l,m}\}$.
\begin{proposition}\label{expansion}
The following algebraic identities hold,
\begin{align}\label{k_0}
&(i_1) \quad\,\,\, {\bf \Gamma}(\varphi_{0,0,0},\varphi_{\tilde{n},\tilde{l},\tilde{m}})=
\lambda^{1}_{\tilde{n},\tilde{l}}\,\varphi_{\tilde{n},\tilde{l},\tilde{m}};\nonumber\\
&(i_2) \quad\,\,\,  {\bf \Gamma}(\varphi_{n,l,m},\varphi_{0,0,0})=\lambda^{2}_{n, l}\,\varphi_{n,l,m};\nonumber\\
&(ii_1) \quad\,\, {\bf \Gamma}(\varphi_{n,0,0},\varphi_{\tilde{n},\tilde{l},\tilde{m}})=\lambda^{rad, 1}_{n,\tilde{n},\tilde{l}}\,\varphi_{n+\tilde{n},\tilde{l},\tilde{m}},\, \,\mbox{for}\,\,n\geq1;\nonumber\\
&(ii_2) \quad\,\, {\bf \Gamma}(\varphi_{n,l,m},\varphi_{\tilde{n},0,0})\nonumber=\lambda^{rad, 2}_{n, \tilde{n}, l}\,\varphi_{n+\tilde{n},l,m},\, \, \,\mbox{for}\,\,n\in\mathbb{N},\,\,l\geq1;\qquad\quad\nonumber\\
&(iii) \quad\,\,{\bf \Gamma}(\varphi_{n,l,m},\varphi_{\tilde{n},\tilde{l},\tilde{m}})=
\sum^{k_0(l,\tilde{l},m,\tilde{m})}_{k=0}
\mu^{m,\tilde{m},m+\tilde{m}}_{n,\tilde{n},l,
\tilde{l},k}\,\,\varphi_{n+\tilde{n}+k,l+\tilde{l}-2k,m+\tilde{m}}\,\,\text{with}\,\nonumber\\
&\qquad\qquad\quad
k_0(l,\tilde{l},m,\tilde{m})=
\min\left(\Big[\frac{l+\tilde{l}-|m+\tilde{m}|}{2}\Big],l,\tilde{l}
\right)\\
&\qquad\quad \,\mbox{for}\,\, l\geq1,\,\tilde{l}\geq1,|m|\leq l,\,|\tilde{m}|\leq \tilde{l}\nonumber.
\end{align}
\end{proposition}

The notations in the above Proposition are as following:
\begin{align*}
\lambda^1_{\tilde{n},\tilde{l}}&=\int_{|\theta|\leq\frac{\pi}{4}}\beta(\theta)
((\cos\theta)^{2\tilde{n}+\tilde{l}}P_{\tilde{l}}(\cos\theta)-1)d\theta;\\
\lambda^2_{n,l}&=\int_{|\theta|\leq\frac{\pi}{4}}\beta(\theta)((\sin\theta)^{2n+l}
P_{l}(\sin\theta)-\delta_{0,n}\delta_{0,l})d\theta;
\\
\lambda^{rad, 1}_{n,\tilde{n},\tilde{l}}&
={\color{black} \frac{1}{\sqrt{4\pi}}}\Big(\frac{2\pi^{\frac{3}{2}}(n+\tilde{n})!\Gamma(n+\tilde{n}+\tilde{l}+\frac{3}{2})}
{\tilde{n}!\Gamma(\tilde{n}+\tilde{l}+\frac{3}{2})n!\Gamma(n+\frac{3}{2})}\Big)^{\frac{1}{2}}
\\
&\quad\times\int_{|\theta|\leq\frac{\pi}{4}}
\beta(\theta)(\sin\theta)^{2n}(\cos\theta)^{2\tilde{n}+\tilde{l}}
P_{\tilde{l}}(\cos\theta)d\theta;
\\
\lambda^{rad, 2}_{n, \tilde{n}, l}&={\color{black} \frac{1}{\sqrt{4\pi}}}\Big(\frac{2\pi^{\frac{3}{2}}(n+\tilde{n})!\Gamma(n+\tilde{n}+l+\frac{3}{2})}{\tilde{n}!\Gamma(\tilde{n}+\frac{3}{2})n!\Gamma(n+l+\frac{3}{2})}\Big)^{\frac{1}{2}}
\\
&\quad\times\int_{|\theta|\leq\frac{\pi}{4}}\beta(\theta)(\sin\theta)^{2n+l}
(\cos\theta)^{2\tilde{n}}P_{l}(\sin\theta)d\theta
\end{align*}
and
for $|m^{\star}|\leq l+\tilde{l}-2k$,
\begin{align*}
\mu^{m,\tilde{m},m^{\star}}_{n,\tilde{n},l,\tilde{l},k}
&=
(-1)^k
\Big(\frac{2\pi^{\frac{3}{2}}(n+\tilde{n}+k)!\Gamma(n+\tilde{n}+l+\tilde{l}-k+\frac{3}{2})}
{\tilde{n}!\Gamma(\tilde{n}+\tilde{l}+\frac{3}{2})n!\Gamma(n+l+\frac{3}{2})}\Big)^{\frac{1}{2}}
\\
&\quad\times
\int_{\mathbb{S}^2_{\kappa}}\Big[\int_{\mathbb{S}^2}b(\kappa\cdot\sigma)
\Big(\frac{|\kappa-\sigma|}{2}\Big)^{2n+l}
\Big(\frac{|\kappa+\sigma|}{2}\Big)^{2\tilde{n}+\tilde{l}}\\
&\quad\qquad\times
Y^{m}_l\Big(\frac{\kappa-\sigma}
{|\kappa-\sigma|}\Big) Y^{\tilde{m}}_{\tilde{l}}\Big(\frac{\kappa+\sigma}{|\kappa+\sigma|}\Big)\, d\sigma\Big]\,
\overline{Y^{m^{\star}}_{l+\tilde{l}-2k}}
(\kappa)d\kappa.
\end{align*}
We remark that
$\mu^{m,\tilde{m},m^{\star}}_{n,\tilde{n},l,\tilde{l},k}$ vanishes to 0 if $m+\tilde{m}\neq m^{\star}$, so
\begin{align}\label{remark-mu}
\mu^{m,\tilde{m},m^{\star}}_{n,\tilde{n},l,\tilde{l},k}=\mu^{m,\tilde{m},m^{\star}}_{n,\tilde{n},l,\tilde{l},k}
\delta_{m^{\star},m+\tilde{m}}.
\end{align}
The coefficient $\mu^{m,\tilde{m},m^{\star}}_{n,\tilde{n},l,\tilde{l},k}$ satisfies the following orthogonal property.
\begin{proposition}\label{expansion2}
For any integers $0\leq\,k_1,\,k_2\leq\,\min(l,\tilde{l})$, $|m^{\star}_1|\leq\,l+\tilde{l}-2k_1,\,|m^{\star}_2|\leq\,l+\tilde{l}-2k_2$,
we have
\begin{equation}\label{orth1}
\sum_{|m|\leq\,l}\sum_{|\tilde{m}|\leq\,\tilde{l}}
\mu^{m,\tilde{m},m_1^{\star}}_{n,\tilde{n},l,
\tilde{l},k_1}\overline{\mu^{m,\tilde{m},m_2^{\star}}_{n,\tilde{n},l,
\tilde{l},k_2}}=\sum_{|m|\leq\,l}\sum_{|\tilde{m}|\leq\,\tilde{l}}
\Big|\mu^{m,\tilde{m},m_1^{\star}}_{n,\tilde{n},l,
\tilde{l},k_1}\Big|^2\delta_{k_1,k_2}\delta_{m^{\star}_1,m^{\star}_2}.
\end{equation}
\end{proposition}
{\color{black} We prove Proposition \ref{expansion}, Proposition \ref{expansion2} and the claim \eqref{remark-mu} in Section \ref{S5-1}}.
\begin{remark}\label{remark-decomp}
1) Similar to the radially symmetric case, the property $(iii)$
of {\color{black} Proposition} \ref{expansion} and
the above Proposition \ref{expansion2} imply that we have also a ``triangular effect'' but with a noise
of order $k_0(l,\tilde{l},m,\tilde{m})$.
{\color{black} In the 3-dimensional case, this effect is not easy to understand. For more details, we refer to subsection 2.3.}

2) We have also
\begin{equation}\label{eigenvalue-from-Fourier}
\lambda_{n,l}=-\lambda^1_{n,l}-\lambda^2_{n,l}.
\end{equation}
It is trivial to obtain that $\lambda_{0,0}=\lambda_{1,0}=\lambda_{0,1}=0$~and the others are strictly positive,~
since when $l\neq0$,~and for $n\neq0,1$,~
$$
\lambda_{n,l}\geq2\int^{\frac{\pi}{4}}_{0}(1-\sin^{2n}\theta-\cos^{2n}
\theta)\beta(\theta)d\theta=\lambda_{n,0}>0.
$$
{\color{black} Moreover, from Theorem 2.2 in \cite{NYKC2}\,(see also Theorem 2.3 in \cite{HAOLI})},
there exists a constant  $0<c_1<1$ dependent on $s$ such that, for any $n,l\in\mathbb{N}$ and $n+l\geq2$,
\begin{equation}\label{eq:3.111}
c_1\Big((2n+l+\frac 32)^s+l^{2s}\Big)\le \lambda_{n,l}\le \frac{1}{c_1}\Big((2n+l+\frac 32)^s+l^{2s}\Big)\,   .
\end{equation}
\end{remark}

\subsection{Formal and explicit solution of the Cauchy problem}\label{subsection2,3}

Now we solve explicitly the Cauchy problem associated to the non-cutoff spatially homogeneous Boltzmann equation with Maxwellian molecules.
Consider the solution {\color{black} of} the Cauchy problem \eqref{eq-1} in the form
$$
g(t)=\sum^{+\infty}_{n=0}\sum^{+\infty}_{l=0}\sum_{|m|\leq l}g_{n,l,m}(t)\varphi_{n,l,m},
$$
with initial data
$$
g(0)=\sum^{+\infty}_{n=0}\sum^{+\infty}_{l=0}\sum_{|m|\leq l}g_{n,l,m}^0\varphi_{n,l,m}\in L^2(\mathbb{R}^3),
$$
where
$$
g_{n,l,m}(t)=\left(g(t),\, \varphi_{n,l,m}\right)_{L^2(\mathbb{R}^3)},\qquad
g_{n,l,m}^0=\left(g_0,\, \varphi_{n,l,m}\right)_{L^2(\mathbb{R}^3)}.
$$
In the following, we will use the short notation
$$
\sum^{+\infty}_{n,l,m}=\sum^{+\infty}_{n=0}\sum^{+\infty}_{l=0}\sum_{|m|\leq l}\,\, .
$$
This summation is divided into three terms, which {\color{black} are}
\begin{equation}\label{summation}
\sum^{+\infty}_{n=0}\sum^{+\infty}_{l=0}\sum_{|m|\leq l}f_{n,l,m}=f_{0,0,0}
+\sum^{+\infty}_{n=1}f_{n,0,0}+\sum^{+\infty}_{n=0}\sum^{+\infty}_{l=1}\sum_{|m|\leq l}f_{n,l,m}.
\end{equation}
It follows from {\color{black} ${\bf \Gamma}(\varphi_{0,0,0},\varphi_{0,0,0})={\bf \Gamma}(\sqrt{\mu},\sqrt{\mu})=0$},
Proposition \ref{expansion} and
the above decomposition \eqref{summation}

that, for suitable function $g$,
\begin{align*}
{\bf \Gamma}(g,g)
&=\sum^{+\infty}_{\tilde{n}=0}\sum^{+\infty}_{\tilde{l}=0}\sum_{|\tilde{m}|\leq \tilde{l}}
g_{0,0,0}(t)g_{\tilde{n},\tilde{l},\tilde{m}}(t)
\Big(\lambda^{1}_{\tilde{n},\tilde{l}}+\lambda^2_{\tilde{n},\tilde{l}}\Big)\varphi_{\tilde{n},
\tilde{l},\tilde{m}}\\
&\quad+\sum^{+\infty}_{n=1}\sum^{+\infty}_{\tilde{n}=1}g_{n,0,0}(t)g_{\tilde{n},0,0}(t)\lambda^{rad, 1}_{n, \tilde{n},0}
\varphi_{\tilde{n}+n,0,0}\\
&\quad+\sum^{+\infty}_{n=1}\sum^{+\infty}_{\tilde{n}=0}
\sum^{+\infty}_{\tilde{l}=1}\sum_{|\tilde{m}|\leq\,\tilde{l}}g_{n,0,0}(t)g_{\tilde{n},\tilde{l},
\tilde{m}}(t)\lambda^{rad, 1}_{n, \tilde{n},\tilde{l}}
\varphi_{\tilde{n}+n,\tilde{l},\tilde{m}}\\
&\quad+\sum^{+\infty}_{n=0}
\sum^{+\infty}_{l=1}\sum^{+\infty}_{\tilde{n}=1}\sum_{|m|\leq l}g_{n,l,m}(t) g_{\tilde{n},0,0}(t)\lambda^{rad, 2}_{n, \tilde{n},l} \varphi_{\tilde{n}+n,l,m}\\
&\quad+\sum^{+\infty}_{n=0}\sum^{+\infty}_{l=1}
\sum^{+\infty}_{\tilde{n}=0}\sum^{+\infty}_{\tilde{l}=1}
\sum_{|m|\leq l}\sum_{|\tilde{m}|\leq \tilde{l}}
\sum^{k_0(l,\tilde{l},m,\tilde{m})}_{k=0}\,
g_{n,l,m}(t)g_{\tilde{n},\tilde{l},\tilde{m}}(t)
\mu^{m,\tilde{m},m+\tilde{m}}_{n,\tilde{n},l,
\tilde{l},k}\,\,\varphi_{n+\tilde{n}+k,\,l+\tilde{l}-2k,\,m+\tilde{m}},
\end{align*}
where  $k_0(l,\tilde{l},m,\tilde{m})$ was given in \eqref{k_0}.
Since for fixed {\color{black}$l,\tilde{l}\in \mathbb{N}$},
\begin{align*}
&\left\{(m,\tilde{m},k)\in\mathbb{Z}^2\times\mathbb{N};
|m|\leq l,|\tilde{m}|\leq \tilde{l},0\leq k\leq k_0(l,\tilde{l},m,\tilde{m})\right\}\\
&=
\left\{(m,\tilde{m},k)\in\mathbb{Z}^2\times\mathbb{N};
|m|\leq l,|\tilde{m}|\leq \tilde{l},0\leq k\leq\min(l,\tilde{l}),|m+\tilde{m}|\leq l+\tilde{l}-2k\right\},
\end{align*}
we obtain by {\color{black}changing} the order of the summation
\begin{align}\label{remark-k}
\sum_{|m|\leq l}\sum_{|\tilde{m}|\leq \tilde{l}}
\sum^{k_0(l,\tilde{l},m,\tilde{m})}_{k=0}
&=\sum^{\min(l,\tilde{l})}_{k=0}
\sum_{\substack{|m|\leq l,|\tilde{m}|\leq \tilde{l}\\|m+\tilde{m}|\leq l+\tilde{l}-2k}}
\end{align}
where $\sum_{\substack{|m|\leq l,|\tilde{m}|\leq \tilde{l}\\|m+\tilde{m}|\leq l+\tilde{l}-2k}}$ is the double summation of $m$ and $\tilde{m}$ with constraints $|m+\tilde{m}|\leq l+\tilde{l}-2k$.
Using {\color{black}the equality \eqref{eigenvalue-from-Fourier}}
{\color{black} (that is $\lambda^{1}_{\tilde{n},\tilde{l}}+\lambda^2_{\tilde{n},\tilde{l}}=-\lambda_{\tilde{n},\tilde{l}}$)}
and \eqref{remark-k}, ${\bf \Gamma}(g,g)$ can be rewritten as
\begin{align*}
{\bf \Gamma}(g,g)&=-g_{0,0,0}(t)\sum^{+\infty}_{\tilde{n}=0}\sum^{+\infty}_{\tilde{l}=0}\sum_{|\tilde{m}|\leq \tilde{l}}
g_{\tilde{n},\tilde{l},\tilde{m}}(t)\lambda_{\tilde{n},\tilde{l}}\varphi_{\tilde{n},
\tilde{l},\tilde{m}}\\
&\quad+\sum^{+\infty}_{n=1}\sum^{+\infty}_{\tilde{n}=1}g_{n,0,0}(t)g_{\tilde{n},0,0}(t)\lambda^{rad, 1}_{n, \tilde{n},0}
\varphi_{\tilde{n}+n,0,0}\\
&\quad+\sum^{+\infty}_{n=1}\sum^{+\infty}_{\tilde{n}=0}
\sum^{+\infty}_{\tilde{l}=1}\sum_{|\tilde{m}|\leq\,\tilde{l}}g_{n,0,0}(t)g_{\tilde{n},\tilde{l},
\tilde{m}}(t)\left(\lambda^{rad, 1}_{n, \tilde{n},\tilde{l}}+\lambda^{rad, 2}_{\tilde{n}, n,\tilde{l}}\right)
\varphi_{\tilde{n}+n,\tilde{l},\tilde{m}}\\
&\quad+\sum^{+\infty}_{n=0}\sum^{+\infty}_{l=1}
\sum^{+\infty}_{\tilde{n}=0}\sum^{+\infty}_{\tilde{l}=1}
\sum^{\min(l,\tilde{l})}_{k=0}
\sum_{\substack{|m|\leq l,|\tilde{m}|\leq \tilde{l}\\|m+\tilde{m}|\leq l+\tilde{l}-2k}}\,
g_{n,l,m}(t)g_{\tilde{n},\tilde{l},\tilde{m}}(t)
\mu^{m,\tilde{m},m+\tilde{m}}_{n,\tilde{n},l,
\tilde{l},k}\,\,\varphi_{n+\tilde{n}+k,\,l+\tilde{l}-2k,\,m+\tilde{m}}.
\end{align*}
For {\color{black} suitable} function $g$, we also have
\begin{equation*}
\mathcal{L}g=\sum^{+\infty}_{n,l,m}\lambda_{n,l}\,g_{n,l,m}(t)\,\varphi_{n,l,m}.
\end{equation*}
Formally, {\color{black} taking} {\color{black} the} inner product with $\overline{\varphi_{n^{\star},l^{\star},m^{\star}}}$ on both sides of \eqref{eq-1},
we find that the functions $\{g_{n^{\star},l^{\star},m^{\star}}(t)\}$ satisfy the following infinite system
of the differential equations
\begin{align*}
\partial_t g_{n^{\star},l^{\star},m^{\star}}(t)
&+\lambda_{n^{\star},l^{\star}}(1+g_{0,0,0}(t))
g_{n^{\star},l^{\star},m^{\star}}(t)\nonumber
\\
&=\delta_{l^{\star},0}\sum_{\substack{n+\tilde{n}=n^{\star}\\n\geq1,\tilde{n}\geq1}}
  \lambda^{rad, 1}_{n,\tilde{n},0}
  g_{n,0,0}(t)
  g_{\tilde{n},0,0}(t)\nonumber\\
&\quad+\sum_{\substack{n+\tilde{n}=n^{\star}\\ n\geq1,\tilde{n}\geq0 }}\sum_{\tilde{l}\geq1}
 \delta_{\tilde{l},l^{\star}}
 \left(\lambda^{rad, 1}_{n, \tilde{n},l^{\star}}+\lambda^{rad, 2}_{\tilde{n}, n,l^{\star}}\right)
  g_{\tilde{n},l^{\star},m^{\star}}(t)
  g_{n,0,0}(t)
\\
&\quad+
\sum_{(n,\tilde{n},l,\tilde{l},k,m,\tilde{m})\in\Delta_{n^{\star},l^{\star},m^{\star}}}
\mu^{m,\tilde{m},m^{\star}}_{n,\tilde{n},l,\tilde{l},k}
  g_{n,l,m}(t)
  g_{\tilde{n},\tilde{l},\tilde{m}}(t)\nonumber
\end{align*}
with initial data
\begin{equation}\label{ODE-I}
g_{n^{\star},l^{\star},m^{\star}}(0)=g^0_{n^{\star},l^{\star},m^{\star}}
\end{equation}
and where
\begin{align*}
\Delta_{n^{\star},l^{\star},m^{\star}}
&=\Big\{
(n,\tilde{n},l,\tilde{l},k,m,\tilde{m})\in\mathbb{N}^5\times\mathbb{Z}^2;\\
&\qquad l\geq1,\tilde{l}\geq1,0\leq k\leq \min(l,\tilde{l}),|m|\leq l, |\tilde{m}|\leq \tilde{l}\\
&\qquad\qquad\text{and}\,\,n+\tilde{n}+k=n^{\star},\,l+\tilde{l}-2k=l^{\star},\,m+\tilde{m}=m^{\star}
\Big\},
\end{align*}
which is a subset of a hyperplane of dimension 4.
\begin{remark}\label{remark-sum}
The summation in the last term of the previous
equation for $\partial_t g_{n^{\star},l^{\star},m^{\star}}(t)$  
is a bit complicated.
For the sake of simplicity,
it will be convenient to abuse the notation in this summation and write
\begin{align}\label{ODE}
\begin{split}
\partial_t g_{n^{\star},l^{\star},m^{\star}}(t)
&+\lambda_{n^{\star},l^{\star}}(1+g_{0,0,0}(t))
g_{n^{\star},l^{\star},m^{\star}}(t)
\\
&=\delta_{l^{\star},0}\sum_{\substack{n+\tilde{n}=n^{\star}\\n\geq1,\tilde{n}\geq1}}
  \lambda^{rad, 1}_{n,\tilde{n},0}
  g_{n,0,0}(t)
  g_{\tilde{n},0,0}(t)\\
&\quad+\sum_{\substack{n+\tilde{n}=n^{\star}\\n\geq1,\tilde{n}\geq0}}
 \sum_{\tilde{l}\geq1}
 \delta_{\tilde{l},l^{\star}}
 \left(\lambda^{rad, 1}_{n, \tilde{n},l^{\star}}+\lambda^{rad, 2}_{\tilde{n}, n,l^{\star}}\right)
  g_{\tilde{n},l^{\star},m^{\star}}(t)
  g_{n,0,0}(t)
\\
&\quad+\sum_{\substack{n+\tilde{n}+k=n^{\star}\\n\geq0,\tilde{n}\geq0}}
\sum_{\substack{l+\tilde{l}-2k=l^{\star} \\ l\geq1,\tilde{l}\geq1\\0\leq k\leq \min(l,\tilde{l})}}
\sum_{\substack{m+\tilde{m}=m^{\star} \\ |m|\leq l,|\tilde{m}|\leq\tilde{l}}}
\mu^{m,\tilde{m},m^{\star}}_{n,\tilde{n},l,\tilde{l},k}
  g_{n,l,m}(t)
  g_{\tilde{n},\tilde{l},\tilde{m}}(t).
\end{split}
\end{align}
Here and after, we always use this notation.
\end{remark}
\begin{theorem}\label{prop:ODE}
For any initial data $\{g^0_{n^{\star},l^{\star}, m^{\star}};\, n^{\star},l^{\star}\in\mathbb{N}, |m^{\star}|\le  l^{\star}\}$ with
\begin{equation}\label{eq:2.55}
g_{0,0,0}^0=g_{1,0,0}^0=g_{0,1,0}^0=g_{0,1,1}^0=g_{0,1,-1}^0=0,
\end{equation}
the system \eqref{ODE} admits a global solution $\{g_{n^{\star},l^{\star}, m^{\star}}(t);\, n^{\star},l^{\star}
\in\mathbb{N}, |m^{\star}|\le  l^{\star}\}$.
\end{theorem}

\begin{remark}
Using the triangular effect property of Proposition \ref{expansion}, we solve explicitly the infinite system of differential {\color{black} equations} \eqref{ODE} with any initial data $\{g^0_{n^{\star},l^{\star}, m^{\star}};\, n^{\star},l^{\star}\in\mathbb{N}, |m^{\star}|\le  l^{\star}\}$.
{\color{black}
Note that the initial data doesn't need to belong to $\ell^2$. That means we can construct the formal solution for any initial data $g_0\in\mathscr{S}^{\prime}(\mathbb{R}^3)$.
}
\end{remark}

\begin{proof}
Formally, the system \eqref{ODE} is non linear of quadratic form, but the infinite matrix
of this quadratic form is in fact inferior triangular (see \cite{NYKC3} for
{\color{black} the} radially symmetric case {\color{black} involving a} simple index).
Since the sequence is defined by {\color{black}multi-indices}, we prove this property by the following different case, and in each case by induction.\smallskip
\\
{\bf Induction on $n^{\star}\in \mathbb{N}$:}\smallskip
 \\
{\bf (1) the case : $n^{\star}=0$.} Now we prove the existence of $\{g_{0,l^{\star}, m^{\star}}(t);\,l^{\star}\in\mathbb{N}, |m^{\star}|\le  l^{\star}\}$ by induction on
$l^{\star}\in \mathbb{N}$. From the assumption \eqref{eq:2.55},
and $\lambda_{0,0}=\lambda_{0,1}=\lambda_{1,0}=0$, one {\color{black}gets} that
$$
g_{0,0,0}(t)=g_{1,0,0}(t)=g_{0,1,0}(t)=g_{0,1,1}(t)=g_{0,1,-1}(t)=0.
$$
Let now $l^{\star}\ge 1$, we put the following
{\color{black} induction assumption}:
\\
{\bf (H-1) :}
{\em For any $l\le l^{\star}-1$,\, $|m|\le l$,
the functions $g_{0, l, m}(t)$ solve the equation \eqref{ODE}
with initial data \eqref{ODE-I}.}
\smallskip

We consider the following equation
\begin{align*}
\partial_tg_{0,l^{\star},m^{\star}}(t)+\lambda_{0,l^{\star}}
g_{0,l^{\star},m^{\star}}(t)=\sum_{\substack{l+\tilde{l}=l^{\star}\\l\geq1,\tilde{l}\geq1}}
\sum_{\substack{m+\tilde{m}=m^{\star}\\|m|\leq l,|\tilde{m}|\leq \tilde{l}}}\mu^{m,\tilde{m},m^{\star}}_{0,0,l,\tilde{l},0}g_{0,l,m}(t)
g_{0,\tilde{l},\tilde{m}}(t).
\end{align*}
This differential equation can be solved since the functions $g_{0,l,m}(t)$ on the right hand side are only involving the functions $\{g_{0,l,m}(t)\}_{l\leq l^{\star}-1}$ which have been already known by {\color{black} induction assumption} {\bf (H-1)}.

In particular,\,for any $|m|\leq2$,
\begin{equation}\label{2-0}
g_{0,2,m}(t)=e^{-\lambda_{0,2}t}g_{0,2,m}(0)
\end{equation}
{\color{black} and we compute easily from the identity
$P_2(x)=\frac{3}{2}x^2-\frac{1}{2}$}
$$
{\color{black}\lambda_{0,2}=3\int_{|\theta|\le \frac{\pi}{4}} \beta(\theta)\sin^2\theta\cos^2\theta d\theta>0.}
$$

{\bf (2) the case : $n^{\star}\ge 1$.} Now we put the following induction assumption:
\\
{\bf (H-2) :} {\em For any $n\leq n^{\star}-1$, $l\in\mathbb{N}$ and $|m|\leq l$, the functions $\{ g_{n, l, m}(t)\}$ solve the equation \eqref{ODE} with initial data \eqref{ODE-I}.}
\smallskip

First, we want to solve the function $g_{n^{\star},0,0}$ in \eqref{ODE}.\,\,Since $l^{\star}=m^{\star}=0$, \eqref{ODE} can be written as
\begin{align*}
\partial_tg_{n^{\star},0,0}(t)&+\lambda_{n^{\star},0}
g_{n^{\star},0,0}(t)
\\
&=\sum_{\substack{n+\tilde{n}=n^{\star}
\\
n\geq1,\tilde{n}\geq1}}\lambda^{rad, 1}_{n,\tilde{n},0}g_{n,0,0}(t)
g_{\tilde{n},0,0}(t)
\\
&\quad+\sum_{\substack{n+\tilde{n}+k=n^{\star}\\n\geq0,\tilde{n}\geq0}}
\sum_{\substack{l+\tilde{l}-2k=0 \\ l\geq1,\tilde{l}\geq1\\0\leq k\leq \min(l,\tilde{l})}}
\sum_{\substack{m+\tilde{m}=0\\|m|\leq l,|\tilde{m}|\leq\tilde{l}}}
  \mu^{m,\tilde{m},0}_{n,\tilde{n},l,\tilde{l},k}
  g_{n,l,m}(t)
  g_{\tilde{n},\tilde{l},\tilde{m}}(t).
\end{align*}
In the last term, when $k=0$, we have
$$l+\tilde{l}=0\,\,\text{with\,\,constraints}\,\, l\geq1,\tilde{l}\geq1,$$
which is impossible.  Therefore, the equation \eqref{ODE} is:
\begin{align*}
\partial_tg_{n^{\star},0,0}(t)&+\lambda_{n^{\star},0}
g_{n^{\star},0,0}(t)
\\
&=\sum_{\substack{n+\tilde{n}=n^{\star}
\\
n\geq1,\tilde{n}\geq1}}\lambda^{rad, 1}_{n,\tilde{n},0}g_{n,0,0}(t)
g_{\tilde{n},0,0}(t)
\\
&\quad+\sum_{\substack{n+\tilde{n}+k=n^{\star}\\0\leq n,\tilde{n}\leq n^{\star}-1}}
\sum_{\substack{l+\tilde{l}-2k=0 \\ l\geq1,\tilde{l}\geq1\\1\leq k\leq \min(l,\tilde{l})}}
\sum_{\substack{m+\tilde{m}=0\\|m|\leq l,|\tilde{m}|\leq\tilde{l}}}
  \mu^{m,\tilde{m},0}_{n,\tilde{n},l,\tilde{l},k}
  g_{n,l,m}(t)
  g_{\tilde{n},\tilde{l},\tilde{m}}(t).
\end{align*}
This equation can be also solved since the functions on the right hand side are only involving $\{g_{n,l,m}(t)\}_{n\leq n^{\star}-1,l\in\mathbb{N},|m|\leq l}$, which have been already given in the induction assumption {\bf (H-2)}.

Finally, let $l^{\star}\geq1$, we can improve the
{\color{black} induction assumption } as following:
\\
{\bf (H-3) :}
{\em For any
$n\leq n^{\star}-1,\,l\in\mathbb{N},\,|m|\leq l$
or 
$n=n^{\star}, l\leq l^{\star}-1,\, |m|\leq l, $
the functions $\{ g_{n, l, m}(t)\}$ solve the equation \eqref{ODE} with initial data \eqref{ODE-I}.
}

\smallskip
We want to solve the functions $g_{n^{\star} ,l^{\star}, m^{\star}}(t)$ for all $|m^{\star}|\leq l^{\star}$ in \eqref{ODE},\,which is
\begin{align*}
&\partial_t g_{n^{\star}, l^{\star}, m^{\star}}(t)+\lambda_{n^{\star}, l^{\star}}g_{n^{\star}, l^{\star}, m^{\star}}(t)\\
&=
\sum_{\substack{n+\tilde{n}=n^{\star}\\ n\geq1,\tilde{n}\geq0}}
 \left(\lambda^{rad, 1}_{n, \tilde{n},l^{\star}}+\lambda^{rad, 2}_{\tilde{n}, n,l^{\star}}\right)
  g_{\tilde{n},l^{\star},m^{\star}}(t)
  g_{n,0,0}(t)\\
&\qquad+\sum_{\substack{n+\tilde{n}=n^{\star}\\n\geq0,\tilde{n}\geq0}}  \,\,
\sum_{\substack{l+\tilde{l}=l^{\star} \\ l\geq1,\tilde{l}\geq 1}}  \,\,
\sum_{\substack{m+\tilde{m}=m^{\star}\\|m|\leq l,|\tilde{m}|\leq\tilde{l}}}\mu^{m,\tilde{m},m^{\star}}_{n,\tilde{n},l,\tilde{l},0}
g_{n,l,m}(t)g_{\tilde{n},\tilde{l},\tilde{m}}(t)\\
&\qquad+\sum_{\substack{n+\tilde{n}+k=n^{\star}\\ n\geq0,\tilde{n}\geq0}}
\sum_{\substack{l+\tilde{l}-2k=l^{\star} \\ l\geq1,\tilde{l}\geq1\\1\leq k\leq \min(l,\tilde{l})}}
\sum_{\substack{m+\tilde{m}=m^{\star}\\|m|\leq l,|\tilde{m}|\leq\tilde{l}}}\mu^{m,\tilde{m},m^{\star}}_{n,
\tilde{n},l,\tilde{l},k}g_{n,l,m}(t)g_{\tilde{n},\tilde{l},\tilde{m}}(t).
\end{align*}
Here the summation in the last two terms is understanding as Remark \ref{remark-sum}.
 This equation can be also solved since the functions on the right hand side are only involving $\{g_{n,l,m}(t)\}_{n\leq n^{\star}-1, l\in\mathbb{N}}$ and
$\{g_{n,l,m}(t)\}_{n=n^{\star}, l\leq l^{\star}-1}$ which is given by the improved {\color{black} induction assumption} {\bf (H-3)}.
\end{proof}

Now the proof of Theorem \ref{trick} is reduced to prove the convergence of following series
\begin{equation}\label{ODE-b}
g(t)=\sum^{+\infty}_{n, l, m}g_{n, l, m}(t)\varphi_{n, l, m}
\end{equation}
in the {\color{black} suitable} function space.


\section{The estimate of the non linear operators}\label{S3}

\subsection{{\color{black}The estimate of the trilinear term}}
To prove the convergence of the formal solution obtained in the precedent section, we need to estimate the following trilinear terms
$$
\left({\bf \Gamma}(f,g),h\right)_{L^2(\mathbb{R}^3)},
\,\,\,f,g,h\in\mathscr{S}(\mathbb{R}^3)\cap\mathcal{N}^{\perp}.
$$

Using the spectral representation of ${\bf \Gamma}(\,\cdot,\,\cdot\,)$ given in Proposition \ref{expansion},
we need to estimate {\color{black}their} coefficients.

\begin{proposition}\label{est}.

\noindent
1) For $n\geq1$, $\tilde{n},\tilde{l}\in\mathbb{N}$, we have,
\begin{equation*}
|\lambda^{rad, 1}_{n,\tilde{n},\tilde{l}}|^2\lesssim
\tilde{n}^s(\tilde{n}+\tilde{l})^sn^{-\frac{5}{2}-2s}.
\end{equation*}
2) For all $\tilde{n}\geq1$, $n,l\in\mathbb{N}$, $n+l\geq2$, we have
\begin{equation*}
|\lambda^{rad, 2}_{n,\tilde{n},l}|^2\lesssim \frac{\tilde{n}^{2s}}{(n+1)^s(n+l)^{\frac{5}{2}+s}}.
\end{equation*}
3) For any $n^{\star},l^{\star}\in\mathbb{N}$, $|m^{\star}|\leq l^{\star}$, we have also
$$
\sum_{\substack{n+\tilde{n}+k=n^{\star}\\n+l\geq2,\tilde{n}+\tilde{l}\geq2\\n\geq0,\tilde{n}\geq0}}
\sum_{\substack{l+\tilde{l}-2k=l^{\star}\\l\geq1,\tilde{l}\geq1\\0\leq\,k\leq\,\min(l,\tilde{l})}}\,\,
\sum_{|m|\leq\,l} \,\,
\sum_{|\tilde{m}|\leq\,\tilde{l}}
\frac{|\mu^{m,\tilde{m},m^{\star}}_{n,\tilde{n},l,\tilde{l},k}|^2}{\lambda_{\tilde{n},\tilde{l}}}
\,\lesssim\,
 \lambda_{n^{\star},l^{\star}}.
$$
\end{proposition}
The constraint of the above summation is
{\color{black}
\begin{align}\label{remark-summation}
\Delta_{n^{\star},l^{\star}}
&=\Big\{
(n,\tilde{n},l,\tilde{l},k,m,\tilde{m})\in\mathbb{N}^5\times\mathbb{Z}^2;\,n+l\geq2,\,\tilde{n}+\tilde{l}\geq2,\nonumber\\
&\qquad l\geq1,\tilde{l}\geq1,|m|\leq l, |\tilde{m}|\leq \tilde{l},0\leq k\leq \min(l,\tilde{l})\\
&\qquad\qquad\text{and}\,\,n+\tilde{n}+k=n^{\star},\,l+\tilde{l}-2k=l^{\star}
\Big\}\nonumber.
\end{align}
}
{\color{black} Following Remark \ref{remark-sum}},
we will always write the complicated summation
$$\sum_{(n,\tilde{n},l,\tilde{l},k,m,\tilde{m})\in\Delta_{n^{\star},l^{\star}}}$$
in the simplified form :
$$\sum_{\substack{n+\tilde{n}+k=n^{\star}\\n+l\geq2,\tilde{n}+\tilde{l}\geq2\\n\geq0,\tilde{n}\geq0}}
\sum_{\substack{l+\tilde{l}-2k=l^{\star}\\l\geq1,\tilde{l}\geq1\\0\leq\,k\leq\,\min(l,\tilde{l})}}\, \sum_{|m|\leq\,l}\sum_{|\tilde{m}|\leq\,\tilde{l}}.$$
{\color{black}Since the proof of Proposition \ref{est} is technical, we prove it in Section \ref{S5-2}}.
Now we prove the following trilinear estimates for the non linear Boltzmann operator.

\begin{proposition}\label{prop:3.2}
For all $f,g,h\in\mathscr{S}(\mathbb{R}^3)\bigcap\mathcal{N}^{\perp},$
$$
\left|\left({\bf \Gamma}(f,g),\,h\right)_{L^2(\mathbb{R}^3)}
\right|\lesssim\|f\|_{L^2(\mathbb{R}^3)}\|\mathcal{L}^{\frac{1}{2}}g\|_{L^2(\mathbb{R}^3)}
\|\mathcal{L}^{\frac{1}{2}}h\|_{L^2(\mathbb{R}^3)}\, .
$$
\end{proposition}

\begin{proof}
For any $f, g, h\in \mathscr{S}(\mathbb{R}^3)\bigcap\mathcal{N}^{\perp}$, we use the following spectral
decomposition,
$$
f=\sum_{\substack{n+l\geq2\\n\geq0,l\geq0}}\sum_{|m|\leq l}f_{n,l,m}\varphi_{n,l,m},\,
\,g=\sum_{\substack{n+l\geq2\\n\geq0,l\geq0}}\sum_{|m|\leq l}g_{n,l,m}\varphi_{n,l,m},\,
\,h=\sum_{\substack{n+l\geq2\\n\geq0,l\geq0}}\sum_{|m|\leq l}h_{n,l,m}\varphi_{n,l,m}.
$$
Using the orthogonality of basis $\{  \varphi_{n,l,m}  \}$ and the formula \eqref{remark-k}, we deduce from Proposition \ref{expansion} that,
\begin{align*}
&\left({\bf \Gamma}(f,g),\, h\right)_{L^2(\mathbb{R}^3)}\\
&=
\sum_{\substack{ n^{\star}\geq0,l^{\star}\geq0\\n^{\star}+l^{\star}\geq2}}
\sum_{|m^{\star}|\leq l^{\star}}\overline{h_{n^{\star},l^{\star},m^{\star}}}
\left(\delta_{l^{\star},0}\sum_{\substack{n+\tilde{n}=n^{\star}\\n\geq2,\tilde{n}\geq2}}
  \lambda^{rad, 1}_{n,\tilde{n},0}
f_{n,0,0}\,g_{\tilde{n},
0,0}\right)\\
&\quad
+
\sum_{\substack{ n^{\star}\geq0,l^{\star}\geq0\\n^{\star}+l^{\star}\geq2}}
\sum_{|m^{\star}|\leq l^{\star}}\overline{h_{n^{\star},l^{\star},m^{\star}}}
\left(
\sum_{\substack{n+\tilde{n}=n^{\star},\tilde{l}\geq1\\ n\geq2,\tilde{n}\geq0,\tilde{n}+\tilde{l}\geq2}}
 \delta_{\tilde{l},l^{\star}}
 \lambda^{rad, 1}_{n, \tilde{n},l^{\star}}f_{n,0,0}g_{\tilde{n},l^{\star},m^{\star}}\right)\\
&\quad+
 \sum_{\substack{ n^{\star}\geq0,l^{\star}\geq0\\n^{\star}+l^{\star}\geq2}}
 \sum_{|m^{\star}|\leq l^{\star}}\overline{h_{n^{\star},l^{\star},m^{\star}}}
\left(\sum_{\substack{n+\tilde{n}=n^{\star},l\geq1\\ n\geq0,\tilde{n}\geq2,n+l\geq2}}
 \delta_{l,l^{\star}}
 \lambda^{rad, 2}_{n,\tilde{n} ,l^{\star}}
  f_{n,l^{\star},m^{\star}}
  g_{\tilde{n},0,0}
  \right)
\\
&\quad+
\sum_{\substack{ n^{\star}\geq0,l^{\star}\geq0\\n^{\star}+l^{\star}\geq2}}
\sum_{|m^{\star}|\leq l^{\star}}\overline{h_{n^{\star},l^{\star},m^{\star}}}\left(
\sum_{\substack{n+l\geq2\\n\geq0,l\geq1}}
\sum_{\substack{\tilde{n}+\tilde{l}\geq2\\\tilde{n}\geq0,\tilde{l}\geq1}}
\sum^{\min(l,\tilde{l})}_{k=0}
\sum_{\substack{|m|\leq l,|\tilde{m}|\leq \tilde{l}\\|m+\tilde{m}|\leq l+\tilde{l}-2k}}
\mu^{m,\tilde{m},m+\tilde{m}}_{n,\tilde{n},l,
\tilde{l},k}\,
f_{n,l,m}g_{\tilde{n},\tilde{l},\tilde{m}}\right)\\
&\qquad\times\delta_{n^{\star},n+\tilde{n}+k}\delta_{l^{\star},l+\tilde{l}-2k}\delta_{m^{\star},m+\tilde{m}}.
\end{align*}
For brevity, using the formula \eqref{remark-mu}, we have
\begin{align*}
&\left({\bf \Gamma}(f,g),\, h\right)_{L^2(\mathbb{R}^3)}
\\
&=\sum^{+\infty}_{n^{\star}=4}\overline{h_{n^{\star},0,0}}
\left(\sum_{\substack{n+\tilde{n}=n^{\star}\\n\geq2,\tilde{n}\geq2}}
  \lambda^{rad, 1}_{n,\tilde{n},0}
f_{n,0,0}\,g_{\tilde{n},
0,0}\right)
\\
&\quad+
\sum_{\substack{ n^{\star}\geq0,{\color{black}l^{\star}\geq1}\\n^{\star}+l^{\star}\geq2}}
\sum_{|m^{\star}|\leq l^{\star}}\overline{h_{n^{\star},l^{\star},m^{\star}}}
\left(
\sum_{\substack{n+\tilde{n}=n^{\star}\\ n\geq2,\tilde{n}\geq0,{\color{black}\tilde{n}\geq2-l^{\star}}}}
 \lambda^{rad, 1}_{n, \tilde{n},l^{\star}}f_{n,0,0}g_{\tilde{n},l^{\star},m^{\star}}\right)
\\
&\quad+
\sum_{\substack{ n^{\star}\geq0,{\color{black}l^{\star}\geq1}\\n^{\star}+l^{\star}\geq2}}
\sum_{|m^{\star}|\leq l^{\star}}\overline{h_{n^{\star},l^{\star},m^{\star}}}
\left(\sum_{\substack{n+\tilde{n}=n^{\star}\\ n\geq0,\tilde{n}\geq2,{\color{black}n\geq2-l^{\star}}}}
 \lambda^{rad, 2}_{n,\tilde{n} ,l^{\star}}
  f_{n,l^{\star},m^{\star}}
  g_{\tilde{n},0,0}
  \right)
\\
&\quad+
\sum_{\substack{n\geq0,l\geq1\\n+l\geq2 }}
\sum_{\substack{\tilde{n}\geq0,\tilde{l}\geq1\\\tilde{n}+\tilde{l}\geq2}}
\sum_{|m|\leq l}\sum_{|\tilde{m}|\leq \tilde{l}}
\sum^{\min(l,\tilde{l})}_{k=0}
\sum_{|m^{\star}|\leq l+\tilde{l}-2k}
\mu^{m,\tilde{m},m^{\star}}_{n,\tilde{n},l,
\tilde{l},k}\,
f_{n,l,m}g_{\tilde{n},\tilde{l},\tilde{m}}
\overline{h_{n+\tilde{n}+k,l+\tilde{l}-2k,m^{\star}}}
\\
&=I_1+I_2 +I_3+I_4.
\end{align*}
For the term $I_1$, since
$\lambda_{n,0}\approx n^s$ in \eqref{eq:3.111},
we deduce from Cauchy-Schwarz inequality and $1)$ of Proposition \ref{est} that,
\begin{align*}
|I_1|&\leq
\sum_{n^{\star}\geq4}|h_{n^{\star},0,0}|
\Big(\sum_{\substack{n+\tilde{n}=n^{\star}\\n\geq2,\tilde{n}\geq2}}|\lambda^{rad, 1}_{n,\tilde{n},0}|\, |f_{n,0,0}|\, |g_{\tilde{n},0,0}|\Big)\\
&\leq\|f\|_{L^2}\|\mathcal{L}^{\frac{1}{2}}g\|_{L^2}
\Big(\sum^{+\infty}_{\tilde{n}=2}\sum^{+\infty}_{n=2}|h_{n+\tilde{n},
0,0}|^2\frac{|\lambda^{rad, 1}_{n,\tilde{n},0}|^2}{\lambda_{\tilde{n},0}}\Big)^{\frac{1}{2}}\\
&\lesssim\|f\|_{L^2}\|\mathcal{L}^{\frac{1}{2}}g\|_{L^2}
\Big[\sum^{\infty}_{n^{\star}=2}|h_{n^{\star},0,0}|^2
\Big(\sum_{\substack{n+\tilde{n}=n^{\star}\\
n\geq2,\tilde{n}\geq2}}\frac{\tilde{n}^s}{n^{\frac{5}{2}+2s}}
\Big)\Big]^{\frac{1}{2}}\lesssim\|f\|_{L^2}\|\mathcal{L}^{\frac{1}{2}}g\|_{L^2}\|\mathcal{L}^{\frac{1}{2}}h\|_{L^2}.
\end{align*}
For the term $I_2$, we use Cauchy-Schwarz inequality,
\begin{align*}
|I_2|&\leq
\sum_{\substack{n^{\star}\geq0,l^{\star}\geq1\\n^{\star}+l^{\star}\geq2}}
\sum^{l^{\star}}_{m^{\star}=-l^{\star}}
|h_{n^{\star},l^{\star},m^{\star}}|
\Big(\sum_{\substack{n+\tilde{n}=n^{\star}\\n\geq2,{\color{black}\tilde{n}\geq\max(0,2-l^{\star})}}}|\lambda^{rad, 1}_{n,\tilde{n},l^{\star}}|\, |f_{n,0,0}|\, |g_{\tilde{n},l^{\star},m^{\star}}|\Big)\\
&\leq{\color{black}
\sum_{\substack{\tilde{n}\geq0,l^{\star}\geq1\\\tilde{n}+l^{\star}\geq2}}\sum_{|m^{\star}|\leq l^{\star}}\sum_{n\geq2}
|\lambda^{rad, 1}_{n,\tilde{n},l^{\star}}|\, |f_{n,0,0}|\, |g_{\tilde{n},l^{\star},m^{\star}}|\, |h_{n+\tilde{n},l^{\star},m^{\star}}|}
\\
&\leq{\color{black}\|\mathcal{L}^{\frac{1}{2}}g\|_{L^2}
\Big(
\sum_{\substack{\tilde{n}\geq0,l^{\star}\geq1\\\tilde{n}+l^{\star}\geq2}}\sum_{|m^{\star}|\leq l^{\star}}
\frac{1}{\lambda_{\tilde{n},l^{\star}}}
\Big|\sum_{n\geq2}
|\lambda^{rad, 1}_{n,\tilde{n},l^{\star}}|\, |f_{n,0,0}|\, |h_{n+\tilde{n},l^{\star},m^{\star}}|\Big|^2
\Big)^{\frac{1}{2}}}\\
&\leq\|f\|_{L^2}\|\mathcal{L}^{\frac{1}{2}}g\|_{L^2}
\Big(
{\color{black}\sum_{\substack{\tilde{n}\geq0,l^{\star}\geq1\\\tilde{n}+l^{\star}\geq2}}\sum_{|m^{\star}|\leq l^{\star}}}
\sum_{n\geq2}|h_{n+\tilde{n},
l^{\star},m^{\star}}|^2\frac{|\lambda^{rad, 1}_{n,\tilde{n},l^{\star}}|^2}{\lambda_{\tilde{n},l^{\star}}}\Big)^{\frac{1}{2}}\\
&\leq\|f\|_{L^2}\|\mathcal{L}^{\frac{1}{2}}g\|_{L^2}
\Big[{\color{black}\sum_{n^{\star}\geq2,l^{\star}\geq1}\sum_{|m^{\star}|\leq l^{\star}}}|h_{n^{\star},l^{\star},m^{\star}}|^2
\Big(\sum_{\substack{n+\tilde{n}=n^{\star}\\
n\geq2,{\color{black}\tilde{n}\geq\max(0,2-l^{\star})}}}\frac{1}{\lambda_{\tilde{n},l^{\star}}}
|\lambda^{rad, 1}_{n,\tilde{n},l^{\star}}|^2\Big)\Big]^{\frac{1}{2}}.
\end{align*}
{\color{black}Since $\tilde{n}+l^{\star}\geq2$, we can deduce from the formula \eqref{eq:3.111} that,
$$\lambda_{\tilde{n},l^{\star}}\gtrsim (\tilde{n}+l^{\star})^s+(l^{\star})^{2s}.$$
It follows from $1)$ of Proposition \ref{est} and the formula \eqref{eq:3.111} that,}
\begin{equation}\label{estimateI-2}
\sum_{\substack{n+\tilde{n}=n^{\star}\\n\geq2,{\color{black}\tilde{n}\geq\max(0,2-l^{\star})}}}
\frac{|\lambda^{rad, 1}_{n,\tilde{n},l^{\star}}|^2}{\lambda_{\tilde{n},l^{\star}}}
\lesssim\sum_{\substack{n+\tilde{n}=n^{\star}\\n\geq2,\tilde{n}\geq0}}\tilde{n}^sn^{-\frac{5}{2}
-2s}\lesssim(n^{\star})^s\lesssim \lambda_{n^{\star},l^{\star}}.
\end{equation}
Therefore,
$$
|I_2|\leq\|f\|_{L^2}\|\mathcal{L}^{\frac{1}{2}}g\|_{L^2}\|\mathcal{L}^{\frac{1}{2}}h\|_{L^2}.
$$
Similarly, for the term $I_3$, by using Cauchy-Schwarz inequality, we obtain,
{\color{black}
\begin{align*}
|I_3|
&\leq\sum_{\substack{ n^{\star}\geq0,l^{\star}\geq1\\n^{\star}+l^{\star}\geq2}}
\sum_{|m^{\star}|\leq l^{\star}}|h_{n^{\star},l^{\star},m^{\star}}|
\Big(\sum_{\substack{n+\tilde{n}=n^{\star}\\ \tilde{n}\geq2,n\geq\max(0,2-l^{\star})}}
 |\lambda^{rad, 2}_{n,\tilde{n} ,l^{\star}}|
  |f_{n,l^{\star},m^{\star}}|
  |g_{\tilde{n},0,0}|
  \Big)\\
&\leq\sum_{\tilde{n}\geq2}\sum_{\substack{n\geq0,l^{\star}\geq1\\n+l^{\star}\geq2}}\sum_{|m^{\star}|\leq l^{\star}}
|\lambda^{rad, 2}_{n,\tilde{n} ,l^{\star}}|
  |f_{n,l^{\star},m^{\star}}|
  |g_{\tilde{n},0,0}|
  |h_{n+\tilde{n},l^{\star},m^{\star}}|\\
&\leq\|\mathcal{L}^{\frac{1}{2}}g\|_{L^2}
\Big(\sum_{\tilde{n}\geq2}\frac{1}{\lambda_{\tilde{n},0}}\Big|\sum_{\substack{n\geq0,l^{\star}\geq1\\n+l^{\star}\geq2}}\sum_{|m^{\star}|\leq l^{\star}}
|\lambda^{rad, 2}_{n,\tilde{n} ,l^{\star}}|
  |f_{n,l^{\star},m^{\star}}|
  |h_{n+\tilde{n},l^{\star},m^{\star}}|\Big|^2\Big)^{\frac{1}{2}}\\
&\leq\|\mathcal{L}^{\frac{1}{2}}g\|_{L^2}\|f\|_{L^2}
\Big(\sum_{\tilde{n}\geq2}\sum_{\substack{n\geq0,l^{\star}\geq1\\n+l^{\star}\geq2}}\sum_{|m^{\star}|\leq l^{\star}}
\frac{|\lambda^{rad, 2}_{n,\tilde{n} ,l^{\star}}|^2}{\lambda_{\tilde{n},0}}|h_{n+\tilde{n},l^{\star},m^{\star}}|^2\Big)^{\frac{1}{2}}\\
&\leq\|\mathcal{L}^{\frac{1}{2}}g\|_{L^2}\|f\|_{L^2}
\Bigg(\sum_{n^{\star}\geq2,l^{\star}\geq1}\sum_{|m^{\star}|\leq l^{\star}}|h_{n^{\star},l^{\star},m^{\star}}|^2
\Bigg[\sum_{\substack{n+\tilde{n}=n^{\star}\\\tilde{n}\geq2,n\geq\max(0, 2-l^{\star})}}\frac{|\lambda^{rad, 2}_{n,\tilde{n} ,l^{\star}}|^2}{\lambda_{\tilde{n},0}}\Bigg]\Bigg)^{\frac{1}{2}}.
\end{align*}
}
We use {\color{black}$2)$ of} Proposition \ref{est} and $\lambda_{\tilde{n},0}\approx\tilde{n}^{s}$ in \eqref{eq:3.111}
\begin{align}\label{estimateI-3}
\sum_{\substack{n+\tilde{n}=n^{\star}\\\tilde{n}\geq2,n\geq\max(0,2-l^{\star})}}
\frac{|\lambda^{rad, 2}_{n,\tilde{n},l^{\star}}|^2}{\lambda_{\tilde{n},0}}
&\lesssim\sum_{\substack{n+\tilde{n}=n^{\star}\\\tilde{n}\geq2,n\geq\max(0,2-l^{\star})}}\frac{\tilde{n}^{s}}{(n+1)^s(n+l^{\star})^{\frac{5}{2}+s}}\nonumber\\
&\lesssim\sum_{\substack{n+\tilde{n}=n^{\star}\\n=0}}(n^{\star})^s
+\sum_{\substack{n+\tilde{n}=n^{\star}\\\tilde{n}\geq2,n\geq1}}\frac{(n^{\star})^s}{n^2}
\lesssim(n^{\star})^s\leq \lambda_{n^{\star},l^{\star}},
\end{align}
which gives
$$
|I_3|\lesssim \|f\|_{L^2}\|\mathcal{L}^{\frac{1}{2}}g\|_{L^2}\|\mathcal{L}^{\frac{1}{2}}h\|_{L^2}.
$$
For the term $I_4$, we note that $l\geq1,\,\tilde{l}\geq1$,
\begin{align*}
|I_4|
&\leq\sum_{\substack{\tilde{n}\geq0,\tilde{l}\geq1\\\tilde{n}+\tilde{l}\geq2}}
\sum_{|\tilde{m}|\leq\tilde{l}}
|g_{\tilde{n},\tilde{l},\tilde{m}}|\sum_{\substack{ n\geq0,l\geq1\\n+l\geq2}}
\sum_{|m|\leq l}|f_{n,l,m}|
\\
&\qquad\times\,
\left|
\sum^{\min(l,\tilde{l})}_{k=0}
\sum_{|m^{\star}|\leq l+\tilde{l}-2k}
 \mu^{m,\tilde{m},m^{\star}}_{n,\tilde{n},l,
\tilde{l},k}\,
\overline{h_{n+\tilde{n}+k,l+\tilde{l}-2k,m^{\star}}}
 \right|.
\end{align*}
Applying the Cauchy-Schwarz inequality, we get
\begin{align*}
|I_4|&\leq\Big(
\sum_{\substack{\tilde{n}+\tilde{l}\geq2\\\tilde{n}\geq0,\tilde{l}\geq1}}
\sum_{|\tilde{m}|\leq\,\tilde{l}}
\lambda_{\tilde{n},\tilde{l}}|g_{\tilde{n},\tilde{l},\tilde{m}}|^2\Big)^{\frac{1}{2}}
\Bigg[\sum_{\substack{\tilde{n}+\tilde{l}\geq2\\\tilde{n}\geq0,\tilde{l}\geq1}}\sum_{|\tilde{m}|\leq\,\tilde{l}}\frac{1}{\lambda_{\tilde{n},\tilde{l}}}
\\
&\qquad\Big(\sum_{\substack{n+l\geq2\\n\geq0,l\geq1}}
\sum_{|m|\leq\,l}|f_{n,l,m}|
\Bigg|
\sum^{\min(l,\tilde{l})}_{k=0}
\sum_{|m^{\star}|\leq l+\tilde{l}-2k}\,
\mu^{m,\tilde{m},m^{\star}}_{n,\tilde{n},l,
\tilde{l},k}\,
\overline{h_{n+\tilde{n}+k,l+\tilde{l}-2k,m^{\star}}}
\Bigg|
\Big)^2\Bigg]^{\frac{1}{2}}
\\
&\leq\|\mathcal{L}^{\frac{1}{2}}g\|_{L^2}\|f\|_{L^2}
\\
&\times
\left[\sum_{\substack{\tilde{n}+\tilde{l}\geq2\\\tilde{n}\geq0,\tilde{l}\geq1}}
\sum_{\substack{n+l\geq2\\n\geq0,l\geq1}}
\frac{1}{\lambda_{\tilde{n},\tilde{l}}}
\sum_{|m|\leq\,l}
\sum_{|\tilde{m}|\leq\,\tilde{l}}
\Bigg|
\sum^{\min(l,\tilde{l})}_{k=0}
\sum_{|m^{\star}|\leq l+\tilde{l}-2k}
\mu^{m,\tilde{m},m^{\star}}_{n,\tilde{n},l,
\tilde{l},k}\,
\overline{h_{n+\tilde{n}+k,l+\tilde{l}-2k,m^{\star}}}
 \Bigg|^2\right]^{\frac{1}{2}}.
\end{align*}
We observe the summation
\begin{align*}
&\sum_{|m|\leq\,l}
\sum_{|\tilde{m}|\leq\,\tilde{l}}
\Bigg|
\sum^{\min(l,\tilde{l})}_{k=0}
\sum_{|m^{\star}|\leq l+\tilde{l}-2k}
\,
\mu^{m,\tilde{m},m^{\star}}_{n,\tilde{n},l,
\tilde{l},k}\,
\overline{h_{n+\tilde{n}+k,l+\tilde{l}-2k,m^{\star}}}
 \Bigg|^2\\
&=\sum^{\min(l,\tilde{l})}_{k_1=0}\sum_{|m^{\star}_1|\leq l+\tilde{l}-2k_1}\sum^{\min(l,\tilde{l})}_{k_2=0}\sum_{|m^{\star}_2|\leq l+\tilde{l}-2k_2}
\overline{h_{n+\tilde{n}+k_1,l+\tilde{l}-2k_1,m_1^{\star}}}h_{n+\tilde{n}+k_2,l+\tilde{l}-2k_2,m_2^{\star}}\,\\
&\qquad\times
\left(\sum_{|m|\leq\,l}
\sum_{|\tilde{m}|\leq\,\tilde{l}}\, \,\mu^{m,\tilde{m},m_1^{\star}}_{n,\tilde{n},l,\tilde{l},k_1}
 \overline{\mu^{m,\tilde{m},m_2^{\star}}_{n,\tilde{n},l,\tilde{l},k_2}}\right).
\end{align*}
By using the formula \eqref{orth1}
in {Proposition \ref{expansion2}}, we obtain
$$\sum_{|m|\leq\,l}
\sum_{|\tilde{m}|\leq\,\tilde{l}}\mu^{m,\tilde{m},m_1^{\star}}_{n,\tilde{n},l,\tilde{l},k_1}
 \overline{\mu^{m,\tilde{m},m_2^{\star}}_{n,\tilde{n},l,\tilde{l},k_2}}=\sum_{|m|\leq\,l}
\sum_{|\tilde{m}|\leq\,\tilde{l}}\Big|\mu^{m,\tilde{m},m_1^{\star}}_{n,\tilde{n},l,\tilde{l},k_1}
 \Big|^2\delta_{k_1,k_2}\delta_{m^{\star}_1,m^{\star}_2}.$$
{\color{black}Therefore, the summation can be re-written as}
\begin{align*}
&\sum_{|m|\leq\,l}
\sum_{|\tilde{m}|\leq\,\tilde{l}}
\Bigg|
\sum^{\min(l,\tilde{l})}_{k=0}
\sum_{|m^{\star}|\leq l+\tilde{l}-2k}
\mu^{m,\tilde{m},m^{\star}}_{n,\tilde{n},l,
\tilde{l},k}\,
\overline{h_{n+\tilde{n}+k,l+\tilde{l}-2k,m^{\star}}}
 \Bigg|^2\\
&=
\sum^{\min(l,\tilde{l})}_{k=0}\sum_{|m^{\star}|\leq l+l-2k}
\Big|h_{n+\tilde{n}+k,l+\tilde{l}-2k,m^{\star}}\Big|^2
\sum_{|m|\leq\,l}\sum_{|\tilde{m}|\leq\,\tilde{l}}\, \,\Big|\mu^{m,\tilde{m},m^{\star}}_{n,\tilde{n},l,\tilde{l},k}\Big|^2.
\end{align*}
It follows that
\begin{align*}
|I_4|&\leq\|f\|_{L^2}\|\mathcal{L}^{\frac{1}{2}}g\|_{L^2}
\Big(\sum_{\substack{\tilde{n}+\tilde{l}\geq2\\\tilde{n}\geq0,\tilde{l}\geq1}}
\sum_{\substack{n+l\geq2\\n\geq0,l\geq1}}
\sum_{0\leq\,k\leq\min(l,\tilde{l})}\sum_{|m^{\star}|\leq l+\tilde{l}-2k} \Big|h_{n+\tilde{n}+k,l+\tilde{l}-2k,m^{\star}}\Big|^2\\
&\times\, \frac{1}{\lambda_{\tilde{n},\tilde{l}}}
\sum_{|m|\leq\,l}\sum_{|\tilde{m}|\leq\,\tilde{l}}\, \,
\Big|\,\mu^{m,\tilde{m},m^{\star}}_{n,\tilde{n},l,\tilde{l},k} \Big|^2\Big)^{\frac{1}{2}}\\
&\leq\|f\|_{L^2}\|\mathcal{L}^{\frac{1}{2}}g\|_{L^2}
\Bigg[\sum^{+\infty}_{n^{\star}=0}\sum^{+\infty}_{l^{\star}=0}\sum_{|m^{\star}|\leq\,l^{\star}}|h_{n^{\star},l^{\star},m^{\star}}|^2\\
&\qquad\times\Big(\sum_{\substack{n+\tilde{n}+k=n^{\star}\\n+l\geq2,
\tilde{n}+\tilde{l}\geq2\\
n\geq0,
\tilde{n}\geq0}}
\sum_{\substack{l+\tilde{l}-2k=l^{\star}\\l\geq1,\tilde{l}\geq1\\0\leq\,k\leq\,\min(l,\tilde{l})}}
\sum_{|m|\leq\,l}\sum_{|\tilde{m}|\leq\,\tilde{l}}\, \,
\frac{1}{\lambda_{\tilde{n},\tilde{l}}}\Big|\,\mu^{m,\tilde{m},m^{\star}}_{n,\tilde{n},l,\tilde{l},k} \Big|^2\Big)\Bigg]^{\frac{1}{2}},
\end{align*}
where the last summation is understanding as \eqref{remark-summation}.
Using {\color{black}$3)$ of} Proposition \ref{est}, we have
$$\sum_{\substack{n+\tilde{n}+k=n^{\star}\\n+l\geq2,\tilde{n}+\tilde{l}\geq2\\
n\geq0,
\tilde{n}\geq0}}
\sum_{\substack{l+\tilde{l}-2k=l^{\star}\\l\geq1,\tilde{l}\geq1\\0\leq\,k\leq\,\min(l,\tilde{l})}}
\sum_{|m|\leq\,l}\sum_{|\tilde{m}|\leq\,\tilde{l}}\, \,\frac{1}{\lambda_{\tilde{n},\tilde{l}}} \,\Big|\,\mu^{m,\tilde{m},m^{\star}}_{n,\tilde{n},l,\tilde{l},k} \Big|^2\leq C\lambda_{n^{\star},l^{\star}}.$$
We get then
$$
|I_4|\lesssim\\|f\|_{L^2}\|\mathcal{L}^{\frac{1}{2}}g\|_{L^2}
\Big[\sum^{+\infty}_{n^{\star}=0}\sum^{+\infty}_{l^{\star}=0}\sum_{|m^{\star}|\leq\,l^{\star}} \lambda_{n^{\star},l^{\star}}|h_{n^{\star},l^{\star},m^{\star}}|^2
\Big]^{\frac{1}{2}},
$$
which ends the proof of {\color{black} Proposition} \ref{prop:3.2}.
\end{proof}

\subsection{{\color{black} The estimate of the trilinear term with exponential weight}}
To prove the regularity in {\color{black} Gelfand-Shilov spaces},
{\color{black}we need an upper bound on }non linear operators with exponential {\color{black} weight}.

\begin{proposition}\label{trilinear}
For any $f$,$g$,$h\in\mathscr{S}(\mathbb{R}^3)\bigcap\mathcal{N}^{\perp}$, any $N\geq0$, and for any $c>0$, we have
\begin{align}\label{tril-B}
\begin{split}
|({\bf \Gamma}&(f,g),e^{c\,
\mathcal{H}^s}\mathbf{S}_Nh)_{L^2}|\\
&\leq C\|e^{\frac{c}{2}\mathcal{H}^s}\mathbf{S}_{N-2}f\|_{L^2}
\|e^{\frac{c}{2}\mathcal{H}^s}\mathcal{L}^{\frac{1}{2}}\mathbf{S}_{N-2}g\|_{L^2}
\|e^{\frac{c}{2}\mathcal{H}^s}\mathcal{L}^{\frac{1}{2}}\mathbf{S}_{N}h\|_{L^2},
\end{split}
\end{align}
where C is a positive constant only dependent on $s$, and $\mathbf{S}_N$ is the
orthogonal projector such that,
\begin{align*}
&\mathbf{S}_N f =
  \sum_{\substack{2n+l\leq N\\n\geq0,l\geq0}} \,\,
  \sum_{|m|\leq l}(f,\varphi_{n,l,m})_{L^2} \, \varphi_{n,l,m},\\
&e^{c\,\mathcal{H}^s}\mathbf{S}_Nf =
  \sum_{\substack{2n+l\leq N\\n\geq0,l\geq0}} \,\,
  \sum_{|m|\leq l}e^{c\, \, (2n+l+\frac{3}{2})^s}
    (f,\varphi_{n,l,m})_{L^2} \, \varphi_{n,l,m}.
\end{align*}
\end{proposition}
\begin{remark}\,\,
1) For $h\in\mathscr{S}(\mathbb{R}^3)$, we can't use $e^{\frac{c}{2}\mathcal{H}^s} h$ as test function, since {\color{black} we don't know if it belongs}
to $\mathscr{S}(\mathbb{R}^3)$.
However, for any $h\in\mathscr{S}^{\prime}(\mathbb{R}^3)$, we have  $e^{\frac{c}{2}\mathcal{H}^s}\mathbf{S}_Nh\,\in \mathscr{S}(\mathbb{R}^3)$.

2) In the right hand side of \eqref{tril-B}, the projector of $f$ and $g$ with $\mathbf{S}_{N-2}$ {\color{black} shows} more {\color{black} clearly} the triangular effect of
 ${\bf \Gamma}(\,\cdot,\,\cdot\,)$.
\end{remark}

\begin{proof}
Since $f$,$g$,$h\in\mathscr{S}(\mathbb{R}^3)\bigcap\mathcal{N}^{\perp}$,\,
similarly to {\color{black} Proposition} \ref{prop:3.2},
we have
\begin{align*}
&\left({\bf \Gamma}(f,g),\, e^{c\,\mathcal{H}^s}\mathbf{S}_N h\right)_{L^2(\mathbb{R}^3)}
\\
&=\sum^{[\frac{N}{2}]}_{n^{\star}=2}
e^{c\,(2n^{\star}+\frac{3}{2})^s} \overline{h_{n^{\star},0,0}}
\left(\sum_{\substack{n+\tilde{n}=n^{\star}\\n\geq2,\tilde{n}\geq2}}
\lambda^{rad, 1}_{n,\tilde{n},0}f_{n,0,0}\,g_{\tilde{n},
0,0}\right)
\\
&\quad+
\sum_{\substack{2n^{\star}+l^{\star}\leq N\\
n^{\star}\geq0,l^{\star}\geq1,{\color{black}n^{\star}+l^{\star}\geq2}}}
\sum_{|m^{\star}|\leq l^{\star}} e^{c\,(2n^{\star}+l^{\star}+\frac{3}{2})^s} \overline{h_{n^{\star},l^{\star},m^{\star}}}
\left(\sum_{\substack{n+\tilde{n}=n^{\star}\\n\geq2,\tilde{n}\geq0,{\color{black}\tilde{n}\geq2-l^{\star}}}}
\lambda^{rad, 1}_{n,\tilde{n},l^{\star}}f_{n,0,0}\,g_{\tilde{n},
l^{\star},m^{\star}}\right)
\\
&\quad+
\sum_{\substack{2n^{\star}+l^{\star}\leq N\\
n^{\star}\geq0,l^{\star}\geq1,{\color{black}n^{\star}+l^{\star}\geq2}}}
\sum_{|m^{\star}|\leq l^{\star}}e^{c\,(2n^{\star}+l^{\star}+\frac{3}{2})^s}\overline{h_{n^{\star},l^{\star},m^{\star}}}
\left(\sum_{\substack{n+\tilde{n}=n^{\star}\\n\geq0,\tilde{n}\geq2,{\color{black}n\geq2-l^{\star}}}}
\lambda^{rad, 2}_{n,\tilde{n},l^{\star}}f_{n,l^{\star},m^{\star}}\,g_{\tilde{n},
0, 0}\right)
\\
&\quad
+\sum_{\substack{2n+2\tilde{n}+l+\tilde{l}\leq N\\n+l\geq2,\tilde{n}+\tilde{l}\geq2\\n\geq0,l\geq1,\tilde{n}\geq0,\tilde{l}\geq1}}
\sum_{|m|\leq l}\sum_{|\tilde{m}|\leq \tilde{l}}
f_{n,l,m}g_{\tilde{n},\tilde{l},\tilde{m}}
e^{c\,(2n+2\tilde{n}+l+\tilde{l}+\frac{3}{2})^s}
\\
&\qquad\qquad\times\left(\sum^{\min(l,\tilde{l})}_{k=0}
\sum_{|m^{\star}|\leq l+\tilde{l}-2k}\,
\mu^{m,\tilde{m},m^{\star}}_{n,\tilde{n},l,
\tilde{l},k}\,
\overline{h_{n+\tilde{n}+k,l+\tilde{l}-2k,m^{\star}}}\right)
\\
&=J_1+J_2 +J_3+J_4.
\end{align*}
{\color{black}
For the term $J_1$, by using $\lambda_{n,0}\approx n^s$
and Cauchy-Schwarz inequality, we obtain,
\begin{align*}
|J_1|&\leq
\sum_{\substack{n+\tilde{n}\leq[\frac{N}{2}]\\n\geq2,\tilde{n}\geq2}}e^{c\,(2(n+\tilde{n})+\frac{3}{2})^s}|\lambda^{rad, 1}_{n,\tilde{n},0}|\, |f_{n,0,0}|\, |g_{\tilde{n},0,0}||h_{n+\tilde{n},0,0}|\\
&\leq
\left(\sum^{[\frac{N}{2}]-2}_{\tilde{n}=2}e^{c\,(2\tilde{n}+\frac{3}{2})^s}\lambda_{\tilde{n},0}|g_{\tilde{n},0,0}|^2\right)^{1/2}
\\
&\qquad\times\left(\sum^{[\frac{N}{2}]-2}_{\tilde{n}=2}e^{-c\,(2\tilde{n}+\frac{3}{2})^s}\frac{1}{\lambda_{\tilde{n},0}}
\Bigg|\sum^{[\frac{N}{2}]-\tilde{n}}_{n=2}e^{c\,(2(n+\tilde{n})+\frac{3}{2})^s}|\lambda^{rad, 1}_{n,\tilde{n},0}|\, |f_{n,0,0}|\, |h_{n+\tilde{n},0,0}|\Bigg|^2\right)^{\frac{1}{2}}\\
&\leq
\|e^{\frac{c}{2}\mathcal{H}^s}\mathcal{L}^{\frac{1}{2}}\mathbf{S}_{N-2}g\|_{L^2}
\|e^{\frac{c}{2}\mathcal{H}^s}\mathbf{S}_{N-2}f\|_{L^2}\\
&\qquad\times
\left(\sum^{[\frac{N}{2}]-2}_{\tilde{n}=2}\sum^{[\frac{N}{2}]-\tilde{n}}_{n=2}e^{2c\,(2n+2\tilde{n}+\frac{3}{2})^s-c\,(2\tilde{n}+\frac{3}{2})^s-c(2n+\frac{3}{2})^s}|h_{n+\tilde{n},
0,0}|^2\frac{|\lambda^{rad, 1}_{n,\tilde{n},0}|^2}{\lambda_{\tilde{n},0}}\right)^{\frac{1}{2}}.
\end{align*}
For any $0<s<1$,$n,\tilde{n}\in \mathbb{N}$,
$$(2n+2\tilde{n}+\frac{3}{2})^s\leq(2\tilde{n}+\frac{3}{2})^s+(2n+\frac{3}{2})^s.$$
Then
\begin{align*}
|J_1|\leq\|e^{\frac{c}{2}\mathcal{H}^s}\mathcal{L}^{\frac{1}{2}}\mathbf{S}_{N-2}g\|_{L^2}
\|e^{\frac{c}{2}\mathcal{H}^s}\mathbf{S}_{N-2}f\|_{L^2}\left(\sum^{[\frac{N}{2}]-2}_{\tilde{n}=2}\sum^{[\frac{N}{2}]-\tilde{n}}_{n=2}
e^{c\,(2n+2\tilde{n}+\frac{3}{2})^s}|h_{n+\tilde{n},
0,0}|^2\frac{|\lambda^{rad, 1}_{n,\tilde{n},0}|^2}{\lambda_{\tilde{n},0}}\right)^{\frac{1}{2}}.
\end{align*}
It follows by the same estimate as $I_1$,
\begin{align*}
|J_1|
&\lesssim\|e^{\frac{c}{2}\mathcal{H}^s}\mathcal{L}^{\frac{1}{2}}\mathbf{S}_{N-2}g\|_{L^2}
\|e^{\frac{c}{2}\mathcal{H}^s}\mathbf{S}_{N-2}f\|_{L^2}
\left(\sum^{[\frac{N}{2}]}_{n^{\star}=4}e^{c\,(2n^{\star}+\frac{3}{2})^s}|h_{n^{\star},0,0}|^2
\Bigg(\sum_{\substack{n+\tilde{n}=n^{\star}\\
n\geq2,\tilde{n}\geq2}}\frac{\tilde{n}^s}{n^{\frac{5}{2}+2s}}
\Bigg)\right)^{\frac{1}{2}}\\
&\lesssim\|e^{\frac{c}{2}\mathcal{H}^s}\mathcal{L}^{\frac{1}{2}}\mathbf{S}_{N-2}g\|_{L^2}
\|e^{\frac{c}{2}\mathcal{H}^s}\mathbf{S}_{N-2}f\|_{L^2}\|e^{\frac{c}{2}\mathcal{H}^s}\mathcal{L}^{\frac{1}{2}}\mathbf{S}_{N}h\|_{L^2}.
\end{align*}
For the term $J_2$,
\begin{align*}
|J_2|\leq&\sum_{\substack{2\tilde{n}+l^{\star}\leq\,N-4\\\tilde{n}\geq0,l^{\star}\geq1,\tilde{n}+l^{\star}\geq2}}\sum_{|m^{\star}|\leq l^{\star}}
e^{\frac{c}{2}(2\tilde{n}+l^{\star}+\frac{3}{2})^s}
|g_{\tilde{n},l^{\star},m^{\star}}|
\sum_{\substack{n\geq2\\2n+2\tilde{n}+l^{\star}\leq N}}e^{\frac{c}{2}(2n+\frac{3}{2})^s}|f_{n,0,0}|\\
&\qquad\times\,e^{c\,(2n+2\tilde{n}+l^{\star}+\frac{3}{2})^s-\frac{c(2n+\frac{3}{2})^s}{2}
-\frac{c(2\tilde{n}+l^{\star}+\frac{3}{2})^s}{2}} |\lambda^{rad, 1}_{n,\tilde{n},l^{\star}}h_{n+\tilde{n},l^{\star},m^{\star}}|.
\end{align*}
Since for any $0<s<1$, $n,\tilde{n},l^{\star}\in \mathbb{N}$,
\begin{equation*}
(2n+2\tilde{n}+l^{\star}+\frac{3}{2})^s\leq(2n+\frac{3}{2})^s+(2\tilde{n}+l^{\star}+\frac{3}{2})^s,
\end{equation*}
we can deduce from Cauchy-Schwarz inequality that,
\begin{align*}
|J_2|&\leq
\left(\sum_{\substack{2\tilde{n}+l^{\star}\leq\,N-4\\\tilde{n}\geq0,l^{\star}\geq1,\tilde{n}+l^{\star}\geq2}}\sum_{|m^{\star}|\leq l^{\star}}
e^{c(2\tilde{n}+l^{\star}+\frac{3}{2})^s}
\lambda_{\tilde{n},l^{\star}}|g_{\tilde{n},l^{\star},m^{\star}}|^2\right)^{\frac{1}{2}}\\
&\qquad\times\Bigg(\sum_{\substack{2\tilde{n}+l^{\star}\leq\,N-4\\\tilde{n}\geq0,l^{\star}\geq1,\tilde{n}+l^{\star}\geq2}}\sum_{|m^{\star}|\leq l^{\star}}
\frac{1}{\lambda_{\tilde{n},l^{\star}}}\Bigg|\sum_{\substack{n\geq2\\2n+2\tilde{n}+l^{\star}\leq N}}e^{\frac{c}{2}(2n+\frac{3}{2})^s}|f_{n,0,0}|\\
&\quad\qquad\times\,e^{\frac{c}{2}\,(2n+2\tilde{n}+l^{\star}+\frac{3}{2})^s} |\lambda^{rad, 1}_{n,\tilde{n},l^{\star}}h_{n+\tilde{n},l^{\star},m^{\star}}|\Bigg|^2\Bigg)^{\frac{1}{2}}
\\
&\leq\|e^{\frac{c}{2}\mathcal{H}^s}\mathcal{L}^{\frac{1}{2}}\mathbf{S}_{N-2}g\|_{L^2}
\|e^{\frac{c}{2}\mathcal{H}^s}\mathbf{S}_{N-2}f\|_{L^2}\\
&\qquad\times
\left(\sum_{\substack{2\tilde{n}+l^{\star}\leq\,N-4\\\tilde{n}\geq0,l^{\star}\geq1,\tilde{n}+l^{\star}\geq2}}
\sum_{|m^{\star}|\leq l^{\star}}\sum_{\substack{n\geq2\\2n+2\tilde{n}+l^{\star}\leq N}}
e^{c(2n+2\tilde{n}+l^{\star}+\frac{3}{2})^s}|h_{n+\tilde{n},
l^{\star},\tilde{m}}|^2\frac{|\lambda^{rad, 1}_{n,\tilde{n},l^{\star}}|^2}{\lambda_{\tilde{n},l^{\star}}}\right)^{\frac{1}{2}}\\
&\leq\|e^{\frac{c}{2}\mathcal{H}^s}\mathcal{L}^{\frac{1}{2}}\mathbf{S}_{N-2}g\|_{L^2}
\|e^{\frac{c}{2}\mathcal{H}^s}\mathbf{S}_{N-2}f\|_{L^2}\\
&\qquad\times\Bigg[\sum_{\substack{4\leq2n^{\star}+l^{\star}\leq\,N\\n^{\star}\geq0,l^{\star}\geq1}}
\sum_{|m^{\star}|\leq\,l^{\star}}
e^{c(2n^{\star}+l^{\star}+\frac{3}{2})^s}|h_{n^{\star},l^{\star},m^{\star}}|^2
\Bigg(\sum_{\substack{n+\tilde{n}=n^{\star}\\
n\geq2,\tilde{n}\geq\max(0,2-l^{\star})}}\frac{1}{\lambda_{\tilde{n},l^{\star}}}
|\lambda^{rad, 1}_{n,\tilde{n},l^{\star}}|^2\Bigg)\Bigg]^{\frac{1}{2}}.
\end{align*}
Recall the estimate \eqref{estimateI-2} that,
\begin{align*}
\sum_{\substack{n+\tilde{n}=n^{\star}\\
n\geq2,\tilde{n}\geq\max(0,2-l^{\star})}}\frac{1}{\lambda_{\tilde{n},l^{\star}}}
|\lambda^{rad, 1}_{n,\tilde{n},l^{\star}}|^2\lesssim(n^{\star})\lesssim\lambda_{n^{\star},l^{\star}},
\end{align*}
we obtain,
$$
|J_2|\lesssim\|e^{\frac{c}{2}\mathcal{H}^s}\mathcal{L}^{\frac{1}{2}}\mathbf{S}_{N-2}g\|_{L^2}
\|e^{\frac{c}{2}\mathcal{H}^s}\mathbf{S}_{N-2}f\|_{L^2}
\|e^{\frac{c}{2}\mathcal{H}^s}\mathcal{L}^{\frac{1}{2}}\mathbf{S}_{N}h\|_{L^2}.
$$
Similarly, for the third term $J_3$,
\begin{align*}
|J_3|&\leq\sum^{\frac{N}{2}-1}_{\tilde{n}=2}e^{\frac{c}{2}(2\tilde{n}+\frac{3}{2})^s}|g_{\tilde{n},0,0}|
\sum_{\substack{n\geq0,l^{\star}\geq1,n+l^{\star}\geq2\\2n+2\tilde{n}+l^{\star}\leq N}}\sum_{|m^{\star}|\leq l^{\star}}e^{\frac{c}{2}(2n+l^{\star}+\frac{3}{2})^s}|f_{n,l^{\star},m^{\star}}|\\
&\qquad\times\,
e^{c\,(2n+2\tilde{n}+l^{\star}+\frac{3}{2})^s-\frac{c(2\tilde{n}+\frac{3}{2})^s}{2}-\frac{c(2n+l^{\star}+\frac{3}{2})^s}{2}
} |\lambda^{rad, 2}_{n,\tilde{n},l^{\star}}h_{n+\tilde{n},l^{\star},m^{\star}}|.
\end{align*}
By using Cauchy-Schwarz inequality and the inequality
\begin{equation*}
(2n+2\tilde{n}+l^{\star}+\frac{3}{2})^s\leq(2\tilde{n}+\frac{3}{2})^s+(2n+l^{\star}+\frac{3}{2})^s,  \forall\,\,0<s<1,
\end{equation*}
we obtain,
\begin{align*}
|J_3|&\leq\|e^{\frac{c}{2}\mathcal{H}^s}\mathcal{L}^{\frac{1}{2}}\mathbf{S}_{N-2}g\|_{L^2}
\|e^{\frac{c}{2}\mathcal{H}^s}\mathbf{S}_{N-2}f\|_{L^2}\\
&\quad\times \Bigg[\sum_{\substack{4\leq2n^{\star}+l^{\star}\leq\,N\\n^{\star}\geq0,l^{\star}\geq1}}
\sum_{|m^{\star}|\leq\,l^{\star}}
e^{c(2n^{\star}+l^{\star}+\frac{3}{2})^s}|h_{n^{\star},l^{\star},m^{\star}}|^2
\Bigg(\sum_{\substack{n+\tilde{n}=n^{\star}\\\tilde{n}\geq2,n\geq\max(0, 2-l^{\star})}}\frac{|\lambda^{rad, 2}_{n,\tilde{n} ,l^{\star}}|^2}{\lambda_{\tilde{n},0}}\Bigg)\Bigg]^{\frac{1}{2}}.
\end{align*}
By using the estimate \eqref{estimateI-3}, one has,
$$
|J_3|\lesssim\|e^{\frac{c}{2}\mathcal{H}^s}\mathcal{L}^{\frac{1}{2}}\mathbf{S}_{N-2}g\|_{L^2}
\|e^{\frac{c}{2}\mathcal{H}^s}\mathbf{S}_{N-2}f\|_{L^2}
\|e^{\frac{c}{2}\mathcal{H}^s}\mathcal{L}^{\frac{1}{2}}\mathbf{S}_{N}h\|_{L^2}.
$$
Consider the fourth term $J_4$,
\begin{align*}
|J_4|&\leq
\sum_{\substack{2\tilde{n}+\tilde{l}\leq N-2\\ \tilde{n}+\tilde{l}\geq2,\,\tilde{n}\geq0, \tilde{l}\geq1}}
\sum_{|\tilde{m}|\leq \tilde{l}}
(\lambda_{\tilde{n},\tilde{l}})^{\frac{1}{2}}
e^{\frac{c(2\tilde{n}+\tilde{l}+\frac{3}{2})^s}{2}}|g_{\tilde{n},\tilde{l},\tilde{m}}| \sum_{\substack{2n+l\leq N-2\\n+l\geq2,\,l\geq1,n\geq0\\2n+2\tilde{n}+l+\tilde{l}\leq N}}
\sum_{|m|\leq l}e^{\frac{c\,(2n+l+\frac{3}{2})^s}{2}}|f_{n,l,m}|\\
&\qquad\times
e^{c\,(2n+2\tilde{n}+l+\tilde{l}+\frac{3}{2})^s-\frac{c(2n+l+\frac{3}{2})^s}{2}
-\frac{c(2\tilde{n}+\tilde{l}+\frac{3}{2})^s}{2}}\\
&\qquad\times\Big|
\sum^{\min(l,\tilde{l})}_{k=0}
\sum_{|m^{\star}|\leq l+\tilde{l}-2k}
\mu^{m,\tilde{m},m^{\star}}_{n,\tilde{n},l,
\tilde{l},k}\Big(\lambda_{\tilde{n},\tilde{l}}\Big)^{-\frac{1}{2}}\,
\overline{h_{n+\tilde{n}+k,l+\tilde{l}-2k,m^{\star}}}\Big|.
\end{align*}
Using the fact that} for any $0<s<1$, $n,\tilde{n},l,\tilde{l}\in \mathbb{N}$,

$$
(2n+l+2\tilde{n}+\tilde{l}+\frac{3}{2})^s\leq{\color{red}(2n+l+\frac{3}{2})^s+(2\tilde{n}+\tilde{l}+\frac{3}{2})^s,}
$$
we can deduce from Cauchy-Schwarz inequality that,
\begin{align*}
|J_4|&\leq\|e^{\frac{c}{2}\mathcal{H}^s}\mathcal{L}^{\frac{1}{2}}\mathbf{S}_{N-2}g\|_{L^2}
\|e^{\frac{c}{2}\mathcal{H}^s}\mathbf{S}_{N-2}f\|_{L^2}
\\
&\quad\times\Big[\sum_{\substack{4\leq2n+2\tilde{n}+l+\tilde{l}\leq N\\2\leq n+l,2\leq \tilde{n}+\tilde{l}\\n\geq0,l\geq1,\tilde{n}\geq0,\tilde{l}\geq1}}
\sum^{\min(l,\tilde{l})}_{k=0}\sum_{|m^{\star}|\leq l+\tilde{l}-2k}e^{c\,(2n+2\tilde{n}+l+\tilde{l}+\frac{3}{2})^s}\\
&\qquad\quad\times|h_{n+\tilde{n}+k,l+\tilde{l}-2k,m^{\star}}|^2
\sum_{|\tilde{m}|\leq \tilde{l}}\sum_{|m|\leq l}\frac{|\mu^{m,\tilde{m},m^{\star}}_{n,\tilde{n},l,
\tilde{l},k}|^2}{\lambda_{\tilde{n},\tilde{l}}}\Big]^{\frac{1}{2}}\\
&\leq\|e^{\frac{c}{2}\mathcal{H}^s}\mathcal{L}^{\frac{1}{2}}\mathbf{S}_{N-2}g\|_{L^2}
\|e^{\frac{c}{2}\mathcal{H}^s}\mathbf{S}_{N-2}f\|_{L^2}\\
&\qquad\times\Bigg[\sum_{4\leq2n^{\star}+l^{\star}\leq N}\sum_{|m^{\star}|\leq l+\tilde{l}-2k}e^{c\,(2n^{\star}+l^{\star}+\frac{3}{2})^s}|h_{n^{\star},l^{\star},m^{\star}}|^2\\
&\qquad\quad\times\Bigg(\sum_{\substack{n+\tilde{n}+k=n^{\star}\\n+l\geq2,\tilde{n}+\tilde{l}\geq2\\
n\geq0,
\tilde{n}\geq0}}
\sum_{\substack{l+\tilde{l}-2k=l^{\star}\\l\geq1,\tilde{l}\geq1\\0\leq\,k\leq\,\min(l,\tilde{l})}}
\sum_{|m|\leq\,l}
\sum_{|\tilde{m}|\leq\,\tilde{l}}
\frac{\Big|\,\mu^{m,\tilde{m},m^{\star}}_{n,\tilde{n},l,\tilde{l},k} \Big|^2}{\lambda_{\tilde{n},\tilde{l}}}\Bigg)\Bigg]^{\frac{1}{2}},
\end{align*}
The last summation is understanding as \eqref{remark-summation}.
We use again $3)$ of Proposition \ref{est},
$$\sum_{\substack{n+\tilde{n}+k=n^{\star}\\n+l\geq2,\tilde{n}+\tilde{l}\geq2\\
n\geq0,
\tilde{n}\geq0}}
\sum_{\substack{l+\tilde{l}-2k=l^{\star}\\l\geq1,\tilde{l}\geq1\\0\leq\,k\leq\,\min(l,\tilde{l})}}
\sum_{|m|\leq\,l}
\sum_{|\tilde{m}|\leq\,\tilde{l}}
\frac{\Big|\,\mu^{m,\tilde{m},m^{\star}}_{n,\tilde{n},l,\tilde{l},k} \Big|^2}{\lambda_{\tilde{n},\tilde{l}}}\leq C\lambda_{n^{\star},l^{\star}}.$$
We can finish the proof exactly as that of Proposition \ref{prop:3.2}.
\end{proof}


\section{The proof of the main Theorem}\label{S4}

In this section, we study the convergence of the formal solutions obtained on Section \ref{S2} with small $L^2$ initial data
{\color{black} and this ends} the proof of Theorem  \ref{trick}.

\subsection{The uniform estimate}
Let $\{g_{n,l,m}(t)\}$ be the solution of \eqref{ODE}. For any $N\in \mathbb{N}$, set
\begin{equation}\label{SN}
\mathbf{S}_Ng(t)=
\sum_{\substack{ 2n+l\leq N\\n\geq0,l\geq0}} \,\,
\sum_{|m|\leq l}g_{n,l,m}(t)\varphi_{n,l,m}.
\end{equation}
Then $\mathbf{S}_Ng(t),  e^{c_0\,t\,\mathcal{H}^s} \mathbf{S}_Ng(t) \in\mathscr{S}(\mathbb{R}^3)\bigcap\mathcal{N}^{\perp}$,

Multiplying $e^{c_0t(2n^\star+l^\star+\frac{3}{2})^s}\overline{g_{n^\star,l^\star,m^{\star}}}(t)$ on both sides of \eqref{ODE}
and {\color{black} taking} summation for $ 2n^\star+l^\star\leq N$, then Proposition \ref{expansion} and the orthogonality of the basis $(\varphi_{n,l,m})_{n,l\geq0,|m|\leq l}$ imply
\begin{align*}
&\Big(\partial_t (\mathbf{S}_Ng)(t),e^{c_0 t \mathcal{H}^s}\mathbf{S}_N g(t)\Big)_{L^2(\mathbb{R}^3)}+\Big(\mathcal{L}(\mathbf{S}_Ng)(t), e^{c_0 t \mathcal{H}^s}\mathbf{S}_N g(t)\Big)_{L^2(\mathbb{R}^3)}\\
&\qquad=\Big(\mathbf{\Gamma}(\mathbf{S}_Ng,\mathbf{S}_Ng), e^{c_0 t \mathcal{H}^s}\mathbf{S}_N g(t)\Big)_{L^2(\mathbb{R}^3)}.
\end{align*}
Since $\mathbf{S}_Ng(t)\in\mathscr{S}(\mathbb{R}^3)\bigcap\mathcal{N}^{\perp}$, we have
$$
\Big(\mathcal{L}(\mathbf{S}_Ng)(t), e^{c_0 t \mathcal{H}^s}\mathbf{S}_N g(t)\Big)_{L^2(\mathbb{R}^3)}=
\|e^{\frac{c_0 t}{2} \mathcal{H}^s}\mathcal{L}^{\frac 12}\mathbf{S}_Ng(t)\|_{L^2(\mathbb{R}^3)},
$$
we then obtain that
\begin{align*}
\frac{1}{2}\frac{d}{dt}\|e^{\frac{c_0t}{2}\mathcal{H}^s}\mathbf{S}_Ng(t)\|^2_{L^2}
&-\frac{c_0}{2}\|e^{\frac{c_0t}{2}\mathcal{H}^s}\mathcal{H}^{\frac{s}{2}}
\mathbf{S}_{N}g\|^2_{L^2}+{\color{black}\|e^{\frac{c_0 t}{2} \mathcal{H}^s}\mathcal{L}^{\frac 12}\mathbf{S}_Ng(t)\|^2_{L^2}}\\
&=\Big(\mathbf{\Gamma}((\mathbf{S}_Ng),(\mathbf{S}_Ng)),e^{c_0t\mathcal{H}^s}
\mathbf{S}_Ng(t)\Big)_{L^2}.
\end{align*}
It follows from \eqref{eq:3.111} and {\color{black} Proposition} \ref{trilinear} that, for $0\le c_0\le c_1$  and any $N\geq2$, $t\geq0$,
\begin{align}\label{ele}
\begin{split}
\frac{1}{2}\frac{d}{dt}\|e^{\frac{c_0t}{2}\mathcal{H}^s}\mathbf{S}_Ng(t)\|^2_{L^2}
&+\frac{1}{2}\|e^{\frac{c_0t}{2}\mathcal{H}^s}\mathcal{L}^{\frac{1}{2}}\mathbf{S}_{N}g\|^2_{L^2}\\
&\leq C\|e^{\frac{c_0t}{2}\mathcal{H}^s}\mathbf{S}_{N-2}g\|_{L^2}
\|e^{\frac{c_0t}{2}\mathcal{H}^s}\mathcal{L}^{\frac{1}{2}}\mathbf{S}_{N}g\|^2_{L^2}.
\end{split}
\end{align}

\begin{proposition}\label{induction}
There exist $\epsilon_0>0$ and $\tilde{c}_0>0$ such that for all $0<\epsilon\leq\epsilon_0, 0\le c_0\le \tilde{c}_0$,\, $g_0\in L^2\bigcap\mathcal{N}^{\perp}$ with $\|g_0\|_{L^2}\leq\epsilon$, then,
\begin{equation*}
\|e^{\frac{c_0t}{2}\mathcal{H}^s}\mathbf{S}_{N}g(t)\|^2_{L^2(\mathbb{R}^3)}
+\frac{1}{2}\int^t_0\|e^{\frac{c_0 \tau }{2}\mathcal{H}^s}\mathcal{L}^{\frac{1}{2}}
\mathbf{S}_{N}g(\tau)\|^2_{L^2}d\tau
\le\|g_0\|^2_{L^2(\mathbb{R}^3)},
\end{equation*}
for any $t\geq0,\, N\geq 0$.
\end{proposition}

\begin{proof} We prove the Proposition by induction on $N$.

{\bf 1). For $N\le 2$.} we have
$\|e^{\frac{c_0t}{2}\mathcal{H}^s}\mathbf{S}_{0}g\|^2_{L^2}=|g_{0,0,0}(t)|^2=0,\,$
$$\|e^{\frac{c_0t}{2}\mathcal{H}^s}\mathbf{S}_{1}g\|^2_{L^2}=
|g_{0,0,0}(t)|^2+\sum_{|m|\leq1}e^{c_0\,t}|g_{0,1,m}(t)|^2=0,$$
and
$$
\|e^{\frac{c_0t}{2}\mathcal{H}^s}\mathbf{S}_{2}g\|^2_{L^2}=
{\color{black}e^{c_0t(2+\frac{3}{2})^s}}|g_{1,0,0}(t)|^2+\sum_{|m|\leq2}e^{c_0t(2+\frac{2}{3})^s}|g_{0,2,m}(t)|^2.
$$
{\color{black} Recalling from \eqref{2-0}\,\text{ in Section 2.2}} that for all $t>0$
$$
g_{0,2,m}(t)=e^{-\lambda_{0,2}t}g_{0,2,m}(0),
$$
we choose $0<\tilde{c}_0$ small such that $\tilde{c}_0(2+3/2)^s-2\lambda_{0, 2}\le 0$.
{\color{black} Therefore}
$$
\|e^{\frac{c_0t}{2}\mathcal{H}^s}\mathbf{S}_{2}g\|^2_{L^2}\leq\sum_{|m|\leq2}|g_{0,2,m}(0)|^2\leq\|g_0\|^2_{L^2}\leq\epsilon^2
$$
for all $0\le c_0\le \tilde{c}_0, 0<\epsilon\le \epsilon_0$.

{\bf 2). For $N> 2$.} We want to prove that
$$
\|e^{\frac{c_0t}{2}\mathcal{H}^s}\mathbf{S}_{N-1}g\|_{L^2}\leq\epsilon\leq\epsilon_0
$$
{\color{black} implies}
$$
\|e^{\frac{c_0t}{2}\mathcal{H}^s}\mathbf{S}_{N}g\|_{L^2}\leq\epsilon.
$$
Take now $\epsilon_0>0$ such that
$$
0<\epsilon_0\leq\frac{1}{4C}.
$$
Then we deduce from \eqref{ele} that
\begin{align*}
\frac{1}{2}\frac{d}{dt}\|e^{\frac{c_0t}{2}\mathcal{H}^s}\mathbf{S}_Ng(t)\|^2_{L^2}
&+\frac{1}{2}\|e^{\frac{c_0t}{2}\mathcal{H}^s}\mathcal{L}^{\frac{1}{2}}
\mathbf{S}_{N}g\|^2_{L^2}\\
&\leq C\|e^{\frac{c_0t}{2}\mathcal{H}^s}\mathbf{S}_{N-2}g\|_{L^2}
\|e^{\frac{c_0t}{2}\mathcal{H}^s}\mathcal{L}^{\frac{1}{2}}
\mathbf{S}_{N}g\|^2_{L^2}\\
&\leq\frac{1}{4}\|e^{\frac{c_0t}{2}\mathcal{H}^s}
\mathcal{L}^{\frac{1}{2}}\mathbf{S}_{N}g\|^2_{L^2}.
\end{align*}
{\color{black} Therefore}
\begin{equation}\label{indu1}
\frac{d}{dt}\|e^{\frac{c_0t}{2}\mathcal{H}^s}\mathbf{S}_Ng(t)\|^2_{L^2}
+\frac{1}{2}\|e^{\frac{c_0t}{2}\mathcal{H}^s}\mathcal{L}^{\frac{1}{2}}
\mathbf{S}_{N}g\|^2_{L^2}\le 0,
\end{equation}
{\color{black} and this ends} the proof of the Proposition.
\end{proof}

\subsection{Existence of the weak solution}\,\,
Now we prove the convergence of the sequence
$$
g(t)=\sum_{\substack{n,l\geq0}} \,\,
\sum_{|m|\leq l}g_{n,l,m}(t)\varphi_{n,l,m}
$$
defined in \eqref{ODE-b}.

Multiplying $\varphi_{n^\star,l^\star,m^{\star}}(v)$ on both sides of \eqref{ODE}
and {\color{black} taking the} summation for $ 2n^\star+l^\star\leq N$,
then for all $N\geq2$,
$\mathbf{S}_{N}g(t)$ satisfies the following Cauchy problem
\begin{equation}\label{eq-2}
\left\{ \begin{aligned}
         &\partial_t \mathbf{S}_{N}g+\mathcal{L}(\mathbf{S}_{N}g)=\mathbf{S}_{N}\mathbf{\Gamma}(\mathbf{S}_{N}g, \mathbf{S}_{N}g),\,\\
         &
         \mathbf{S}_{N}g(0)=
         \sum_{\substack{n\geq0, l\geq0\\2n+l \leq N}}
         \,\,
         \sum_{|m|\leq\,l}
         g^0_{n,l,m}\varphi_{n,l,m}.
\end{aligned} \right.
\end{equation}
By Proposition \ref{induction} with $c_0=0$, we have for all $t>0$, $N\in\mathbb{N}$,
\begin{equation*}
\|\mathbf{S}_{N}g(t)\|^2_{L^2(\mathbb{R}^3)}
+\frac{1}{2}\int^t_0\|\mathcal{L}^{\frac{1}{2}}
\mathbf{S}_{N}g(\tau)\|^2_{L^2}d\tau
\le\|g_0\|^2_{L^2(\mathbb{R}^3)}.
\end{equation*}
The orthogonality of the basis $(\varphi_{n,l,m})_{n,l\geq0,|m|\leq l}$ implies that
$$
\|\mathbf{S}_{N}g(t)\|^2_{L^2(\mathbb{R}^3)}=
\sum_{\substack{ 2n+l\leq N\\n\geq0,l\geq0}} \,\,
\sum_{|m|\leq l}|g_{n,l,m}(t)|^2.
$$
By using the monotone convergence theorem, the sequence
$$
g(t)=\sum_{n,l\geq0}\,\,
\sum_{|m|\leq l}g_{n,l,m}(t)\varphi_{n,l,m}
$$
is convergent, for any $t\ge 0$,
$$
\lim_{N\to \infty}\|\mathbf{S}_{N}g-g\|_{L^\infty([0, t]; L^2(\mathbb{R}^3))}=0
$$
and
$$
\lim_{N\to \infty}\|\mathcal{L}^{\frac{1}{2}}(\mathbf{S}_{N}g-g)\|_{L^2([0, t]; L^2(\mathbb{R}^3))}=0.
$$
For any $\phi(t)\in\,C^1\Big(\mathbb{R}_+,\mathscr{S}(\mathbb{R}^3)\Big)$,  we have
\begin{align*}
&\Big|\int^t_0\Big(\mathbf{S}_{N}\mathbf{\Gamma}(\mathbf{S}_{N}g, \mathbf{S}_{N}g)-\mathbf{\Gamma}(g, g),
\phi(\tau)\Big)_{L^2(\mathbb{R}^3)}d\tau\Big|\\
&=\Big|\int^t_0\Big(\mathbf{\Gamma}(\mathbf{S}_{N}g, \mathbf{S}_{N}g),
\mathbf{S}_{N}\phi(\tau)\Big)_{L^2(\mathbb{R}^3)}-\Big(\mathbf{\Gamma}(g, g),\phi(\tau)\Big)_{L^2(\mathbb{R}^3)}d\tau\Big|\\
&\le\Big|\int^t_0\Big(\mathbf{\Gamma}(\mathbf{S}_{N}g, \mathbf{S}_{N}g),\mathbf{S}_{N}\phi(\tau)-\phi(\tau)\Big)_{L^2}d\tau\Big|\\
&\qquad+\Big|\int^t_0\Big(\mathbf{\Gamma}(\mathbf{S}_{N}g-g, \mathbf{S}_{N}g),\phi(\tau)\Big)_{L^2}d\tau\Big|+\Big|\int^t_0\Big(\mathbf{\Gamma}(g, \mathbf{S}_{N}g-g),\phi(\tau)\Big)_{L^2}d\tau\Big|.
\end{align*}
By Proposition \ref{prop:3.2}, one can verify that
\begin{align*}
&\Big|\int^t_0\Big(\mathbf{S}_{N}\mathbf{\Gamma}(\mathbf{S}_{N}g, \mathbf{S}_{N}g)-\mathbf{\Gamma}(g, g),
\phi(\tau)\Big)_{L^2(\mathbb{R}^3)}d\tau\Big|\\
&\leq\,C\int^t_0\|\mathbf{S}_{N}g\|_{L^2}\|\mathcal{L}^{\frac 12}\mathbf{S}_{N}g\|_{L^2} \Big\|\mathcal{L}^{\frac 12}(
\mathbf{S}_{N}\phi-\phi)\Big\|_{L^2(\mathbb{R}^3)} dt\\
&\qquad+C\int^t_0 \|\mathbf{S}_{N}g-g\|_{L^2}\|\mathcal{L}^{\frac 12}\mathbf{S}_{N}g\|_{L^2} \|\mathcal{L}^{\frac 12}\phi\|_{L^2(\mathbb{R}^3)}
dt\\
&\qquad+C\int^t_0\|g\|_{L^2}\|\mathcal{L}^{\frac 12}(\mathbf{S}_{N}g-g)\|_{L^2(\mathbb{R}^3)}\|\mathcal{L}^{\frac 12}\phi\|_{L^2(\mathbb{R}^3)}
dt.
\end{align*}
Using Proposition \ref{induction} with $c_0=0$
\begin{align*}
&\Big|\int^t_0\Big(\mathbf{S}_{N}\mathbf{\Gamma}(\mathbf{S}_{N}g, \mathbf{S}_{N}g)-\mathbf{\Gamma}(g, g),
\phi(\tau)\Big)_{L^2(\mathbb{R}^3)}d\tau\Big|\\
&\leq\,C\|\mathbf{S}_{N}g\|_{L^\infty([0, t]; L^2)}\|\mathcal{L}^{\frac 12}\mathbf{S}_{N}g\|_{L^2([0, t]; L^2)} \Big\|\mathcal{L}^{\frac 12}(
\mathbf{S}_{N}\phi-\phi)\Big\|_{L^2([0, t]; L^2(\mathbb{R}^3))} \\
&\qquad+C\|\mathbf{S}_{N}g-g\|_{L^\infty([0, t]; L^2)}\|\mathcal{L}^{\frac 12}\mathbf{S}_{N}g\|_{L^2([0, t]; L^2)} \|\mathcal{L}^{\frac 12}\phi\|_{L^2([0, t]; L^2)}\\
&\qquad+C\|g\|_{L^\infty([0, t]; L^2)}\|\mathcal{L}^{\frac 12}(\mathbf{S}_{N}g-g)\|_{L^2([0, t]; L^2)}\|\mathcal{L}^{\frac 12}\phi\|_{L^2([0, t]; L^2)},
\end{align*}
then
\begin{align*}
&\Big|\int^t_0\Big(\mathbf{S}_{N}\mathbf{\Gamma}(\mathbf{S}_{N}g, \mathbf{S}_{N}g)-\mathbf{\Gamma}(g, g),
\phi(\tau)\Big)_{L^2(\mathbb{R}^3)}d\tau\Big|\\
&\leq\,{\color{black}\sqrt{2}C}\|g_0\|^2_{L^2} \Big\|\mathcal{L}^{\frac 12}(
\mathbf{S}_{N}\phi-\phi)\Big\|_{L^2([0, t]; L^2(\mathbb{R}^3))} \\
&\qquad+{\color{black}\sqrt{2}C}\|\mathbf{S}_{N}g-g\|_{L^\infty([0, t]; L^2)}\|g_0\|_{L^2} \|\mathcal{L}^{\frac 12}\phi\|_{L^2([0, t]; L^2)}\\
&\qquad+C\|g_0\|_{L^2}\|\mathcal{L}^{\frac 12}(\mathbf{S}_{N}g-g)\|_{L^2([0, t]; L^2)}\|\mathcal{L}^{\frac 12}\phi\|_{L^2([0, t]; L^2)}.
\end{align*}
{\color{black} Letting} $N\rightarrow+\infty$ in \eqref{eq-2}, we conclude that, for any $\phi(t)\in\,C^1\Big(\mathbb{R}_+,\mathscr{S}(\mathbb{R}^3)\Big)$,
\begin{align*}
&\Big(g(t), \phi(t)\Big)_{L^2(\mathbb{R}^3)}-\Big(g(0), \phi(0)\Big)_{L^2(\mathbb{R}^3)}
{\color{black} -\int^t_{0} \Big(g(\tau), \partial_t \phi(\tau)\Big)_{L^2(\mathbb{R}^3)}
d\tau}
\\
&=-\int^t_{0}\Big(\mathcal{L}g(\tau), \phi(\tau)\Big)_{L^2(\mathbb{R}^3)}d\tau+\int^t_{0}\Big(\mathbf{\Gamma}(g(\tau),g(\tau)), \phi(\tau)\Big)_{L^2(\mathbb{R}^3)}d\tau,
\end{align*}
which shows that $g\in L^\infty(]0, +\infty[; L^2(\mathbb{R}^3))$ is a global weak solution of the Cauchy problem \eqref{eq-1}.

\subsection{Regularity of the solution.}\,\, For $\mathbf{S}_{N}g$ defined in \eqref{SN},\,since
$$
\lambda_{n,l}\geq\lambda_{2,0}>0, \, \forall\,n+l\geq2,
$$
we deduce from the formulas  \eqref{indu1} that
\begin{align*}
&\frac{d}{dt}\|e^{\frac{c_0t}{2}\mathcal{H}^s}\mathbf{S}_Ng(t)\|^2_{L^2}
+\frac{\lambda_{2,0}}{2}\|e^{\frac{c_0t}{2}\mathcal{H}^s}\mathbf{S}_{N}g\|^2_{L^2}\\
&\leq \frac{d}{dt}\|e^{\frac{c_0t}{2}\mathcal{H}^s}\mathbf{S}_Ng(t)\|^2_{L^2}
+\frac{1}{2}\|e^{\frac{c_0t}{2}\mathcal{H}^s}\mathcal{L}^{\frac{1}{2}}\mathbf{S}_{N}g\|^2_{L^2}\leq0.
\end{align*}
We {\color{black} then} have
$$
\frac{d}{dt}\Big(e^{\frac{\lambda_{2,0}t}{2}}\|e^{\frac{c_0t}{2}\mathcal{H}^s}
\mathbf{S}_Ng(t)\|^2_{L^2}\Big)
\leq 0,
$$
and we derive that for any $t>0$, and $N\in\mathbb{N}$,
{\color{black}\begin{equation*}
\|e^{\frac{c_0t}{2}\mathcal{H}^s}\mathbf{S}_{N}g(t)\|_{L^2(\mathbb{R}^3)}
\le\,e^{-\frac{\lambda_{2,0}t}{4}}\|g_0\|_{L^2(\mathbb{R}^3)}.
\end{equation*}
}
The {\color{black} orthogonality} of the basis $(\varphi_{n,l,m})_{n,l\geq0,|m|\leq l}$ implies that
$$
\|e^{\frac{c_0t}{2}\mathcal{H}^s}\mathbf{S}_{N}g(t)\|^2_{L^2(\mathbb{R}^3)}=
\sum_{\substack{2n+l\leq N\\n\geq0,l\geq0}}\sum_{|m|\leq l}
{\color{black}e^{c_0t(2n+l+\frac 32)^s}}|g_{n,l,m}(t)|^2.
$$
Using the monotone convergence theorem, we conclude that
{\color{black}
$$
\|e^{\frac{c_0t}{2}\mathcal{H}^s}g(t)\|_{L^2(\mathbb{R}^3)}\le\,
e^{-\frac{\lambda_{2,0}t}{4}}\|g_0\|_{L^2(\mathbb{R}^3)}.
$$
}
This ends the proof of Theorem \ref{trick}.


\section{The spectral representation}\label{S5-1}

This section is devoted to the proof of {\color{black} Proposition}
\ref{expansion}, {\color{black} Proposition} \ref{expansion2}
and some Propositions used in section \ref{S5-2}.

\subsection{Harmonic identities}

We prepare some technical computation.
In all this section, $n$, $l$, $\tilde{n}$, $\tilde{l}$
will be fixed integers of $\mathbb{N}$, {\color{black}$m\in\mathbb{Z}, |m|\leq l$} and
we will use the following notation in this section :
\begin{equation*}
 a_{l,m}=\frac{(l-|m|)!}{(l+|m|)!}.
\end{equation*}
For any unit vector
\[
\sigma = (\sigma_1,\sigma_2,\sigma_3)
= (\cos\theta,\sin\theta\cos\phi,\sin\theta\sin\phi)\in\mathbb{S}^2
\]
with $\theta\in[0,\pi]$ and $\phi\in[0,2\pi]$,
the orthonormal basis of spherical harmo\-nics $Y^{m}_{l}(\sigma)$
($|m|\leq l$) is (see the definition \eqref{Plm} of $P_l^{|k|}$)
\begin{align}\label{Ylm2}
\begin{split}
Y^{m}_{l}(\sigma)
&=
\sqrt{\frac{2l+1}{4\pi} \, \frac{(l-|m|)!}{(l+|m|)!}} \,
P^{|m|}_{l}(\cos\theta) e^{im\phi},
\\
&=
\sqrt{\frac{2l+1}{4\pi} \, \frac{(l-|m|)!}{(l+|m|)!}} \,
\left( \frac{d^{|m|} P_l}{dx^{|m|}}\right)(\sigma_1) \,
(\sigma_2 + i\, \operatorname{sgn}(m) \, \sigma_3)^{|m|}.
\end{split}
\end{align}
We recall the following properties
(see {\color{black}$(7-34)$ of Chapter 7 in the book \cite{JCSlater},
(VIII) of Sec.19, Chap.III in the book \cite{San}
and Theorem 1 of Sec.4, Chap 1 in the book \cite{CM}}):
\\
- Addition theorem:
For any integer $l\ge0$ and
$\alpha_1$, $\alpha_2$ in $\mathbb{S}^2$,
\begin{align}\label{addition theorem Y}
P_{l}(\alpha_1\cdot\alpha_2)
=\frac{4\pi}{2\,l+1}\sum^{l}_{m=-l}
Y^{m}_{l}(\alpha_1)\,
Y^{-m}_{l}(\alpha_2).
\end{align}
If we set the following coordinates
{\color{black}
$$\alpha_j=(\cos\theta_j, \sin\theta_j\cos\phi_j,
  \sin\theta_j\sin\phi_j)$$
}
for $j=1,2$,
the previous addition theorem reads as follows
\begin{align}\label{addition theorem}
P_{l}(\cos\theta_1
 &\cos\theta_2  +\sin\theta_1\sin\theta_2 \cos(\phi_1-\phi_2))\nonumber\\
&=\sum^{l}_{m=-l}
\frac{(l-|m|)!}{(l+|m|)!}
P^{|m|}_{l}(\cos\theta_1)\,
P^{|m|}_{l}(\cos\theta_2)\,e^{i m(\phi_1-\phi_2)}.
\end{align}
- Integral form of the addition theorem:
For any integer $l\ge0$ and $m$, $|m|\leq l$,
any $\sigma\in \mathbb{S}^2$,
\begin{align}\label{addition theorem int}
Y_l^m(\sigma) = \frac{2\,l+1}{4\,\pi}
\int_{\mathbb{S}^2_{\eta}} P_l(\sigma\cdot\eta) \, Y_l^m(\eta) \, d\eta.
\end{align}
- Funk-Hecke Formula: For any continuous function $f\in \mathcal{C}([-1,1])$, any $\sigma\in \mathbb{S}^2$ and integers $l\ge0$, $|m|\leq l$,
\begin{align}\label{funk-hecke}
\int_{\mathbb{S}^2_{\eta}} f(\sigma\cdot\eta) \, Y_l^m(\eta) \, d\eta=
\left(2\pi \int^1_{-1} f(x)\, P_{l}(x)\,dx\right) \, Y_l^m(\sigma) .
\end{align}

For $\kappa\in \mathbb{S}^2$ fixed,
we can find $\theta_0\in[0,\pi]$, $\phi_0\in[0, 2\pi])$ such that
\begin{equation}
\kappa=(\cos\theta_0,\sin\theta_0\cos\phi_0,\sin\theta_0\sin\phi_0).\nonumber
\end{equation}
Construct the orthogonal vectors
with respect to $\kappa$
\begin{equation}\label{k,k1,k2}
\kappa^1=(-\sin\theta_0,\cos\theta_0\cos\phi_0,\cos\theta_0\sin\phi_0),\,\, \kappa^2=(0,\sin\phi_0,-\cos\phi_0),
\end{equation}
and for $\phi\in\mathbb{R}$
\begin{equation}\label{k orthog}
\kappa^{\bot}(\phi)=\kappa^1\cos\phi+\kappa^2\sin\phi.
\end{equation}
Then $\kappa,\kappa^1,\kappa^2$ constitute
an orthonormal frame in $\mathbb{R}^3$.
For any unit vector $\sigma$,
we can find $\theta\in[0,\pi]$
and $\phi\in[0,2\pi]$ such that
\begin{equation}\label{tran}
\sigma = \kappa\cos\theta + \kappa^1\sin\theta\cos\phi
                          + \kappa^2\sin\theta\sin\phi.
\end{equation}
It is easy to verify
\begin{align}\label{k+s}
\frac{\kappa+\sigma}{|\kappa+\sigma|}
&= \kappa \, \cos\frac{\theta}{2} \,
+ \sin\frac{\theta}{2} \,
  \left( \kappa^1\cos\phi+\kappa^2\sin\phi \right),
\\
\frac{\kappa-\sigma}{|\kappa-\sigma|}
&= \kappa \, \sin\frac{\theta}{2}
- \cos\frac{\theta}{2} \,
  \left( \kappa^1\cos\phi+\kappa^2\sin\phi \right). \quad
  \label{k-s}
\end{align}

In the proof of {\color{black} Proposition} \ref{expansion},
we need the following lemma.
\begin{lemma}\label{lemma funk 2}
For any function $f$ in $C([-1,1])$ any $\kappa\in\mathbb{S}^2$,
$l \in\mathbb{N}$\,and\,$|m|\leq l$,\,we have
\begin{align*}
(i)\,\,\,
\int_{\mathbb{S}^2}
f(\kappa\cdot\sigma) \,
Y^{m}_{l}
  \left({\textstyle\frac{\kappa+\sigma}{|\kappa+\sigma|}}\right)
d\sigma\nonumber
&=
\left( 2 \,\pi
\int_{0}^{\pi}
f(\cos\theta) \,\sin\theta  \,
 P_{l}
   \left(\cos{\textstyle \frac{\theta}{2}} \right)
d\theta
\right)\,
Y^{m}_{l}(\kappa),
\\
(ii)\,\,\,
\int_{\mathbb{S}^2}
f(\kappa\cdot\sigma) \,
Y^{m}_{l}
  \left({\textstyle\frac{\kappa-\sigma}{|\kappa-\sigma|}}\right)
d\sigma\nonumber
&=
\left(2 \,\pi
\int_{0}^{\pi}
f(\cos\theta) \, \sin\theta  \,
 P_{l}
   \left(\sin{\textstyle \frac{\theta}{2}} \right)
d\theta
\right)\,
Y^{m}_{l}(\kappa).
\end{align*}
\end{lemma}

\begin{proof}
For $\kappa\in \mathbb{S}^2$ fixed,
we can find $\theta_0\in[0,\pi]$,
$\phi_0\in[0,2\pi]$ such that
\begin{equation}
\kappa=(\cos\theta_0,\sin\theta_0\cos\phi_0,\sin\theta_0\sin\phi_0).\nonumber
\end{equation}
In the orthonormal frame $(\kappa,\kappa^1,\kappa^2)$
constructed in \eqref{k,k1,k2}, for any $\sigma\in\mathbb{S}^2$
we have
\begin{equation*}
\sigma = \kappa\cos\theta + \kappa^1\sin\theta\cos\phi
                          + \kappa^2\sin\theta\sin\phi
\end{equation*}
with $\theta\in[0,\pi]$ and $\phi\in[0,2\pi]$.\,\,
Therefore, $\kappa\cdot\sigma=\cos\theta$,
and for any $\eta\in\mathbb{S}^2$ with
\begin{equation*}
\eta = \kappa\cos\theta_1 + \kappa^1\sin\theta_1\cos\phi_1
                          + \kappa^2\sin\theta_1\sin\phi_1.
\end{equation*}
we deduce from \eqref{k+s}-\eqref{k-s}
\begin{align}
\frac{\kappa-\sigma}{|\kappa-\sigma|}\cdot\eta
&= \sin\frac{\theta}{2}\,\cos\theta_1
 -\cos\frac{\theta}{2}\,\sin\theta_1\, \cos(\phi-\phi_1),
\label{k-s.eta}\\
\frac{\kappa+\sigma}{|\kappa+\sigma|}\cdot\eta
&=\cos\frac{\theta}{2} \cos\theta_1
+ \sin\frac{\theta}{2} \sin\theta_1  \cos(\phi-\phi_1).
\label{k+s.eta}
\end{align}
{Proof of (i).}
Applying the formula \eqref{addition theorem int}  for
$Y^{m}_{l}
  \left(\frac{\kappa+\sigma}{|\kappa+\sigma|}\right)$,
we have
\begin{align*}
\int_{\mathbb{S}^2_\sigma}
  f(\kappa\cdot\sigma) \,
  &Y^{m}_{l}
  \left(\frac{\kappa+\sigma}{|\kappa+\sigma|}\right)
d\sigma
\\
&=\frac{2\,l+1}{4\,\pi}
    \int_{\mathbb{S}^2_\sigma} f(\kappa\cdot\sigma) \,
    \int_{\mathbb{S}^2_\eta}
    P_{l}
  \left(\frac{\kappa+\sigma}{|\kappa+\sigma|}\cdot\eta\right) \,
    Y^{m}_{l}(\eta)
  d\eta \,
d\sigma
\\
&={\color{black}  \frac{2\,l+1}{4\,\pi}\int_{\mathbb{S}^2_\eta}
  Y^{m}_{l}(\eta)
  \, A(\eta)
d\eta}
\end{align*}
where
\begin{equation*}
  A(\eta) = \int_{\mathbb{S}^2}
  f(\kappa\cdot\sigma) \,
  P_{l}
  \left(\frac{\kappa+\sigma}{|\kappa+\sigma|}\cdot\eta\right) \,
d\sigma.
\end{equation*}
Then, applying the addition theorem \eqref{addition theorem}
and \eqref{k+s.eta}
\begin{align*}
P_{l}\left(\frac{\kappa+\sigma}{|\kappa+\sigma|}\cdot\eta\right)
&=
P_{l}(
  \cos\frac{\theta}{2} \cos\theta_1
+ \sin\theta_1 \sin\frac{\theta}{2} \cos(\phi-\phi_1)
)\\
&=
\sum^{l}_{{m}=-l}
\frac{(l-|m|)!}{(l+|m|)!}
  P^{|m|}_{l}(\cos\frac{\theta}{2})
  P^{|m|}_{l}(\cos\theta_1)
  e^{im(\phi-\phi_1)},
\end{align*}
direct calculation shows that
\begin{align*}
A(\eta)
&=\left(
2 \, \pi\int_{0}^{\pi}
  f(\cos\theta) \, \sin\theta \,
  P_{l} \left(\cos\frac{\theta}{2} \right)
  d\theta
\right) \,
P_{l}(\kappa\cdot\eta).
\end{align*}
Henceforth, we get that
\begin{align*}
&\int_{\mathbb{S}^2_\sigma}
f(\kappa\cdot\sigma) \,
  Y^{m}_{l}
  \left(\frac{\kappa+\sigma}{|\kappa+\sigma|}\right)
d\sigma
\\
&=
\left(2 \,\pi
\int_{0}^{\pi}
f(\cos\theta)  \,\sin\theta  \,
 P_{l}\left(\cos \frac{\theta}{2}\right)
d\theta
\right)\,
\frac{2\,l+1}{4\,\pi}
\int_{\mathbb{S}^2_\eta}
  Y^{m}_{l}(\eta) \,
P_{l}(\kappa\cdot\eta)
d\eta
\end{align*}
and we conclude by formula \eqref{addition theorem int}.

The proof of $(ii)$ is similar by using \eqref{k-s.eta}.
\end{proof}

As a direct consequence of part $(i)$ of the previous lemma, we have :
\begin{corollary}\label{coroll int f Y+}
For $\tilde{l},\tilde{m}\in\mathbb{N}$\,
and\,$|\tilde{m}|\leq \tilde{l}$,\,we have
for the cross section $b(\cos\theta)$ satisfying \eqref{beta} {\color{black}with $\cos\theta\geq0$} ,
\begin{align}\label{funk 2}
\begin{split}
\int_{\mathbb{S}^2}
b(\kappa\cdot\sigma)
&\left(
  Y^{\tilde{m}}_{\tilde{l}}
  \left(\frac{\kappa+\sigma}{|\kappa+\sigma|}\right)
  \left(\frac{1+\kappa\cdot\sigma}{2}\right)^{\frac{2\tilde{n}+\tilde{l}}{2}}
  -Y^{\tilde{m}}_{\tilde{l}}(\kappa)
\right)d\sigma\\
&=\left[
\int_{|\theta|\leq\frac{\pi}{4}}
\beta(\theta)\left((\cos\theta)^{2\tilde{n}+\tilde{l}}
 P_{\tilde{l}}(\cos\theta)-1\right)d\theta
\right]Y^{\tilde{m}}_{\tilde{l}}(\kappa).
\end{split}
\end{align}
\end{corollary}

\begin{lemma}\label{lemma:5.1}
Let $\kappa\in \mathbb{S}^2$ and the cross section $b(\cos\theta)$ satisfies \eqref{beta} {\color{black}with $\cos\theta\geq0$}.
Assume also that $n,l,\tilde{n},\tilde{l}\in\mathbb{N}$
with $l\geq1$, $\tilde{l}\geq1$, $|m|\leq l$, $|\tilde{m}|\leq\tilde{l}$.
Then there {\color{black} exist} some constants
$c_{n,l,m,\tilde{n},\tilde{l},\tilde{m}}^k$ such that
\begin{align}\label{nonlinear}
\begin{split}
\int_{\mathbb{S}^2}
  b(\kappa\cdot\sigma)
  &Y^{m}_l\left(\frac{\kappa-\sigma}{|\kappa-\sigma|}\right)
  Y^{\tilde{m}}_{\tilde{l}}\left(\frac{\kappa+\sigma}{|\kappa+\sigma|}\right)
  \left(\frac{1-\kappa\cdot\sigma}{2}\right)^{\frac{2n+l}{2}}
  \left(\frac{1+\kappa\cdot\sigma}{2}\right)^{\frac{2\tilde{n}+\tilde{l}}{2}}
d\sigma\\
&=
\sum_{k=0}^{k_0(l,\tilde{l},m,\tilde{m})}
  c_{n,l,m,\tilde{n},\tilde{l},\tilde{m}}^k \,
  Y^{m+\tilde{m}}_{l+\tilde{l}-2k}(\kappa).
\end{split}
\end{align}
\end{lemma}

{\color{black} 
\begin{proof}
Without loss of generality, we can assume that $\min(l,\tilde{l})=\tilde{l}$
(the integral in \eqref{nonlinear} is invariant if $\sigma$ is changed to $-\sigma$ and the integers $(l,m)$ and $(\tilde{l},\tilde{m})$ are exchanged).
%
\\
- {\sl Step 1}. We first claim {\color{black}that} there exist some constants $C_{\tilde{l},l_1,l_2}$
such that
\begin{align}\label{muti2}
\begin{split}
Y^{\tilde{m}}_{\tilde{l}}
  \left(\frac{\kappa+\sigma}{|\kappa+\sigma|}\right)
=
&\sum_{ l_1+l_2=\tilde{l} }
  C_{\tilde{l},l_1,l_2} \,
  \frac{(\frac{1-\kappa\cdot\sigma}{2})^{\frac{l_2}{2}}}
       {(\frac{1+\kappa\cdot\sigma}{2})^{\frac{\tilde{l}}{2}}}\\
&\times
\sum^{l_1}_{m_1=-l_1}
\sum^{l_2}_{m_2=-l_2}
\left(
  \int_{\mathbb{S}^2_\eta}
    Y^{\tilde{m}}_{\tilde{l}}(\eta)
    Y^{-m_1}_{l_1}(\eta)
    Y^{-m_2}_{l_2}(\eta)d\eta\right)
    Y^{m_1}_{l_1}(\kappa)
    Y^{m_2}_{l_2}\left(\frac{\kappa-\sigma}{|\kappa-\sigma|}\right)
\end{split}
\end{align}
where, in all the sequel of the proof, $l_1$ and $l_2$ in the summation will be {\color{black}non-negative} integers.

Proof of \eqref{muti2} :
From the integral addition theorem \eqref{addition theorem int}
we have
\begin{align*}
Y^{\tilde{m}}_{\tilde{l}}
  \left(\frac{\kappa+\sigma}{|\kappa+\sigma|}\right)
=
\frac{2\tilde{l}+1}{4\pi}
  \int_{\mathbb{S}^2_{\eta}}
    P_{\tilde{l}}
      \left(\frac{\kappa+\sigma}{|\kappa+\sigma|}\cdot\eta\right)
    Y^{\tilde{m}}_{\tilde{l}}(\eta)d\eta.
\end{align*}
From the formula (see $(43)$ in Chapter III in \cite{San}),
there is some real constants $c_{\tilde{l},q}$ such that
\begin{equation}\label{x^l=sum Pl}
x^{\tilde{l}}  =  
\sum_{0\leq q \leq \tilde{l}/2} c_{{\tilde{l}},q}  P_{{\tilde{l}}-2q}(x)
\end{equation}
{\color{black}
where $$c_{\tilde{l},0}=\frac{\tilde{l}!}{1\times3\times5\times\cdots\times(2\tilde{l}-1)}\neq0.$$
}
Writing  $P_{{\tilde{l}}}(x)  =
(1/c_{{\tilde{l}},0})  x^{\tilde{l}}
- (1/c_{{\tilde{l}},0})
\sum_{1\leq q \leq \tilde{l}/2}
  c_{{\tilde{l}},q}  P_{{\tilde{l}}-2q}(x)$
  in the previous integral,
we obtain
\begin{equation}\label{Ylm=1}
Y^{\tilde{m}}_{\tilde{l}}
  \left(\frac{\kappa+\sigma}{|\kappa+\sigma|}\right)
=
\frac{2\tilde{l}+1}{4\pi \, c_{\tilde{l},0}}
  \int_{\mathbb{S}^2_{\eta}}
      \left(\frac{\kappa+\sigma}{|\kappa+\sigma|}\cdot\eta\right)^{\tilde{l}}
    Y^{\tilde{m}}_{\tilde{l}}(\eta)d\eta + R_1
\end{equation}
where
\begin{equation*}
R_1 = - \frac{2\tilde{l}+1}{4\pi\, c_{\tilde{l},0}}
\sum_{{\color{black}1\leq q< \tilde{l}/2}} c_{\tilde{l},q}
  \int_{\mathbb{S}^2_{\eta}}
      P_{\tilde{l}-2q}\left(\frac{\kappa+\sigma}{|\kappa+\sigma|}\cdot\eta\right)
    Y^{\tilde{m}}_{\tilde{l}}(\eta)d\eta.
\end{equation*}
We observe that $R_1=0$. Indeed,
for any $q\neq \tilde{l}$ and $\gamma\in\mathbb{S}^2$, we have from the Funck-Hecke
Formula \eqref{funk-hecke} and the orthogonality of the polynomials $(P_l)_l$,
\begin{align*}
\int_{\mathbb{S}^2_{\eta}}
 P_q(\gamma\cdot\eta) \,  Y^{\tilde{m}}_{\tilde{l}}(\eta)d\eta
 = \left(2\pi \int_{-1}^1 P_q(x) P_{\tilde{l}}(x) \, dx \right)
 Y^{\tilde{m}}_{\tilde{l}}(\gamma) =0.
\end{align*}
Set $\gamma=\frac{\kappa-\sigma}{|\kappa-\sigma|}$.
Therefore we have
\[
\kappa+\sigma  = 2\kappa-2(\kappa\cdot\gamma)\gamma,\,\,
  |\kappa+\sigma| = 2\sqrt{1-(\kappa\cdot\gamma)^2}
\]
and
\begin{equation*} 
\left(\frac{\kappa+\sigma}{|\kappa+\sigma|} \cdot \eta\right)^{\tilde{l}}
= \frac{\left((\kappa\cdot\eta)
  - (\kappa\cdot\gamma)\,(\gamma \cdot \eta)\right)^{\tilde{l}}}
     {(\sqrt{1-(\kappa\cdot\gamma)^2})^{\tilde{l}}}.
\end{equation*}
Expanding this identity thanks to the binomial formula
and plugging it in \eqref{Ylm=1},
we get \begin{align}\label{muti1}
Y^{\tilde{m}}_{\tilde{l}}
  \left(\frac{\kappa+\sigma}{|\kappa+\sigma|}\right)
&=
\frac{2\tilde{l}+1}{4\pi \, c_{\tilde{l},0}}
\sum_{l_1+l_2=\tilde{l}}
  \frac{\tilde{l}!}{l_1!l_2!}
 \frac{(-\kappa\cdot\gamma)^{l_2}}{(\sqrt{1-(\kappa\cdot\gamma)^2})^{\tilde{l}}}
  \int_{\mathbb{S}^2_{\eta}}
    (\kappa\cdot\eta)^{l_1}  (\gamma\cdot\eta)^{l_2}
    Y^{\tilde{m}}_{\tilde{l}}(\eta)d\eta . 
\end{align}
Using the expansion \eqref{x^l=sum Pl} for $x^{l_1}$ and $x^{l_2}$,
we express the previous integrals :
\begin{align*} 
\int_{\mathbb{S}^2_{\eta}}
  (\kappa\cdot\eta)^{l_1}  (\gamma\cdot\eta)^{l_2}
  Y^{\tilde{m}}_{\tilde{l}}(\eta)d\eta
= \sum_{\substack{0\leq q_1 \leq l_1/2 \\ 0\leq q_2 \leq l_2/2}}
c_{l_1,q_1} c_{l_2,q_2}
\int_{\mathbb{S}^2_{\eta}}
  P_{l_1-2q_1}(\kappa\cdot\eta)  P_{l_2-2q_2}(\gamma\cdot\eta)
  Y^{\tilde{m}}_{\tilde{l}}(\eta)d\eta.
\end{align*}
{\color{black}
From the addition Theorem \eqref{addition theorem Y} applied to
both $P_{l_1-2q_1}(\kappa\cdot\eta)$ and $P_{l_2-2q_2}(\gamma\cdot\eta)$,
}
we derive
\begin{align*}
\int_{\mathbb{S}^2_{\eta}}
&    P_{l_1-2q_1}(\kappa\cdot\eta) \, P_{l_2-2q_2}(\gamma\cdot\eta)\,
    Y^{\tilde{m}}_{\tilde{l}}(\eta) \, d\eta
= {\frac{4\pi}{2(l_1-2q_1)+1} \, \frac{4\pi}{2(l_2-2q_2)+1}}
\nonumber\\
&\times
\sum^{l_1-2q_1}_{m_1=-(l_1-2q_1)}
\sum^{l_2-2q_2}_{m_2=-(l_2-2q_2)}
\left(
  \int_{\mathbb{S}^2_\eta}
    Y^{-m_1}_{l_1-2q_1}(\eta)
    Y^{-m_2}_{l_2-2q_2}(\eta)
    Y^{\tilde{m}}_{\tilde{l}}(\eta)
  d\eta\right)
Y^{m_1}_{l_1-2q_1}(\kappa)Y^{m_2}_{l_2-2q_2}(\gamma).
\end{align*}
From \eqref{con},
the integral
$(\int_{\mathbb{S}^2}
Y^{-m_1}_{l_1-2q_1} Y^{-m_2}_{l_2-2q_2} Y^{\tilde{m}}_{\tilde{l}})$
is equal to 0 when the parameters do not satisfy
\begin{equation*}
\tilde{m}=m_1+m_2\,\,\text{and}\,\,\tilde{l}=(l_1-2q_1)+(l_2-2q_2)-2j
\,\,\text{with}\,\,0\leq j\leq\min(l_1-2q_1,l_2-2q_2),
\end{equation*}
consequently for $(q_1,q_2)\neq(0,0)$ (recall that $l_1+l_2=\tilde{l}$).
Therefore we obtain
\begin{align*}
\int_{\mathbb{S}^2_{\eta}}
  (\kappa\cdot\eta)^{l_1}  (\gamma\cdot\eta)^{l_2}
&  Y^{\tilde{m}}_{\tilde{l}}(\eta)d\eta
= {\frac{4\pi \, c_{l_1,0} }{2l_1+1} \, \frac{4\pi \, c_{l_2,0}}{2l_2+1}}
\\
&\times
\sum^{l_1}_{m_1=-l_1}
\sum^{l_2}_{m_2=-l_2}
\left(
  \int_{\mathbb{S}^2_\eta}
    Y^{-m_1}_{l_1}(\eta)
    Y^{-m_2}_{l_2}(\eta)
    Y^{\tilde{m}}_{\tilde{l}}(\eta)
  d\eta\right)
Y^{m_1}_{l_1}(\kappa)Y^{m_2}_{l_2}(\gamma).
\end{align*}
Plugging the previous identity 
into \eqref{muti1}, we have for some constants
$C_{\tilde{l},l_1,l_2}$
\begin{align*}
Y^{\tilde{m}}_{\tilde{l}}
  \left(\frac{\kappa+\sigma}{|\kappa+\sigma|}\right)
=
\sum_{l_1+l_2=\tilde{l}}
&  C_{\tilde{l},l_1,l_2}
  \frac{(-\kappa\cdot\gamma)^{l_2}}
    {(\sqrt{1-(\kappa\cdot\gamma)^2})^{\tilde{l}}}
\nonumber\\
&\times
\sum^{l_1}_{m_1=-l_1}
\sum^{l_2}_{m_2=-l_2}
\left(
  \int_{\mathbb{S}^2_\eta}
    Y^{\tilde{m}}_{\tilde{l}}(\eta)
    Y^{-m_1}_{l_1}(\eta)
    Y^{-m_2}_{l_2}(\eta)d\eta\right)
    Y^{m_1}_{l_1}(\kappa)Y^{m_2}_{l_2}(\gamma).
\end{align*}
Remembering that
$\gamma=\frac{\kappa-\sigma}{|\kappa-\sigma|}$
and checking that
$\kappa\cdot\gamma = (\frac{1-\kappa\cdot\sigma}{2})^{\frac{1}{2}}$,
$\sqrt{1-(\kappa\cdot\gamma)^2} =
  (\frac{1+\kappa\cdot\sigma}{2})^{\frac{1}{2}}$,
we conclude the proof of  \eqref{muti2}.

- {\sl Step 2}. We now prove \eqref{nonlinear}. We note
\begin{align*}
{\bf N} =
\int_{\mathbb{S}^2}
  b(\kappa\cdot\sigma)
  &Y^{m}_l\left(\frac{\kappa-\sigma}{|\kappa-\sigma|}\right)
  Y^{\tilde{m}}_{\tilde{l}}\left(\frac{\kappa+\sigma}{|\kappa+\sigma|}\right)
  \left(\frac{1-\kappa\cdot\sigma}{2}\right)^{\frac{2n+l}{2}}
  \left(\frac{1+\kappa\cdot\sigma}{2}\right)^{\frac{2\tilde{n}+\tilde{l}}{2}}
d\sigma.
\end{align*}
From step 1, we derive
\begin{align*}
{\bf N}
&= \sum_{l_1+l_2=\tilde{l}}
C_{\tilde{l},l_1,l_2} \,
\sum^{l_1}_{m_1=-l_1}
\sum^{l_2}_{m_2=-l_2}
\left(
  \int_{\mathbb{S}^2}
    Y^{\tilde{m}}_{\tilde{l}}
    Y^{-m_1}_{l_1}
    Y^{-m_2}_{l_2}\right)
\\
&\times\left(
\int_{\mathbb{S}^2}
  b(\kappa\cdot\sigma)
  \left(\frac{1-\kappa\cdot\sigma}{2}\right)^{\frac{2n+l+l_2}{2}}
  \left(\frac{1+\kappa\cdot\sigma}{2}\right)^{\frac{2\tilde{n}}{2}}
  Y^{m}_l\left(\frac{\kappa-\sigma}{|\kappa-\sigma|}\right)
 Y^{m_2}_{l_2}\left(\frac{\kappa-\sigma}{|\kappa-\sigma|}\right)
d\sigma    \right)
Y^{m_1}_{l_1}(\kappa).
\end{align*}
Moreover we have from \eqref{represent} and \eqref{con}
\begin{align*}
Y_l^{m}\left(\frac{\kappa-\sigma}{|\kappa-\sigma|}\right)
Y^{m_2}_{l_2}\left(\frac{\kappa-\sigma}{|\kappa-\sigma|}\right)
&=\sum_{l^{\prime}}
\sum_{m^{\prime}=-l^{\prime}}^{l^{\prime}}
\left(
  \int_{\mathbb{S}^2}Y_l^{m}Y^{m_2}_{l_2}
  \overline{Y^{m^{\prime}}_{l^{\prime}}}
\right)
Y^{m^{\prime}}_{l^{\prime}}\left(\frac{\kappa-\sigma}{|\kappa-\sigma|}\right)
\end{align*}
where $l^{\prime}$ is defined by $l^{\prime}=l+l_2-2j_1$
with $0 \leq j_1 \leq\min(l,l_2)$.
{\color{black}
Therefore
}%
\begin{align*}
{\bf N}
=
&\sum_{l_1+l_2=\tilde{l}}
C_{\tilde{l},l_1,l_2} \,
\sum^{l_1}_{m_1=-l_1}
\sum^{l_2}_{m_2=-l_2}
\left(
  \int_{\mathbb{S}^2}
    Y^{\tilde{m}}_{\tilde{l}}
    Y^{-m_1}_{l_1}
    Y^{-m_2}_{l_2}\right)
\sum_{l^{\prime}}\sum_{m^{\prime}=-l^{\prime}}^{l^{\prime}}
\left(
  \int_{\mathbb{S}^2}Y_l^{m}Y^{m_2}_{l_2}
  \overline{Y^{m^{\prime}}_{l^{\prime}}}
\right)
\\
&\times
\left(\int_{\mathbb{S}^2}
  b(\kappa\cdot\sigma)
  \left(\frac{1-\kappa\cdot\sigma}{2}\right)^{\frac{2n+l+l_2}{2}}
  \left(\frac{1+\kappa\cdot\sigma}{2}\right)^{\frac{2\tilde{n}}{2}}
  Y^{m^{\prime}}_{l^{\prime}}\left(\frac{\kappa-\sigma}{|\kappa-\sigma|}\right)
d\sigma
\right)   \,
Y^{m_1}_{l_1}(\kappa).
\end{align*}
We apply the part $(ii)$ of lemma \ref{lemma funk 2} with
$f(x)= b(x) \left(\frac{1-x}{2}\right)^{\frac{2n+l+l_2}{2}}
  \left(\frac{1+x}{2}\right)^{\frac{2\tilde{n}}{2}} $,
 {\color{black}the assumption $\cos\theta\geq0$ and the notation $\beta(\theta)=2\pi b(\cos2\theta)\sin2\theta$,
\begin{align*}
&\int_{\mathbb{S}^2}
  b(\kappa\cdot\sigma)
  \left(\frac{1-\kappa\cdot\sigma}{2}\right)^{\frac{2n+l+l_2}{2}}
  \left(\frac{1+\kappa\cdot\sigma}{2}\right)^{\frac{2\tilde{n}}{2}}
  Y^{m^{\prime}}_{l^{\prime}}\left(\frac{\kappa-\sigma}{|\kappa-\sigma|}\right)
d\sigma\\
&=\Big(2\pi\int^{\frac{\pi}{2}}_0b(\cos\theta)
  (\sin(\theta/2))^{2n+l+l_2}
  (\cos(\theta/2))^{2\tilde{n}}P_{l'}(\sin(\theta/2))\sin\theta d\theta\Big)Y^{m^{\prime}}_{l^{\prime}}(\kappa)\\
&=\Big(4\pi\int^{\frac{\pi}{4}}_0b(\cos2\theta)\sin2\theta
  (\sin\theta)^{2n+l+l_2}
  (\cos\theta)^{2\tilde{n}}P_{l'}(\sin\theta) d\theta\Big)Y^{m^{\prime}}_{l^{\prime}}(\kappa)\\
&=\Big(2\int^{\frac{\pi}{4}}_0\beta(\theta)
  (\sin\theta)^{2n+l+l_2}
  (\cos\theta)^{2\tilde{n}}P_{l'}(\sin\theta) d\theta\Big)Y^{m^{\prime}}_{l^{\prime}}(\kappa).
\end{align*}
}Therefore,%
\begin{align*}
{\bf N}
=
&\sum_{l_1+l_2=\tilde{l}}
C_{\tilde{l},l_1,l_2} \,
\sum^{l_1}_{m_1=-l_1}
\sum^{l_2}_{m_2=-l_2}
\left(
  \int_{\mathbb{S}^2}
    Y^{\tilde{m}}_{\tilde{l}}
    Y^{-m_1}_{l_1}
    Y^{-m_2}_{l_2}\right)
\sum_{l^{\prime}}\sum_{m^{\prime}=-l^{\prime}}^{l^{\prime}}
\left(
  \int_{\mathbb{S}^2}Y_l^{m}Y^{m_2}_{l_2}
  \overline{Y^{m^{\prime}}_{l^{\prime}}}
\right)
\\
&\times
\left(
\left(
2\int^{\frac{\pi}{4}}_0
  \beta(\theta)(\sin\theta)^{2n+l+l_2}
  (\cos\theta)^{2\tilde{n}}
  P_{l^{\prime}}(\sin\theta)
d\theta
\right)
Y^{m^{\prime}}_{l^{\prime}}(\kappa)
\right)
Y^{m_1}_{l_1}(\kappa).
\end{align*}
We again derive from \eqref{represent} and \eqref{con}
\begin{align}
Y_{l^{\prime}}^{m^{\prime}}(\kappa)
Y^{m_1}_{l_1}(\kappa)
&=\sum_{l^{\prime\prime}}
\sum_{m^{\prime\prime} =- l^{\prime\prime}}^{l^{\prime\prime}}
\left(
  \int_{\mathbb{S}^2}Y^{m^{\prime}}_{l^{\prime}}
  Y^{m_1}_{l_1}
  \overline{Y^{m^{\prime\prime}}_{l^{\prime\prime}}}
\right)
Y^{m^{\prime\prime}}_{l^{\prime\prime}}(\kappa)\nonumber
\end{align}
where $l^{\prime\prime}$ is defined by $l^{\prime\prime}=l^{\prime}+l_1-2j_2$
with $0 \leq j_2 \leq \min(l',l_1)$.
We conclude to
\begin{align*}
{\bf N}
=
&\sum_{l_1+l_2=\tilde{l}}
C_{\tilde{l},l_1,l_2} \,
\sum^{l_1}_{m_1=-l_1}
\sum^{l_2}_{m_2=-l_2}
\left(
  \int_{\mathbb{S}^2}
    Y^{\tilde{m}}_{\tilde{l}}
    Y^{-m_1}_{l_1}
    Y^{-m_2}_{l_2}\right)
\sum_{l^{\prime}}\sum_{m^{\prime}=-l^{\prime}}^{l^{\prime}}
\left(
  \int_{\mathbb{S}^2}Y_l^{m}Y^{m_2}_{l_2}
  \overline{Y^{m^{\prime}}_{l^{\prime}}}
\right)
\\
&\times
\left(
2\int^{\frac{\pi}{4}}_0
  \beta(\theta)(\sin\theta)^{2n+l+l_2}
  (\cos\theta)^{2\tilde{n}}
  P_{l^{\prime}}(\sin\theta)
d\theta
\right)
\sum_{l^{\prime\prime}}
\sum_{m^{\prime\prime} =- l^{\prime\prime}}^{l^{\prime\prime}}
\left(
  \int_{\mathbb{S}^2}Y^{m^{\prime}}_{l^{\prime}}
  Y^{m_1}_{l_1}
  \overline{Y^{m^{\prime\prime}}_{l^{\prime\prime}}}
\right)
Y^{m^{\prime\prime}}_{l^{\prime\prime}}(\kappa)
\end{align*}
which is nonzero from \eqref{con} when
$$m^{\prime\prime}=m^{\prime}+m_1=m+m_2+m_1=m+\tilde{m}.$$
From the previous expressions of $l^{\prime\prime}$ and $l^{\prime}$,
$$l^{\prime\prime}=l+l_1+l_2-2(j_1+j_2)=\tilde{l}+l-2(j_1+j_2)$$
and
$0\leq j_1+j_2\leq  \min(l,l_2) +  \min(l',l_1)$.
From $l_1+l_2=\tilde{l}$ and the assumption $\tilde{l}\leq l$,
we get
$$  \min(l,l_2) +  \min(l',l_1) = l_2 +  \min(l',l_1)
= \min(l_2+l',\tilde{l}) \leq \tilde{l}.$$
We derive
$0\leq j_1+j_2\leq \tilde{l} = \min(l,\tilde{l})$ .\,\,
For $l$, $\tilde{l}$, $m$, $\tilde{m}$ fixed,
we can define $l''=l+\tilde{l}-2k$ with
$0\leq k\leq  \min(l,\tilde{l})$.
Then the coefficient of $Y^{m^{\prime\prime}}_{l^{\prime\prime}}(\kappa)$
is nonzero when
 $$|m+\tilde{m}|\leq l+\tilde{l}-2k.$$
 Therefore,
 $$k\leq\frac{l+\tilde{l}-|m+\tilde{m}|}{2}. $$
 In conclusion,
 $$0\leq k\leq k_0(l,\tilde{l},m,\tilde{m})$$
 where $k_0(l,\tilde{l},m,\tilde{m})$ was given in \eqref{k_0}.
This ends the proof of \eqref{nonlinear}.
\end{proof}
}

\subsection{The proof of {\color{black} Proposition} \ref{expansion} }

The spectral representation will be based on the Bobylev formula, which is the Fourier transform of the Boltzmann operator
{\color{black}in the Maxwellian molecules case}:
\begin{equation*}
\mathcal{F}(Q(g,f))(\xi)=\int_{\mathbb{S}^2}b\left(\frac{\xi}{|\xi|}
\cdot\sigma\right)
\Big[\hat{g}(\xi^-)\hat{f}(\xi^+)-\hat{g}(0)\hat{f}(\xi)\Big]d\sigma,
\end{equation*}
where
$$
\xi^-=\frac{\xi-|\xi|\sigma}{2}=\frac{|\xi|}{2}(\kappa-\sigma),\qquad \xi^+=\frac{\xi+|\xi|\sigma}{2}=\frac{|\xi|}{2}(\kappa+\sigma)
$$
with $ \kappa =\frac{\xi}{|\xi|}$. Remark that
$$
\kappa\cdot\sigma =\cos\theta,
\quad
|\xi^-|
=|\xi|\,|\sin (\theta/2)|,\quad
|\xi^+|
=|\xi|\cos (\theta/2).
$$
Let $\varphi_{n,l,{m}}$ be the functions defined in \eqref{v}, then for $n,l\in\mathbb{N},\,|m|\leq l$, we have
(see Lemma \ref{Fouriertransform})
\begin{equation}\label{four-1}
\widehat{\sqrt{\mu}\varphi_{n,l,{m}}}(\xi)=
A_{n,l}
\Big(\frac{|\xi|}{\sqrt{2}}\Big)^{2n+l}e^{-\frac{|\xi|^2}{2}}Y_l^{m}\Big(\frac{\xi}{|\xi|}\Big),
\end{equation}
where
\begin{equation*}
A_{n,l}=(-i)^l(2\pi)^\frac{3}{4}\left(\frac{1}{\sqrt{2}n!
\Gamma(n+l+\frac{3}{2})}\right)^\frac{1}{2}.
\end{equation*}
In the special case $l=0$, {\color{black}this is the Hermite function,}
\begin{equation}\label{four-2}
\widehat{\sqrt{\mu}\varphi_{n,0,0}}(\xi)
=\frac{1}{\sqrt{(2n+1)!}}|\xi|^{2n}e^{-\frac{|\xi|^2}{2}}\,.
\end{equation}
We deduce from the Bobylev formula that,
$\forall\,n,\,l,\,m,\,\tilde{n},\,\tilde{l},\,\tilde{m}\in
\mathbb{N}$, with $|m|\leq l$,\,$|\tilde{m}|\leq \tilde{l}$,
\begin{align}\label{F G phi phi}
\begin{split}
&\mathcal{F}\left(\sqrt{\mu} \,
{\bf \Gamma}(\varphi_{n,l,m},
  \varphi_{\tilde{n},\tilde{l},\tilde{m}}\right)(\xi)
= \mathcal{F}(Q(\sqrt{\mu}\varphi_{n,l,m},\sqrt{\mu}
\varphi_{\tilde{n},\tilde{l},\tilde{m}}))(\xi)
\\
&=\int_{\mathbb{S}^2}b\left(\frac{\xi}{|\xi|}\cdot\sigma\right)
\Big[\widehat{\sqrt{\mu}
\varphi_{n,l,m}}(\xi^-)\widehat{\sqrt{\mu}
\varphi_{\tilde{n},\tilde{l},\tilde{m}}}(\xi^+)-
\widehat{\sqrt{\mu}\varphi_{n,l,m}}(0)\widehat{\sqrt{\mu}
\varphi_{\tilde{n},\tilde{l},\tilde{m}}}(\xi)\Big]d\sigma.
\end{split}
\end{align}
In the next propositions, we will compute the terms
${\bf \Gamma}(\varphi_{n,l,m},
  \varphi_{\tilde{n},\tilde{l},\tilde{m}})$
and proposition \ref{expansion} will follows.

\begin{proposition}\label{expansion 000 nlm}
The following algebraic identities hold,
{\color{black}
\begin{align}
(i) \quad\,\, &{\bf \Gamma}(\varphi_{0,0,0},\varphi_{\tilde{n},\tilde{l},\tilde{m}})\nonumber
\\
&=\left(
\int_{|\theta|\leq\frac{\pi}{4}}
  \beta(\theta)
  \left(
    (\cos\theta)^{2\tilde{n}+\tilde{l}}
    P_{\tilde{l}}(\cos\theta) - 1
  \right)
d\theta
\right)
\varphi_{\tilde{n},\tilde{l},\tilde{m}};\nonumber
\\
(ii) \quad\,\,
&{\bf \Gamma}(\varphi_{n,l,m},\varphi_{0,0,0})\nonumber
\\
&=\left(
\int_{|\theta|\leq\frac{\pi}{4}}
  \beta(\theta)
  \left(
    (\sin\theta)^{2n+l}
    P_{l}(\sin\theta) - \delta_{0,n} \, \delta_{0,l}
    \right)
d\theta
\right)\varphi_{n,l,m}.
\nonumber
\end{align}
}
\end{proposition}
This is exactly $(i_1)$ and $(i_2)$ of {\color{black} Proposition} \ref{expansion}.

\begin{proof} Since
$$
\widehat{\sqrt{\mu}\varphi_{n,l,m}}(0)=\delta_{n,0}\delta_{l,0},
$$
then when $n=0,l=0$,\,by using \eqref{four-1} and \eqref{four-2},  {\color{black}recall} that $\kappa=\xi/|\xi|$, we have
\begin{align}
&\mathcal{F}(Q(\sqrt{\mu}\varphi_{0,0,0},\sqrt{\mu}
\varphi_{\tilde{n},\tilde{l},\tilde{m}}))(\xi)\nonumber\\
&=\int_{\mathbb{S}^2}b
\left(\frac{\xi}{|\xi|}\cdot\sigma\right)
\left[
  \widehat{\sqrt{\mu}\varphi_{0,0,0}}(\xi^-)\widehat{\sqrt{\mu}
  \varphi_{\tilde{n},\tilde{l},\tilde{m}}}(\xi^+) -
  \widehat{\sqrt{\mu}\varphi_{0,0,0}}(0)
  \widehat{\sqrt{\mu}\varphi_{\tilde{n},\tilde{l},\tilde{m}}}(\xi)
\right] d\sigma.\nonumber\\
&=A_{\tilde{n},\tilde{l}}
e^{-\frac{|\xi|^2}{2}}
\left(\frac{|\xi|}{\sqrt{2}}\right)^{2\tilde{n}+\tilde{l}}
\int_{\mathbb{S}^2}b(\kappa\cdot\sigma)
\left[
  Y^{\tilde{m}}_{\tilde{l}}
  \left(\frac{\kappa+\sigma}{|\kappa+\sigma|}\right)
  \left(\frac{|\kappa+\sigma|}{2}\right)^{2\tilde{n}
  +\tilde{l}}-Y^{\tilde{m}}_{\tilde{l}}(\kappa)
\right]d\sigma.\nonumber
\end{align}
Apply now the identity \eqref{funk 2} of {\color{black} Corollary} \ref{coroll int f Y+}, one can find that,
\begin{align*}
&\mathcal{F}(Q(\sqrt{\mu}\varphi_{0,0,0},\sqrt{\mu}
\varphi_{\tilde{n},\tilde{l},\tilde{m}}))(\xi)\\
&=\int_{|\theta|\leq\frac{\pi}{4}}\beta(\theta)
\Big[(\cos\theta)^{2\tilde{n}+\tilde{l}}P_{\tilde{l}}
(\cos\theta)-1\Big]d\theta \widehat{\sqrt{\mu}\varphi_{\tilde{n},\tilde{l},
\tilde{m}}}(\xi).
\end{align*}
Hence by the inverse Fourier transform
\begin{align*}
{\bf \Gamma}(\varphi_{0,0,0},\varphi_{\tilde{n},\tilde{l},\tilde{m}})
&=\frac{1}{\sqrt{\mu}}Q(\sqrt{\mu}\varphi_{0,0,0},\sqrt{\mu}
\varphi_{\tilde{n},\tilde{l},\tilde{m}})\\
&=\left(\int_{|\theta|\leq\frac{\pi}{4}}\beta(\theta)
\Big((\cos\theta)^{2\tilde{n}+\tilde{l}}P_{\tilde{l}}
(\cos\theta)-1\Big)d\theta\right) \varphi_{\tilde{n},\tilde{l},
\tilde{m}}.
\end{align*}
The result of $(i)$ follows. Similar arguments apply to the case $(ii)$,
and this ends the proof of Proposition \ref{expansion 000 nlm}.
\end{proof}

\begin{proposition}\label{expansion n00 nlm}
The following algebraic identities hold,
\begin{align}
(i) \quad\,\,
&{\bf \Gamma}(\varphi_{n,0,0},\varphi_{\tilde{n},\tilde{l},\tilde{m}})\nonumber
\\
&={\color{black} \frac{1}{\sqrt{4\pi}}}\frac{A_{\tilde{n},\tilde{l}}A_{n,0}}{A_{n+\tilde{n},\tilde{l}}}
\left(
\int_{|\theta|\leq\frac{\pi}{4}}
  \beta(\theta)(\sin\theta)^{2n}
  (\cos\theta)^{2\tilde{n}+\tilde{l}}
  P_{\tilde{l}}(\cos\theta)
d\theta
\right)
\varphi_{n+\tilde{n},\tilde{l},\tilde{m}},\, \,\,n\geq1;\nonumber
\end{align}
\begin{align}
(ii) \quad\,\, &{\bf \Gamma}(\varphi_{n,l,m},\varphi_{\tilde{n},0,0})\nonumber
\\
&={\color{black} \frac{1}{\sqrt{4\pi}}}\frac{A_{\tilde{n},0}A_{n,l}}{A_{n+\tilde{n},l}}
\left(
\int_{|\theta|\leq\frac{\pi}{4}}
  \beta(\theta)(\sin\theta)^{2n+l}
  P_{l}(\sin\theta)
  (\cos\theta)^{2\tilde{n}}
d\theta
\right)
\varphi_{n+\tilde{n},l,m},\, \,\,l\geq1.\nonumber
\end{align}
\end{proposition}
This is exactly $(ii_1)$ and $(ii_2)$ of {\color{black} Proposition} \ref{expansion}.

\begin{proof}
For $n\geq1$, using \eqref{four-2},
we obtain
 $$\widehat{\sqrt{\mu}\varphi_{n,0,0}}(0)=0.$$
By using \eqref{four-1} for \eqref{F G phi phi}, {\color{black}recall} that $\kappa=\xi/|\xi|$, $Y^0_0=1/\sqrt{4\pi}$,  one gets
\begin{align*}
&\mathcal{F}
(Q(\sqrt{\mu}\varphi_{n,0,0},\sqrt{\mu}\varphi_{\tilde{n},
\tilde{l},\tilde{m}}))(\xi)\nonumber
= A_{\tilde{n},\tilde{l}}
A_{n,0}e^{-\frac{|\xi|^2}{2}}
\left(\frac{|\xi|}{\sqrt{2}}\right)^{2(n+\tilde{n})+\tilde{l}}
\\
&\quad\times
{\color{black} \frac{1}{\sqrt{4\pi}}}\int_{\mathbb{S}^2}b(\kappa\cdot\sigma)
\Big(\frac{1-\kappa\cdot\sigma}{2}\Big)^{n}
\Big(\frac{|\kappa+\sigma|}{2}\Big)^{2\tilde{n}+\tilde{l}}
Y^{\tilde{m}}_{\tilde{l}}
\Big(\frac{\kappa+\sigma}{|\kappa+\sigma|}\Big)
d\sigma.  \nonumber
\end{align*}
We then apply the identity \eqref{funk 2} of corollary \ref{coroll int f Y+}
and again \eqref{four-1} and derive
\begin{align*}
&\mathcal{F}(Q(\sqrt{\mu}\varphi_{n,0,0},\sqrt{\mu}\varphi_{\tilde{n},
\tilde{l},\tilde{m}}))(\xi)\nonumber\\
&={\color{black} \frac{1}{\sqrt{4\pi}}}\frac{A_{\tilde{n},\tilde{l}}A_{n,0}}{A_{n+\tilde{n},
\tilde{l}}}\Big[\int_{|\theta|\leq\frac{\pi}{4}}
\beta(\theta)\Big((\sin\theta)^{2n}(\cos\theta)^{2\tilde{n}
+\tilde{l}}P_{\tilde{l}}(\cos\theta)\Big)d\theta\Big] \widehat{\sqrt{\mu}\varphi_{\tilde{n}+n,\tilde{l},\tilde{m}}}(\xi).\nonumber
\end{align*}
We obtain that by the inverse Fourier transform
\begin{align*}
{\bf \Gamma}(\varphi_{n,0,0},\varphi_{\tilde{n},
\tilde{l},\tilde{m}})&=\frac{1}{\sqrt{\mu}}Q(\sqrt{\mu}\varphi_{n,0,0},\sqrt{\mu}\varphi_{\tilde{n},
\tilde{l},\tilde{m}})\\
&={\color{black} \frac{1}{\sqrt{4\pi}}}\frac{A_{\tilde{n},\tilde{l}}A_{n,0}}{A_{n+\tilde{n},
\tilde{l}}}\Big[\int_{|\theta|\leq\frac{\pi}{4}}
\beta(\theta)\Big((\sin\theta)^{2n}(\cos\theta)^{2\tilde{n}
+\tilde{l}}P_{\tilde{l}}(\cos\theta)\Big)d\theta\Big] \varphi_{\tilde{n}+n,\tilde{l},\tilde{m}}.
\end{align*}
Thus $(i)$ follows. Analogously, $(ii)$ holds true.
This ends the proof of Proposition \ref{expansion n00 nlm}.
\end{proof}

\begin{proposition}\label{expansion nlm nlm}
The following algebraic identities hold
for $l\geq1,\,\tilde{l}\geq1$ :
\begin{align}
&{\bf \Gamma}(\varphi_{n,l,m},\varphi_{\tilde{n},\tilde{l},\tilde{m}})\nonumber\\
&=
\sum^{k_0(l,\tilde{l},m,\tilde{m})}_{k=0}
\frac{A_{\tilde{n},\tilde{l}}A_{n,l}}{A_{n+\tilde{n}+k,l+\tilde{l}-2k}}
\left(\int_{\mathbb{S}^2_{\kappa}}G^{m,\tilde{m}}_{n,\tilde{n},l,\tilde{l}}(\kappa)
\overline{Y^{m+\tilde{m}}_{l+\tilde{l}-2k}}
(\kappa)d\kappa\right)\varphi_{n+\tilde{n}+k,l+\tilde{l}-2k,m+\tilde{m}},\,\nonumber
\end{align}
where $k_0(l,\tilde{l},m,\tilde{m})$ is given in \eqref{k_0} and $G^{m,\tilde{m}}_{n,\tilde{n},l,\tilde{l}}$ is defined by
\begin{align}\label{Gamma1}
\begin{split}
G^{m,\tilde{m}}_{n,\tilde{n},l,\tilde{l}}(\kappa)&=
\int_{\mathbb{S}^2}b(\kappa\cdot\sigma)\Big(|\kappa-\sigma|/2\Big)^{2n+l}
\Big(|\kappa+\sigma|/2\Big)^{2\tilde{n}+
\tilde{l}}\\
&\quad\qquad\times
Y^{m}_l\Big(\frac{\kappa-\sigma}
{|\kappa-\sigma|}\Big) Y^{\tilde{m}}_{\tilde{l}}\Big(\frac{\kappa+\sigma}{|\kappa+\sigma|}\Big)\, d\sigma.
\end{split}
\end{align}
\end{proposition}
This is exactly $(iii)$ of {\color{black} Proposition} \ref{expansion}.

\begin{proof}
Now we consider the case when $l\geq1\,\text{and}\,\tilde{l}\geq1$.\,\,
Since $\widehat{\sqrt{\mu}\varphi_{n,l,m}}(0)=0$,
we get from \eqref{four-1}-\eqref{F G phi phi}
\begin{align}
&\mathcal{F}(Q(\sqrt{\mu}\varphi_{n,l,m},
  \sqrt{\mu}\varphi_{\tilde{n},\tilde{l},\tilde{m}}))(\xi)
=A_{\tilde{n},\tilde{l}}A_{n,l}
e^{-\frac{|\xi|^2}{2}}
\left(\frac{|\xi|}{\sqrt{2}}\right)^{2(n+\tilde{n})+\tilde{l}+l}
\nonumber\\
&\qquad\times\,
\int_{\mathbb{S}^2}b(\kappa\cdot\sigma)
  Y^{m}_l\left(\frac{\kappa-\sigma}
  {|\kappa-\sigma|}\right) Y^{\tilde{m}}_{\tilde{l}}
  \left(\frac{\kappa+\sigma}{|\kappa+\sigma|}\right)
  \left(\frac{|\kappa-\sigma|}{2}\right)^{2n+l}
  \left(\frac{|\kappa+\sigma|}{2}\right)^{2\tilde{n}+\tilde{l}}\,
d\sigma\nonumber\\
&=A_{\tilde{n},\tilde{l}}A_{n,l}e^{-\frac{|\xi|^2}{2}}
\Big(\frac{|\xi|}{\sqrt{2}}\Big)^{2(n+\tilde{n})+\tilde{l}+l}
G^{m,\tilde{m}}_{\tilde{n},n,\tilde{l},l}(\kappa)\nonumber.
\end{align}
From {\color{black} Lemma} \ref{lemma:5.1},
$G^{m,\tilde{m}}_{n,\tilde{n},l,\tilde{l}}(\kappa)$
can be decomposed as a finite Laplace series
\begin{align*}
G^{m,\tilde{m}}_{n,\tilde{n},l,\tilde{l}}(\kappa)
&=\sum^{k_0(l,\tilde{l},m,\tilde{m})}_{k=0}
\left(\int_{\mathbb{S}^2_{\kappa}}G^{m,\tilde{m}}_{n,\tilde{n},l,\tilde{l}}(\kappa)
\overline{Y^{m+\tilde{m}}_{l+\tilde{l}-2k}}d\kappa\right)Y^{m+\tilde{m}}_{l+\tilde{l}-2k}(\kappa),
\end{align*}
where $k_0(l,\tilde{l},m,\tilde{m})$ was given in \eqref{k_0}.
By using this expansion, we derive
\begin{align}
\mathcal{F}&(Q(\sqrt{\mu}\varphi_{n,l,m},\sqrt{\mu}\varphi_{\tilde{n},\tilde{l},\tilde{m}}))(\xi)\nonumber\\
&=
\sum^{k_0(l,\tilde{l},m,\tilde{m})}_{k=0}
\frac{A_{\tilde{n},\tilde{l}}A_{n,l}}{A_{n+\tilde{n}+k,l+\tilde{l}-2k}}
\left(\int_{\mathbb{S}^2_{\kappa}}G^{m,\tilde{m}}_{n,\tilde{n},l,\tilde{l}}(\kappa)
\overline{Y^{m+\tilde{m}}_{l+\tilde{l}-2k}}
(\kappa)d\kappa\right)
\mathcal{F}(\sqrt{\mu}\varphi_{n+\tilde{n}+k,l+\tilde{l}-2k,m+\tilde{m}})\nonumber
\end{align}
and we conclude by taking the inverse Fourier transform.
This ends the proof of Proposition \ref{expansion nlm nlm}.
\end{proof}

\subsection{The proof of {\color{black} Proposition} \ref{expansion2}}

We now prove the following identity of Proposition \ref{expansion2}
\begin{equation*}
\sum_{|m|\leq\,l}\sum_{|\tilde{m}|\leq\,\tilde{l}}
  \mu^{m,\tilde{m},m_1^{\prime}}_{n,\tilde{n},l,\tilde{l},k_1}
  \overline{\mu^{m,\tilde{m},m_2^{\prime}}_{n,\tilde{n},l,\tilde{l},k_2}}
=
\sum_{|m|\leq\,l}\sum_{|\tilde{m}|\leq\,\tilde{l}}
\Big|
  \mu^{m,\tilde{m},m_1^{\prime}}_{n,\tilde{n},l,\tilde{l},k_1}
\Big|^2\delta_{k_1,k_2}\delta_{m'_1,m'_2}.
\end{equation*}
We state it in the following proposition with the notations
$G^{m,\tilde{m}}_{n,\tilde{n},l,\tilde{l}}(\kappa)$
given in \eqref{Gamma1}, since
from {\color{black} Proposition} \ref{expansion}, we have, for $|m'|\leq l+\tilde{l}-2k$
\begin{align*}
\mu^{m,\tilde{m},m^{\prime}}_{n,\tilde{n},l,\tilde{l},k}
&=
(-1)^k\Big(\frac{2\pi^{\frac{3}{2}}(n+\tilde{n}+k)!\Gamma(n+\tilde{n}+l+\tilde{l}-k+\frac{3}{2})}
{\tilde{n}!\Gamma(\tilde{n}+\tilde{l}+\frac{3}{2})n!\Gamma(n+l+\frac{3}{2})}\Big)^{\frac{1}{2}}
\\
&\qquad\times
\int_{\mathbb{S}^2_{\kappa}}
 G^{m,\tilde{m}}_{n,\tilde{n},l,\tilde{l}}(\kappa)
\overline{Y^{m^{\prime}}_{l+\tilde{l}-2k}}
(\kappa)d\kappa.
\end{align*}

\begin{proposition}\label{orth}
For $G^{m,\tilde{m}}_{n,\tilde{n},l,\tilde{l}}(\kappa)$
given in \eqref{Gamma1} and any integers
$n,\tilde{n} \geq0$,
{\color{black}$|m^{\prime}|\leq l^{\prime}$},
$|m^{\star}|\leq l^{\star}$,
we have
\begin{align*}
\sum_{|m|\leq l}\sum_{|\tilde{m}|\leq \tilde{l}}
&\left(
\int_{\mathbb{S}^2_{\kappa}}
  G^{m,\tilde{m}}_{n,\tilde{n},l,\tilde{l}}(\kappa)
  \overline{Y^{m^{\prime}}_{l^{\prime}}}(\kappa)
d\kappa\right)
\left(
\int_{\mathbb{S}^2_{\kappa}}
  \overline{G^{m,\tilde{m}}_{n,\tilde{n},l,\tilde{l}}(\kappa)}
  Y^{m^{\star}}_{l^{\star}}
  (\kappa)d\kappa\right)\nonumber\\
&=
\sum_{|m|\leq l}
\sum_{|\tilde{m}|\leq \tilde{l}}
\left|\int_{\mathbb{S}^2_{\kappa}}G^{m,\tilde{m}}_{n,\tilde{n},l,\tilde{l}}(\kappa)
\overline{Y^{m^{\prime}}_{l^{\prime}}}
(\kappa)d\kappa
\right|^2
\delta_{l^{\prime},l^{\star}}\delta_{m^{\prime},m^{\star}}.
\end{align*}
\end{proposition}

\begin{proof}
We recall the definition \eqref{Gamma1} of
$G^{m,\tilde{m}}_{n,\tilde{n},l,\tilde{l}}
  (\kappa)$
\begin{align*}
G^{m,\tilde{m}}_{n,\tilde{n},l,\tilde{l}}(\kappa)
&=\int_{\mathbb{S}^2}
b(\kappa\cdot\sigma)
\left(|\kappa-\sigma|/2\right)^{2n+l}
\left(|\kappa+\sigma|/2\right)^{2\tilde{n}+
\tilde{l}}\nonumber\\
&\quad\qquad\times
Y^{m}_l
  \left(\frac{\kappa-\sigma}{|\kappa-\sigma|}\right)
Y^{\tilde{m}}_{\tilde{l}}
  \left(\frac{\kappa+\sigma}{|\kappa+\sigma|}\right)\,
d\sigma
\end{align*}
and we consider the transform \eqref{k,k1,k2}-\eqref{tran}
for an unit vector $\sigma$
\begin{align*}
\sigma
&= \kappa\cos\theta + \kappa^1\sin\theta\cos\phi
 + \kappa^2\sin\theta\sin\phi
\\
&= \kappa\cos\theta + \kappa^{\bot}(\phi) \, \sin\theta
\end{align*}
with $\theta\in[0,\frac{\pi}{2}]$ and $\phi\in[0,2\pi]$.\,\,
Therefore, using \eqref{k+s}-\eqref{k-s}, the change of variable
$\theta = 2\, \theta_1$, the odd-even parity of $P_l^m$
and the definition  \eqref{beta} of $\beta$ we find
\begin{align}\label{G=}
\begin{split}
G&^{m,\tilde{m}}_{n,\tilde{n},l,\tilde{l}}(\kappa)
=
\int_{|\theta_1|\leq\frac{\pi}{4}}
  \beta(\theta_1)(\sin\theta_1)^{2n+l}
  (\cos\theta_1)^{2\tilde{n}+\tilde{l}} \quad\times
\\
&\quad
\int^{2\pi}_{0}
  Y^m_l(\kappa\sin\theta_1-\kappa^{\bot}(\phi_1)\cos\theta_1)
  Y^{\tilde{m}}_{\tilde{l}}
  (\kappa\cos\theta_1+\kappa^{\bot}(\phi_1)\sin\theta_1)
\frac{d\phi_1}{2\pi}
d\theta_1
\end{split}
\end{align}
and
\begin{align}\label{int GY=}
&\int_{\mathbb{S}^2_{\kappa}}
  G^{m,\tilde{m}}_{n,\tilde{n},l,\tilde{l}}
  (\kappa)\overline{Y^{m^{\prime}}_{l'}}(\kappa)
d\kappa\nonumber\\
&\quad\quad=
\int_{|\theta_1|\leq\frac{\pi}{4}}
  \beta(\theta_1)(\sin\theta_1)^{2n+l}
  (\cos\theta_1)^{2\tilde{n}+\tilde{l}}
  \quad\times\int^{2\pi}_{0}
\\
&\quad\Big(\int_{\mathbb{S}^2_\kappa} Y^m_l(\kappa\sin\theta_1-\kappa^{\bot}(\phi_1)\cos\theta_1)
   Y^{\tilde{m}}_{\tilde{l}}
  (\kappa\cos\theta_1+\kappa^{\bot}(\phi_1)\sin\theta_1)
  \overline{Y^{m^{\prime}}_{l'}}(\kappa)
d\kappa
\Big)
\frac{d\phi_1}{2\pi}
d\theta_1.\nonumber
\end{align}
Using equivalent notations of \eqref{k,k1,k2}-\eqref{tran}
{\color{black}
for an unit vector $\sigma$ expanded in
another orthonormal frame $(\gamma,\gamma^1,\gamma^2)$,
\begin{align*}
\sigma
&=\gamma   \cos2\theta_2 +
  \gamma^1 \sin2\theta_2 \cos\phi_2 + \gamma^2\sin 2\theta_2\sin\phi_2
\\
&=\gamma \cos2\theta_2 + \gamma^{\bot}(\phi_2) \sin2\theta_2
\end{align*}
with $\theta_2\in[0,\pi/4], \phi_2\in[0,2\pi]$, }
we have also
\begin{align*}
&  \int_{\mathbb{S}^2_{\gamma}}
    \overline{G^{m,\tilde{m}}_{n,\tilde{n},l,\tilde{l}}(\gamma)}
    Y^{m^{\star}}_{l^{\star}}(\gamma)
  d\gamma\\
&=
\int_{|\theta_2|\leq\frac{\pi}{4}}
  \beta(\theta_2)(\sin\theta_2)^{2n+l}
  (\cos\theta_2)^{2\tilde{n}+\tilde{l}} \quad\times
\int^{2\pi}_{0}
\int_{\mathbb{S}^2_\gamma}
\\
&
\,\,\,\,  \overline{Y^m_l}(\gamma\sin\theta_2-\gamma^{\bot}(\phi_2)\cos\theta_2)
  \overline{Y^{\tilde{m}}_{\tilde{l}}}
  (\gamma\cos\theta_2+\gamma^{\bot}(\phi_1)\sin\theta_2)
  Y^{m^{\star}}_{l^{\star}}(\gamma)
d\gamma
\frac{d\phi_2}{2\pi}
d\theta_2.
\end{align*}
The goal of this proposition is to compute the following term :
\begin{align*}
{\bf A} = \sum_{|m|\leq l}
\sum_{|\tilde{m}|\leq \tilde{l}}
\left(
  \int_{\mathbb{S}^2_{\kappa}}
    G^{m,\tilde{m}}_{n,\tilde{n},l,\tilde{l}}(\kappa)
    \overline{Y^{m^{\prime}}_{l^{\prime}}}(\kappa)
  d\kappa
\right)
\left(
  \int_{\mathbb{S}^2_{\gamma}}
    \overline{G^{m,\tilde{m}}_{n,\tilde{n},l,\tilde{l}}(\gamma)}
    Y^{m^{\star}}_{l^{\star}}(\gamma)
  d\gamma
\right)
\end{align*}
From the addition theorem \eqref{addition theorem Y},
we find
\begin{align*}
&\sum_{|m|\leq l}
  Y^{m}_l
    (\kappa\sin\theta_1-\kappa^{\bot}(\phi_1)\cos\theta_1)
  \overline{Y^{m}_l}
    (\gamma\sin\theta_2-\gamma^{\bot}(\phi_2)\cos\theta_2)\\
&=\frac{2l+1}{4\pi}
P_l(
  (\kappa\sin\theta_1-\kappa^{\bot}(\phi_1)\cos\theta_1)\cdot
  (\gamma\sin\theta_2-\gamma^{\bot}(\phi_2)\cos\theta_2))
\end{align*}
and
{\color{black}\begin{align*}
&\sum_{|\tilde{m}|\leq {\tilde{l}}}
  Y^{\tilde{m}}_{\tilde{l}}
    (\kappa\cos\theta_1+\kappa^{\bot}(\phi_1)\sin\theta_1)
  \overline{Y^{\tilde{m}}_{\tilde{l}}}
    (\gamma\cos\theta_2+\gamma^{\bot}(\phi_2)\sin\theta_2)\\
&=\frac{2\tilde{l}+1}{4\pi}
P_{\tilde{l}}(
  (\kappa\cos\theta_1+\kappa^{\bot}(\phi_1)\sin\theta_1)\cdot
  (\gamma\cos\theta_2+\gamma^{\bot}(\phi_2)\sin\theta_2)
).
\end{align*}
}
We then plug the two previous identities into the expression of ${\bf A}$ and
we directly derive
\begin{align}\label{A=SSGY}
\begin{split}
{\bf A} &=
\frac{2l+1}{4\pi}\frac{2\tilde{l}+1}{4\pi}
\int_{|\theta_1|\leq\frac{\pi}{4}}
  \left(\beta(\theta_1)(\sin\theta_1)^{2n+l}
  (\cos\theta_1)^{2\tilde{n}+\tilde{l}}\right)
\\
&\quad\times
\int_{|\theta_2|\leq\frac{\pi}{4}}
\left(\beta(\theta_2)(\sin\theta_2)^{2n+l}
  (\cos\theta_2)^{2\tilde{n}+\tilde{l}}
\right)
{\bf B_1}(\theta_1,\theta_2)
d\theta_1 \, d\theta_2
\end{split}
\end{align}
where
\begin{align}\label{B1}
{\bf B_1}(\theta_1,\theta_2) =
  \int_{\mathbb{S}^2_{\gamma}}\int_{\mathbb{S}^2_{\kappa}}
  \int^{2\pi}_{0}
    {\bf B_2}(\kappa,\gamma,\theta_1,\theta_2,\phi_2)
  \frac{d\phi_2}{2\pi}  \,
  {\color{black}\overline{Y^{m^{\prime}}_{l^{\prime}}}(\kappa)Y^{m^{\star}}_{l^{\star}}(\gamma)}
  d\kappa d\gamma
\end{align}
and
\begin{align*}
&{\bf B_2}(\kappa,\gamma,\theta_1,\theta_2,\phi_2)
= \int^{2\pi}_{0}
P_l\left(
  (\kappa\sin\theta_1-\kappa^{\bot}(\phi_1)\cos\theta_1)\cdot
  \gamma^-
\right)\\
&\quad\times
P_{\tilde{l}}\left(
  (\kappa\cos\theta_1+\kappa^{\bot}(\phi_1)\sin\theta_1)\cdot
  \gamma^+
\right)
\frac{d\phi_1}{2\pi}.
\end{align*}
Here, $\gamma^+$ and $\gamma^-$ are defined by
(and depend on $\gamma$, $\theta_2$, $\phi_2$)
\begin{align}\label{g+,g-}
\gamma^+ = \gamma\cos\theta_2
+ \gamma^{\bot}(\phi_2)\sin\theta_2,\quad
  \gamma^- = \gamma\sin\theta_2
- \gamma^{\bot}(\phi_2)\cos\theta_2.
\end{align}
From {\color{black} lemma} \ref{lem int P+P-}
(proved after the Proposition), we have
\begin{align*}
{\bf B_2}(\kappa,\gamma,\theta_1,\theta_2,\phi_2)
&=\sum_{0\leq q\leq l} \, \sum_{0\leq \tilde{q} \leq\tilde{l}}
  b^{q,\tilde{q}}_{l,\tilde{l}}(\theta_1)
  P_{q}(\kappa\cdot\gamma^-)
  P_{\tilde{q}}(\kappa\cdot\gamma^+)
\end{align*}
where $b^{q,\tilde{q}}_{l,\tilde{l}}(\theta_1)$
is a continuous function dependent on $\theta_1$.

Therefore from \eqref{B1} we deduce
\begin{align}\label{sph}
{\bf B_1}(\theta_1,\theta_2)
&=\sum_{0\leq q \leq l} \, \sum_{0\leq \tilde{q} \leq\tilde{l}}
  b^{q,\tilde{q}}_{l,\tilde{l}}(\theta_1) \,
  {\bf B}_{q,\tilde{q}}(\theta_2)
\end{align}
where
\begin{align}\label{B_qq}
{\bf B}_{q,\tilde{q}}(\theta_2) =
  \int_{\mathbb{S}^2_{\gamma}}\int_{\mathbb{S}^2_{\kappa}}
  \left(\int^{2\pi}_{0}
    P_{q}(\kappa\cdot\gamma^-) \,
    P_{\tilde{q}}(\kappa\cdot\gamma^+) \,
  \frac{d\phi_2}{2\pi} \right)\,
   {\color{black}\overline{Y^{m^{\prime}}_{l^{\prime}}}(\kappa)Y^{m^{\star}}_{l^{\star}}(\gamma)} \,
  d\kappa d\gamma.
\end{align}
{\color{black}
Since $(\gamma,\gamma^1,\gamma^2)$ is an orthonormal basis in $\mathbb{R}^3$, then for any unit vector $\kappa$,
we can find $\theta\in[0,\pi]$
and $\phi\in[0,2\pi]$ such that
$$\kappa\cdot\gamma=\cos\theta; \kappa\cdot\gamma^1=\sin\theta\cos\phi; \kappa\cdot\gamma^2=\sin\theta\sin\phi.$$
Therefore (recall that $\gamma^\pm$ defined in \eqref{g+,g-}
depend on $\gamma$, $\theta_2$, $\phi_2$),
\begin{align*}
\kappa\cdot\gamma^-&=(\kappa\cdot\gamma)\sin\theta_2-\kappa\cdot\gamma^1\cos\theta_2\cos\phi_2-\kappa\cdot\gamma^2\cos\theta_2\sin\phi_2\\
&=\cos\theta\sin\theta_2+\sin\theta\cos\theta_2\cos(\phi-\phi_2-\pi);\\
\kappa\cdot\gamma^+&=(\kappa\cdot\gamma)\cos\theta_2+\kappa\cdot\gamma^1\sin\theta_2\cos\phi_2+\kappa\cdot\gamma^2\sin\theta_2\sin\phi_2\\
&=\cos\theta\cos\theta_2+\sin\theta\sin\theta_2\cos(\phi-\phi_2)
\end{align*}
We use again two times the addition theorem \eqref{addition theorem}:
\begin{align*}
P_{q}(\kappa\cdot\gamma^-)&=\sum_{|m|\leq q} a_{q,m}P^{|m|}_q(\cos\theta)P^{|m|}_q(\sin\theta_2)e^{im(\phi-\phi_2-\pi)}\\
&=\sum_{|m|\leq q} a_{q,m}P^{|m|}_q(\kappa\cdot\gamma)P^{|m|}_q(\sin\theta_2)e^{im(\phi-\phi_2-\pi)}\\
P_{\tilde{q}}(\kappa\cdot\gamma^+)&=\sum_{|\tilde{m}|\leq \tilde{q}}a_{\tilde{q},\tilde{m}}P^{|\tilde{m}|}_{\tilde{q}}(\kappa\cdot\gamma){\color{black}P^{|\tilde{m}|}_{\tilde{q}}(\cos\theta_2)}e^{i\tilde{m}(\phi-\phi_2)},
\end{align*}
where $a_{q,m}$, $a_{\tilde{q},\tilde{m}}$ was defined in the beginning of this section.
Then
\begin{align*}
&\int^{2\pi}_{0}
  P_{q}(\kappa\cdot\gamma^-(\phi_2))
  P_{\tilde{q}}(\kappa\cdot\gamma^+(\phi_2))
\frac{d\phi_2}{2\pi}
 \\
& =\quad
\sum_{|k|\leq\min(q,\tilde{q})}
  a_{q,k}
  a_{\tilde{q},k}
  P^{|k|}_{q}(\kappa\cdot\gamma)
  P^{|k|}_{q}(\sin\theta_2)(-1)^k
  P^{|k|}_{\tilde{q}}(\kappa\cdot\gamma)
  P^{|k|}_{\tilde{q}}(\cos\theta_2).
\end{align*}
}
Since
$P^{|k|}_{q}(\kappa\cdot\gamma)
P^{|k|}_{\tilde{q}}(\kappa\cdot\gamma)$
is a continuous function of $\kappa\cdot\gamma$, we apply
the Funk-Hecke Formula \eqref{funk-hecke} and obtain
{\color{black}\[
\int_{\mathbb{S}^2}P^{|k|}_{q}
  (\kappa\cdot\gamma) P^{|k|}_{\tilde{q}}
  (\kappa\cdot\gamma) Y^{m^{\star}}_{l^{\star}}(\gamma)
d\gamma
=
\left( 2\pi \, \int^1_{-1}P^{|k|}_{q}(x)
  P^{|k|}_{\tilde{q}}(x) P_{l^{\star}}(x)dx
\right)
Y^{m^{\star}}_{l^{\star}}(\kappa).
\]}
Combining the two previous relations into \eqref{B_qq},
we obtain
\begin{align}\label{B_qq 2}
{\bf B}_{q,\tilde{q}}(\theta_2)
&=
\sum_{|k|\leq\min(q,\tilde{q})}
  a_{q,k}
  a_{\tilde{q},k}
  P^{|k|}_{q}(\sin\theta_2)(-1)^k
  P^{|k|}_{\tilde{q}}(\cos\theta_2)
\nonumber\\
&\quad\times\left(2\pi \int^1_{-1}P^{|k|}_{q}(x)
  P^{|k|}_{\tilde{q}}(x) P_{l^{\star}}(x)dx
\right)
{\color{black}\int_{\mathbb{S}^2_{\kappa}}
  Y^{m^{\star}}_{l^{\star}}(\kappa)
  \overline{Y^{m^{\prime}}_{l^{\prime}}}(\kappa)
  d\kappa}.
\end{align}
{\color{black}Finally}, if $(l^{\star},m^{\star}) \neq (l^{\prime},m^{\prime})$,
the orthogonality of the spherical harmonics implies that
${\bf B}_{q,\tilde{q}}=0$ for all $q$ and $\tilde{q}$, and so on for ${\bf B_1}$ and ${\bf A}$.
This concludes the proof of {\color{black} Proposition} \ref{orth}.
\end{proof}

\begin{remark}\label{rem sum |GY|}
From the previous proof, in the special case
$(l^{\star},m^{\star}) = (l^{\prime},m^{\prime})$,
we have from \eqref{A=SSGY}, \eqref{sph} and \eqref{B_qq 2}
\begin{align*}
&\sum_{|m|\leq l}
\sum_{|\tilde{m}|\leq \tilde{l}}
\left|\int_{\mathbb{S}^2_{\kappa}}G^{m,\tilde{m}}_{n,\tilde{n},l,\tilde{l}}(\kappa)
\overline{Y^{m^{\prime}}_{l^{\prime}}}
(\kappa)d\kappa
\right|^2
\\
&=
\frac{2l+1}{4\pi}\frac{2\tilde{l}+1}{4\pi}
\int_{|\theta_1|\leq\frac{\pi}{4}}
  \left(\beta(\theta_1)(\sin\theta_1)^{2n+l}
  (\cos\theta_1)^{2\tilde{n}+\tilde{l}}\right)
\\
&\quad\times
\int_{|\theta_2|\leq\frac{\pi}{4}}
\left(\beta(\theta_2)(\sin\theta_2)^{2n+l}
  (\cos\theta_2)^{2\tilde{n}+\tilde{l}}
\right)
B_{l,\tilde{l},l^{\prime}}(\theta_1,\theta_2)
d\theta_1 \, d\theta_2
\end{align*}
where $B_{l,\tilde{l},l^{\prime}}(\theta_1,\theta_2)$ is defined by
\begin{align*}
&B_{l,\tilde{l},l^{\prime}}(\theta_1, \theta_2)
=
\sum_{0\leq q\leq l}\sum_{0\leq \tilde{q}\leq\tilde{l}}
b^{q,\tilde{q}}_{l,\tilde{l}}(\theta_1)
\quad \times
\\
&\sum_{|k|\leq \min(q,\tilde{q})}
a_{q,k}
a_{\tilde{q},k}
P^{|k|}_{q}(\sin\theta_2)\ (-1)^k
P^{|k|}_{\tilde{q}}(\cos\theta_2)\,
\left(2\pi
\int^1_{-1}
P^{|k|}_{q}(x)\,
P^{|k|}_{\tilde{q}}(x)\,
P_{l^{\prime}}(x) \,
dx\right).
\end{align*}
\end{remark}

We now prove the following lemma used in the previous proposition \ref{orth}.

\begin{lemma}\label{lem int P+P-}
For any $l$, $\tilde{l}$, $q$, $\tilde{q} \in \mathbb{N}$,
there exists a continuous function
$b^{q,\tilde{q}}_{l,\tilde{l}}(\theta)$ such that
for any real $\theta$, $\phi$ and unit vectors
$\kappa$, $\eta^+$, $\eta^-$, {\color{black} $(\eta^+,\eta^-)$ being orthogonal}, we have
\begin{align*}
\int^{2\pi}_{0}
&P_l\left(
  (\kappa\sin\theta-\kappa^{\bot}(\phi)\cos\theta)\cdot
  \eta^-
\right)\\
&\quad\times
P_{\tilde{l}}\left(
  (\kappa\cos\theta+\kappa^{\bot}(\phi)\sin\theta)\cdot
  \eta^+
\right)
\frac{d\phi}{2\pi}
\\
&=\sum_{0\leq q\leq l} \, \sum_{0\leq \tilde{q} \leq\tilde{l}}
  b^{q,\tilde{q}}_{l,\tilde{l}}(\theta) \,
  P_{q}(\kappa\cdot\eta^-)
  P_{\tilde{q}}(\kappa\cdot\eta^+)
\end{align*}
where the coefficients {\color{black} $b_{l,\tilde{l}}^{q,\tilde{q}} \equiv 0$}
if $(q,\tilde{q}) \neq (l-2q_1,\tilde{l}-2q_2)$
for integers $q_1$, $q_2$ satisfying $0\leq 2q_1\leq l$, $0\leq 2q_2\leq \tilde{l}$ .
\end{lemma}

\begin{proof}
We apply the addition theorem \eqref{addition theorem}
in the frame $(\kappa,\kappa^1,\kappa^2)$
and the relation \eqref{Ylm2}.
\begin{align*}
P_{\tilde{l}}\left(
  (\kappa\cos\theta+\kappa^{\bot}(\phi)\sin\theta)\cdot
  \eta^+
\right)
&=
\sum_{|{\tilde{k}}|\leq {\tilde{l}}}
  a_{\tilde{l},\tilde{k}}
  P^{|\tilde{k}|}_{\tilde{l}}(\cos\theta)
  \left(\frac{\mathrm{d}^{|\tilde{k}|}
  P_{\tilde{l}}}{\mathrm{d}x^{|\tilde{k}|}}
  \right)(\kappa\cdot\eta^+) \,
  e^{i\,\tilde{k}\,\phi} \, U^+_{\tilde{k}} ,
\\
P_l\left(
  (\kappa\sin\theta-\kappa^{\bot}(\phi)\cos\theta)
  \cdot\eta^-
\right)
&=
P_l\left(
  \left(\kappa\sin(\theta)
  +\kappa^{\bot}\left(\phi+\pi\right)
  \cos(\theta)\right)\cdot
  \eta^-
\right)
\\
&=
\sum_{|k|\leq l}
  a_{l,k}
  P^{|k|}_{l}(\sin\theta)
  \left({\color{black}\frac{\mathrm{d}^{|k|}
  P_l}{\mathrm{d}x^{|k|}}
  }\right)(\kappa\cdot\eta^-) \,
  e^{i\,k\,(\phi+\pi)} \, U^-_{k}
\end{align*}
where
\begin{align*}
U_k^+ = \left((\kappa^1\cdot\eta^+)
 - i\operatorname{sgn}(k) (\kappa^2\cdot\eta^+)\right)^{|k|},
\,\,
U_k^- = \left((\kappa^1\cdot\eta^-)
 - i\operatorname{sgn}(k) (\kappa^2\cdot\eta^-)\right)^{|k|}.
\end{align*}
{\color{black} Since the integrals are zero when $\tilde{k}\not=-k$,
we derive}
\begin{align*}
{\bf M}
&= \int^{2\pi}_{0}
P_l\left(
  (\kappa\sin\theta-\kappa^{\bot}(\phi)\cos\theta)\cdot
  \eta^-
\right) \,
P_{\tilde{l}}\left(
  (\kappa\cos\theta+\kappa^{\bot}(\phi)\sin\theta)\cdot
  \eta^+
\right)
{\color{black}\frac{d\phi}{2\pi}}
\\
&=\sum_{|k|\leq\min(l,\tilde{l})}
  c_1^k(\theta)
  \left({\color{black}\frac{\mathrm{d}^{|k|}
  P_l}{\mathrm{d}x^{|k|}}
  }\right)(\kappa\cdot\eta^-)
  \left({\color{black}\frac{\mathrm{d}^{|k|}
  P_{\tilde{l}}}{\mathrm{d}x^{|k|}}
  }\right)(
    \kappa\cdot\eta^+)\nonumber
 (-1)^k U^-_{-k}  U^+_{k}
\end{align*}
where
{\color{black}
\begin{align*}
c_1^k(\theta) =
  a_{l,k} a_{\tilde{l},-k}
  P^{|k|}_l(\sin\theta)
  P^{|k|}_{\tilde{l}}(\cos\theta).
\end{align*}
}
We then write
\begin{align}\label{M}
{\bf M}
=\sum_{0 \leq k \leq\min(l,\tilde{l})}
  c_1^k(\theta)
   \left({\color{black}\frac{\mathrm{d}^k
  P_l}{\mathrm{d}x^{k}}}\right)(\kappa\cdot\eta^-)
  \left({\color{black}\frac{\mathrm{d}^kP_{\tilde{l}}}{\mathrm{d}x^k}}\right)
  (\kappa\cdot\eta^+)
 V_k
\end{align}
where
$V_0=1$ and $V_k$ is defined for $k\geq1$ by
\begin{align*}
V_k =
&\bigg\{
  \left[
  (\kappa^1\cdot\eta^+
  -i\kappa^2\cdot\eta^+)
  (-\kappa^1\cdot\eta^-
  -i\kappa^2\cdot\eta^-)
   \right]^{k}
\\
&+\left[
  (\kappa^1\cdot\eta^+
  + i\kappa^2\cdot\eta^+)
  (-\kappa^1\cdot\eta^-
  + i\kappa^2\cdot\eta^-)
  \right]^{k}
\bigg\}.
\end{align*}
{\color{black}
We claim that $V_k$ is polynomial type :
$V_k=p_k\bigr((\kappa \cdot \eta^+),(\kappa \cdot \eta^-)\bigl)$
where
\begin{align*} 
p_k(x,y)&= \left(x\,y+i\sqrt{1-x^2-y^2}\right)^k
+ \left(x\,y-i\sqrt{1-x^2-y^2}\right)^k
\\
&= 2  \sum_{0\leq 2r \leq k}
(-1)^r \, \binom{k}{2r} \, x^{k-2r} \,y^{k-2r} \, (1 - x^2-y^2)^r.
\end{align*}
Indeed, we observe that, there exist two real numbers $A$, $B$ such that
\begin{align*}
 &(\kappa^1\cdot\eta^+-i\kappa^2\cdot\eta^+)
 (-\kappa^1\cdot\eta^--i\kappa^2\cdot\eta^-)=A+iB, \\
&(\kappa^1\cdot\eta^++ i\kappa^2\cdot\eta^+)
  (-\kappa^1\cdot\eta^-+ i\kappa^2\cdot\eta^-)=A-iB.
\end{align*}
If we can show that
\begin{align}
A &= (\kappa \cdot \eta^-) (\kappa \cdot \eta^+)
  \label{calcul 1}
\\
B^2
&= 1 - (\kappa\cdot\eta^-)^2 -(\kappa\cdot\eta^+)^2,
   \label{calcul 2}
\end{align}
then it follows that
\begin{align}\label{V_k}
V_k
&= \left[ A + i B \right]^{k} + \left[ A - i B \right]^{k}
\nonumber\\
&= 2  \sum_{0\leq 2r \leq k}
(-1)^r \, \binom{k}{2r} \, A^{k-2r} \, B^{2r}
\nonumber\\
&= 2  \sum_{0\leq 2r \leq k}
(-1)^r \, \binom{k}{2r} \, (\kappa\cdot\eta^+ )^{k-2r}\,
(\kappa \cdot\eta^-)^{k-2r} \,
( 1 - ( \kappa\cdot\eta^+)^2 - (\kappa\cdot\eta^- )^2 )^{r}.
\end{align}
Let us now prove \eqref{calcul 1}, \eqref{calcul 2}.
There exist reals $\phi^+$, $\phi^-$, $\theta^+$ and $\theta^-$ such that
\begin{align*}
\eta^+ &=\cos\theta^+ \kappa +
  \sin\theta^+ \cos\phi^+ \kappa^1 + \sin\theta^+ \sin\phi^+ \kappa^2
\\
\eta^- &=\cos\theta^- \kappa +
  \sin\theta^- \cos\phi^- \kappa^1 + \sin\theta^- \sin\phi^- \kappa^2.
\end{align*}
The orthogonality of $\eta^+$ and $\eta^-$ implies that
\begin{equation}\label{orthog eta}
\cos\theta^+ \cos\theta^- + \sin\theta^+ \sin\theta^- \cos(\phi^+ - \phi^-) = 0.
\end{equation}
We then have
\begin{align*}
&( \kappa^1 \cdot \eta^+ - i\kappa^2 \cdot \eta^+) \,
 (-\kappa^1 \cdot \eta^-  - i\kappa^2 \cdot \eta^-)
\\
&\quad=( \sin\theta^+ \cos\phi^+ - i \, \sin\theta^+ \sin\phi^+) \,
       (-\sin\theta^- \cos\phi^- - i \, \sin\theta^- \sin\phi^-)
\\
&\quad=A + i B.
\end{align*}
- Computation of $A$:
\begin{align*}
A&=
 - \sin\theta^+ \cos\phi^+ \sin\theta^- \cos\phi^-
 - \sin\theta^+ \sin\phi^+ \sin\theta^- \sin\phi^-
\\
&= - \sin\theta^+ \sin\theta^- \cos(\phi^+ - \phi^-)
\\
&= \cos\theta^+ \cos\theta^- \quad\quad\text{(from \eqref{orthog eta})}
\\
&= ( \kappa\cdot\eta^+)  \, (\kappa\cdot\eta^- ).
\end{align*}
- Computation of $B$:
\begin{align*}
B&=
 - \sin\theta^+ \cos\phi^+ \sin\theta^- \sin\phi^-
 {\color{black}+\sin\theta^+ \sin\phi^+ \sin\theta^- \cos\phi^-}
\\
&=  \sin\theta^+ \sin\theta^- \sin(\phi^+ - \phi^-).
\end{align*}
Using again \eqref{orthog eta}
\begin{align*}
B^2
&= \sin^2\theta^+ \sin^2\theta^- \sin^2(\phi^+ - \phi^-)
\\
&=  \sin^2\theta^+ \sin^2\theta^-
  - \sin^2\theta^+ \sin^2\theta^- \cos^2(\phi^+ - \phi^-)
\\
&= (1-\cos^2\theta^+) \, (1-\cos^2\theta^-) - \cos^2\theta^+ \, \cos^2\theta^-
\\
&= 1 - ( \kappa\cdot\eta^+)^2 - ( \kappa\cdot\eta^-)^2.
\end{align*}
This ends the proof of \eqref{calcul 1}, \eqref{calcul 2}.\\
From the expression of \eqref{V_k}, $V_k$ can be re-written as
\begin{align*}
V_k&=\sum_{0\leq2r_1,2r_2\leq k}a_{k,r_1,r_2}(\kappa \cdot \eta^-)^{k-2r_1}\,(\kappa \cdot \eta^+)^{k-2r_2}
\end{align*}
where $a_{k,r_1,r_2}$ are constants dependent on $k,r_1,r_2$.  Recall that $P_l(x)$ is a l-order polynomial of $x$ in Sec. 1 of Chap. III in \cite{San},
$$P_l(x)=\sum^{[\frac{l}{2}]}_{m=0}(-1)^m\frac{1}{2^l}\binom{l}{m}\binom{2l-2m}{l}x^{l-2m},$$
we obtain
\begin{align*}
&\frac{\mathrm{d}^k
  P_l}{\mathrm{d}x^k}(\kappa\cdot\eta^-)=\sum_{0\leq 2m_1\leq l-k} b_{l,k,m_1}(\kappa \cdot \eta^-)^{l-k-2m_1};
\\
&\frac{\mathrm{d}^kP_{\tilde{l}}}{\mathrm{d}x^k}(\kappa\cdot\eta^+)=\sum_{0\leq 2m_2\leq \tilde{l}-k} b_{\tilde{l},k,m_2}(\kappa \cdot \eta^+)^{\tilde{l}-k-2m_2},
\end{align*}
where $b_{l,k,m_1},b_{\tilde{l},k,m_2}$ are constants.
We plug the expression of $V_k, \frac{\mathrm{d}^k
  P_l}{\mathrm{d}x^k}, \frac{\mathrm{d}^kP_{\tilde{l}}}{\mathrm{d}x^k}$ into \eqref{M},
\begin{align*}
{\bf M}
&=\sum_{0 \leq k \leq\min(l,\tilde{l})}
  c_1^k(\theta)\sum_{0\leq2r_1,2r_2\leq k}a_{k,r_1,r_2}\\
&\qquad\times\sum_{0\leq 2m_1\leq l-k}\sum_{0\leq 2m_2\leq \tilde{l}-k}b_{l,k,m_1} b_{\tilde{l},k,m_2}
   (\kappa \cdot \eta^-)^{l-2r_1-2m_1}(\kappa \cdot \eta^+)^{\tilde{l}-2r_2-2m_2}.
\end{align*}
Exchanging the order of the summation of ${\bf M}$, one can verify that
\begin{align*}
{\bf M}
&=\sum_{0\leq 2j_1\leq l}\sum_{0\leq 2j_2\leq \tilde{l}}(\kappa \cdot \eta^-)^{l-2j_1}(\kappa \cdot \eta^+)^{\tilde{l}-2j_2}\\
&\times\Bigg[\sum_{0 \leq k \leq\min(l,\tilde{l})}\sum_{\substack{r_1+m_1=j_1\\0\leq2r_1\leq k,0\leq 2m_1\leq l-k}}
  \sum_{\substack{r_2+m_2=j_2\\0\leq2r_2\leq k,0\leq 2m_2\leq \tilde{l}-k}}c_1^k(\theta)a_{k,r_1,r_2}b_{l,k,m_1} b_{\tilde{l},k,m_2}
   \Bigg].
\end{align*}
We use again the formula \eqref{x^l=sum Pl} for $(\kappa \cdot \eta^-)^{l-2j_1}, (\kappa \cdot \eta^+)^{\tilde{l}-2j_2}$ that,
\begin{align*}
(\kappa \cdot \eta^-)^{l-2j_1}  &=
\sum_{0\leq p_1 \leq \frac{l}{2}-j_1} c_{l-2j_1,p_1}  P_{l-2j_1-2p_1}(\kappa \cdot \eta^-);\\
(\kappa \cdot \eta^+)^{\tilde{l}-2j_2}  &=
\sum_{0\leq p_2 \leq \frac{\tilde{l}}{2}-j_2} c_{\tilde{l}-2j_2,p_2}  P_{\tilde{l}-2j_2-2p_2}(\kappa \cdot \eta^+),
\end{align*}
and exchange the order of the summation of ${\bf M}$ again,
there exists a continuous coefficients
$b^{q,\tilde{q}}_{l,\tilde{l}}(\theta)$ such that
\begin{align*}
{\bf M}
=\sum_{0\leq 2q_1\leq l} \, \sum_{0\leq 2q_2\leq\tilde{l}}
  b^{l-2q_1,\tilde{l}-2q_2}_{l,\tilde{l}}(\theta) \,
  P_{l-2q_1}(\kappa\cdot\eta^-)
  P_{\tilde{l}-2q_2}(\kappa\cdot\eta^+).
\end{align*}
This conclude the proof of  Lemma \ref{lem int P+P-}
and Proposition \ref{orth}.}\\

\end{proof}

\subsection{Reduction of the expression of the non-linear {\color{black} eigenvalues}}

We derive in the following Propositions \ref{orth2} and \ref{prop reduc}
some simplifications of the expression of the non-linear
eigenvalue $ \mu^{m,\tilde{m},m_1^{\prime}}_{n,\tilde{n},l,\tilde{l},k_1}$, which will be used in the next Section~\ref{S5-2}.

\begin{proposition}\label{orth2}
For $G^{m,\tilde{m}}_{n,\tilde{n},l,\tilde{l}}(\kappa)$ given in \eqref{Gamma1} and
any integers
$n,\tilde{n} \geq0$,
$|m|\leq l$, $|m^{\prime}|\leq l^{\prime}$,
we have
\begin{align}\label{crit2}
\begin{split}
&\sum_{|m|\leq l}
\sum_{|\tilde{m}|\leq \tilde{l}}
\left|
  \int_{\mathbb{S}^2_{\kappa}}
  G^{m,\tilde{m}}_{n,\tilde{n},l,\tilde{l}}(\kappa)
  \overline{Y^{m^{\prime}}_{l^{\prime}}}(\kappa)
  d\kappa
\right|^2
=
\frac{2l+1}{4\pi}\frac{2\tilde{l}+1}{4\pi}
\\
&\int_{|\theta_1|\leq\frac{\pi}{4}}
  \beta(\theta_1)(\sin\theta_1)^{2n+l}
  (\cos\theta_1)^{2\tilde{n}+\tilde{l}}
\int_{|\theta_2|\leq\frac{\pi}{4}}
  \beta(\theta_2)(\sin\theta_2)^{2n+l}
  (\cos\theta_2)^{2\tilde{n}+\tilde{l}}
\\
&\times
\left(
2\pi
\int^1_{-1}
  F_{l,\tilde{l}}(x,\theta_1,\theta_2)
  P_{l^{\prime}}(x)
dx\right)
d\theta_2
d\theta_1
\end{split}
\end{align}
where
{\color{black}
\begin{align}\label{F}
\begin{split}
F_{l,\tilde{l}}(x,\theta_1,\theta_2)
& =
\int^{2\pi}_0\int^{2\pi}_0
  P_l\Big(\tau^1(\theta_2,\phi_2)J(x)\Big(\tau^1(\theta_1,\phi_1)\Big)^T
  \Big)\\
& \qquad\times
  P_{\tilde{l}}\Big(\tau(\theta_2,\phi_2)J(x)\Big( \tau(\theta_1,\phi_1)\Big)^T
  \Big)
\frac{d\phi_1}{2\pi}
\frac{d\phi_2}{2\pi},
\end{split}
\end{align}
}
$J(x)$ is the matrix function
$$
J(x)=\left(
    \begin{array}{ccc}
      x & -\sqrt{1-x^2} & 0 \\
      \sqrt{1-x^2} & x & 0 \\
      0& 0 & 1
    \end{array}
  \right)
$$
and $\tau(\theta,\phi)$,\,
$\tau^1(\theta,\phi)$ are the vectors
\begin{align*}
&\tau(\theta,\phi)=
  (\cos\theta,\sin\theta\cos\phi,\sin\theta\sin\phi),\\
&\tau^1(\theta,\phi)=
  (\sin\theta,-\cos\theta\cos\phi,-\cos\theta\sin\phi),
\end{align*}
{\color{black}the column vector $X^T$ is the transposition of the row vector $X=(x_1,x_2,x_3)$.}
\end{proposition}

\begin{remark}
We remark that in the formula \eqref{crit2},
the right hand side is independent of $m'$.
Therefore this implies
\begin{align}\label{GYm=Gm0}
\begin{split}
\sum_{|m|\leq l}
\sum_{|\tilde{m}|\leq \tilde{l}}
\left|
  \int_{\mathbb{S}^2_{\kappa}}
  G^{m,\tilde{m}}_{n,\tilde{n},l,\tilde{l}}(\kappa)
  \overline{Y^{m^{\prime}}_{l^{\prime}}}(\kappa)
  d\kappa
\right|^2
&=
\sum_{|m|\leq l}
\sum_{|\tilde{m}|\leq \tilde{l}}
\left|
  \int_{\mathbb{S}^2_{\kappa}}
  G^{m,\tilde{m}}_{n,\tilde{n},l,\tilde{l}}(\kappa)
  \overline{Y^{0}_{l^{\prime}}}(\kappa)
  d\kappa
\right|^2
\\
&=
\sum_{|q|\leq \min(l,\tilde{l})}
\left|
  \int_{\mathbb{S}^2_{\kappa}}
  G^{q,-q}_{n,\tilde{n},l,\tilde{l}}(\kappa)
  \overline{Y^{0}_{l^{\prime}}}(\kappa)
  d\kappa
\right|^2,
\end{split}
\end{align}
since from \eqref{nonlinear} the integral vanishes if $m+\tilde{m}\neq0$ .
\end{remark}

\begin{proof}
We will prove that
\begin{align*}
2\pi
\int^1_{-1}
  &F_{l,\tilde{l}}(x,\theta_1,\theta_2)
  P_{l^{\prime}}(x)
dx
= B_{l,\tilde{l},l^{\prime}}(\theta_1, \theta_2)
\end{align*}
where $B_{l,\tilde{l},l^{\prime}}(\theta_1, \theta_2)$
is given in the Remark \ref{rem sum |GY|}
and {\color{black}then} conclude.

We express the terms of $ F_{l,\tilde{l}}(x,\theta_1,\theta_2)$
given in \eqref{F}.  {\color{black}Constructing an orthonormal frame in $\mathbb{R}^3$ with respect to x ($|x|\leq1$) such that
$$\kappa_x = (x, \sqrt{1-x^2}, 0),\quad \kappa^1_x=(-\sqrt{1-x^2},x,0),\quad \kappa^2_x=(0,0,1),$$
we note from \eqref{k orthog}
\begin{equation*}
\kappa_x^{\bot}(\phi)=\kappa^1_x\cos\phi+\kappa^2_x\sin\phi=(-\sqrt{1-x^2} \cos\phi, x \cos\phi, \sin\phi).
\end{equation*}
One can verify that
\begin{align*}
& J(x)\Big(\tau^1(\theta_1,\phi_1)\Big)^T\\
&=
\left(
\begin{array}{ccc}
              x &  -\sqrt{1-x^2} & 0 \\
  \sqrt{1-x^2} &             x & 0 \\
              0 &             0 & 1
\end{array}
\right)
\left(
\begin{array}{c}
              \sin\theta_1 \\
  -\cos\theta_1 \cos\phi_1  \\
  -\cos\theta_1 \sin\phi_1
\end{array}
\right)
\\
&=
\left(
    \begin{array}{c}
     x \sin\theta_1 + \sqrt{1-x^2}\cos\theta_1 \cos\phi_1  \\
     \sqrt{1-x^2} \sin\theta_1 -  x \cos\theta_1 \cos\phi_1\\
     -\cos\theta_1 \sin\phi_1
    \end{array}
\right)
\\
&= \sin\theta_1 \, \kappa^T_x - \cos\theta_1 \, \Big(\kappa_x^{\bot}(\phi_1)\Big)^T.
\end{align*}
Similary, we have
\begin{align*}
J(x)\Big(\tau(\theta_1,\phi_1)\Big)^T
= \cos\theta_1 \, \kappa^T_x + \sin\theta_1 \, \Big(\kappa_x^{\bot}(\phi_1)\Big)^T.
\end{align*}
Therefore, we have
\begin{align}\label{another-express}
\begin{split}
& \tau^1(\theta_2,\phi_2)J(x)\Big(\tau^1(\theta_1,\phi_1)\Big)^T=(\sin\theta_1 \, \kappa_x - \cos\theta_1 \, \kappa_x^{\bot}(\phi_1))\cdot\tau^1(\theta_2,\phi_2)\\
&\tau(\theta_2,\phi_2)
    J(x)\Big(\tau(\theta_1,\phi_1)\Big)^T= (\cos\theta_1 \, \kappa_x + \sin\theta_1 \, \kappa_x^{\bot}(\phi_1))\cdot\tau(\theta_2,\phi_2).
\end{split}
\end{align}
}
From lemma \ref{lem int P+P-} and the above equality \eqref{another-express} , we deduce
\begin{align*}
&\int^{2\pi}_0
 P_l\Big(
 \tau^1(\theta_2,\phi_2)J(x)\Big(\tau^1(\theta_1,\phi_1)\Big)^T
  \Big)
  P_{\tilde{l}}\Big(
\tau(\theta_2,\phi_2)
    J(x)\Big(\tau(\theta_1,\phi_1)\Big)^T
  \Big)
\frac{d\phi_1}{2\pi}
\\
&\qquad=
\sum_{0\leq q \leq l} \, \sum_{0\leq \tilde{q} \leq\tilde{l}}
  b^{q,\tilde{q}}_{l,\tilde{l}}(\theta_1) \,
  P_{q}(\kappa_x\cdot\tau^1(\theta_2,\phi_2))
  P_{\tilde{q}}(\kappa_x\cdot\tau(\theta_2,\phi_2)).
\end{align*}
Therefore using \eqref{F}
\begin{align*}
2\pi
&\int^1_{-1}
  F_{l,\tilde{l}}(x,\theta_1,\theta_2)
  P_{l^{\prime}}(x)
dx
=
\sum_{0\leq q \leq l} \,
\sum_{0\leq \tilde{q} \leq\tilde{l}}
  b^{q,\tilde{q}}_{l,\tilde{l}}(\theta_1)
\times
\\
&
2\pi
\int^1_{-1}
\int^{2\pi}_0
  P_{q}(\kappa_x\cdot\tau^1(\theta_2,\phi_2))
  P_{\tilde{q}}(\kappa_x\cdot\tau(\theta_2,\phi_2))
  P_{l^{\prime}}(x)
\frac{d\phi_2}{2\pi}
dx.
\end{align*}
{\color{black}Applying} the addition theorem \eqref{addition theorem}
\begin{align*}
P_{\tilde{q}}(\kappa_x\cdot\tau(\theta_2,\phi_2))
&=\sum_{|\tilde{k}|\leq\tilde{q}}
a_{\tilde{q},k}
P^{|\tilde{k}|}_{\tilde{q}}(x)\,
P^{|\tilde{k}|}_{\tilde{q}}(\cos\theta_2)\,e^{i\tilde{k}\phi_2},
\\
P_{q}(\kappa_x\cdot{\color{black}\tau^1(\theta_2,\phi_2)})&={\color{black}P_q(x\sin\theta_2+\sqrt{1-x^2}\cos\theta_2\cos(\phi_2+\pi))}\\
&=\sum_{|k|\leq q}
a_{q,k}
P^{|k|}_{q}(x)\,
P^{|k|}_{q}(\sin\theta_2)\,e^{i k(\phi_2+\pi)} \\
&={\color{black}\sum_{|k|\leq q}
a_{q,k}
P^{|k|}_{q}(x)\,
P^{|k|}_{q}(\sin\theta_2)\,e^{i k\phi_2}(-1)^k
}
\end{align*}
{\color{black}where we use $e^{ik\pi}=(\cos\pi+i\sin\pi)^k=(-1)^{k}$ in the last equality,}
from the remark \ref{rem sum |GY|}, we find
\begin{align*}
2\pi
&\int^1_{-1}
  F_{l,\tilde{l}}(x,\theta_1,\theta_2)
  P_{l^{\prime}}(x)
dx
=
\sum_{0\leq q\leq l}\sum_{0\leq \tilde{q}\leq\tilde{l}}
b^{q,\tilde{q}}_{l,\tilde{l}}(\theta_1)
\quad \times
\\
&\sum_{|k|\leq \min(q,\tilde{q})}
  a_{q,k}
  a_{\tilde{q},k}
  (-1)^k P^{|k|}_{q}(\sin\theta_2)\
  P^{|k|}_{\tilde{q}}(\cos\theta_2)\,
  2\pi
  \int^1_{-1}
    P^{|k|}_{q}(x)\,
    P^{|k|}_{\tilde{q}}(x)\,
    P_{l^{\prime}}(x) \,
  dx,
\\
&= B_{l,\tilde{l},l^{\prime}}(\theta_1, \theta_2).
\end{align*}
This ends the proof of the formula \eqref{crit2}.
\end{proof}

The following Proposition will provide a convenient expression
to estimate the nonlinear eigenvalue
$\mu^{m,\tilde{m},m_1^{\prime}}_{n,\tilde{n},l,\tilde{l},k_1}$ in Section \ref{S5-2}.

\begin{proposition}\label{prop reduc}
For $G^{m,\tilde{m}}_{n,\tilde{n},l,\tilde{l}}(\kappa)$
given in \eqref{Gamma1},
and any integers
$n,\tilde{n} \geq0$,
$|m|\leq l$, $|m^{\prime}|\leq l^{\prime}$,
we have
\begin{align}\label{reduc}
\begin{split}
\sum_{|m|\leq l}\sum_{|\tilde{m}|\leq \tilde{l}}
&\left|\int_{\mathbb{S}^2_{\kappa}}G^{m,\tilde{m}}_{n,\tilde{n},l,\tilde{l}}(\kappa)
\overline{Y^{m^{\prime}}_{l^{\prime}}}
(\kappa)d\kappa\right|^2
\\
&=
\sum_{|q|\leq\min(l,\tilde{l})}
  \left(
  \left(\frac{4\pi}{2l^{\prime}+1}\right)^{\frac12}
  G^{q,-q}_{n,\tilde{n},l,\tilde{l}}(e_1)
  \right)
  \,
  \left(\int_{\mathbb{S}^2_{\kappa}}
    G^{q,-q}_{n,\tilde{n},l,\tilde{l}}(\kappa)
    \overline{Y^{0}_{l'}}(\kappa)
  d\kappa\right),
\end{split}
\end{align}
where $e_1=(1,0,0)$ and
\begin{align*}
G^{q,-q}_{n,\tilde{n},l,\tilde{l}}(e_1)
&=(-1)^q
{
\left(\frac{2l+1}
           {4\pi}
\right)^{\frac12}
\left(\frac{2\tilde{l}+1}
           {4\pi}
\right)^{\frac12}
\left(\frac{(l-|q|)!}{(l+|q|)!}\right)^{\frac12}
\left(\frac{(\tilde{l}-|q|)!}
    {(\tilde{l}+|q|)!}\right)^{\frac12}
} \,\,
\\
&\quad\times\int_{|\theta|\leq\frac{\pi}{4}}
  \beta(\theta)
  (\sin\theta)^{2n+l}
  (\cos\theta)^{2\tilde{n}+\tilde{l}}
  P^{|q|}_{l}(\sin\theta)
  P^{|q|}_{\tilde{l}}(\cos\theta)
d\theta.
\end{align*}
\end{proposition}

\begin{proof}
For $0\leq\,k\leq\min(l,\tilde{l})$ and $|m^{\prime}|\leq l'$, we deduce from \eqref{crit2} that,
\begin{align}\label{I=SS SGY^2}
&{\bf I} = \sum_{|m|\leq l}
\sum_{|\tilde{m}|\leq \tilde{l}}
\left|
\int_{\mathbb{S}^2_{\kappa}}
  G^{m,\tilde{m}}_{n,\tilde{n},l,\tilde{l}}(\kappa)
  \overline{Y^{m^{\prime}}_{l'}}
(\kappa)
d\kappa
\right|^2
=\frac{2l+1}{4\pi} \frac{2\tilde{l}+1}{4\pi}\quad\times
\nonumber\\
&\,\,\,\,
\int_{|\theta_1|\leq\frac{\pi}{4}}
  \beta(\theta_1)
  (\sin\theta_1)^{2n+l}(\cos\theta_1)^{2\tilde{n}+\tilde{l}}
  \quad\times
\\
&\,\,
\int_{|\theta_2|\leq\frac{\pi}{4}}
  \beta(\theta_2)
  (\sin\theta_2)^{2n+l}
  (\cos\theta_2)^{2\tilde{n}+\tilde{l}}
  \left(
    2\pi
    \int^1_{-1}
      F_{l,\tilde{l}}(x,\theta_1,\theta_2) \,
      P_{l^{\prime}}(x)dx
  \right)
d\theta_2{\color{black}d\theta_1},\nonumber
\end{align}
where $F_{l,\tilde{l}}(x,\theta_1,\theta_2)$
was defined in \eqref{F}, such that
\begin{align*}
F_{l,\tilde{l}}(x,\theta_1,\theta_2)&=
\int^{2\pi}_0
\int^{2\pi}_0
  P_l\Big(
  \tau^1(\theta_2,\phi_2)J(x)\Big(\tau^1(\theta_1,\phi_1)\Big)^T
  \Big)\\
& \qquad\times
  P_{\tilde{l}}\Big(
  \tau(\theta_2,\phi_2)J(x)\Big( \tau(\theta_1,\phi_1)\Big)^T
  \Big)
\frac{d\phi_1}{2\pi}
\frac{d\phi_2}{2\pi}.
\end{align*}
We inherit the notation in Proposition \ref{orth2}, apply the formula \eqref{another-express} and the addition theorem \eqref{addition theorem Y}, then
\begin{align*}
& P_l\Big(
  \tau^1(\theta_2,\phi_2)J(x)\Big(\tau^1(\theta_1,\phi_1)\Big)^T
  \Big)= {\textstyle \frac{4\pi}{2\,l+1} }
\sum^{l}_{q=-l}
Y^{q}_{l}(\sin\theta_1 \, \kappa_x - \cos\theta_1 \, \kappa_x^{\bot}(\phi_1))\,
Y^{-q}_{l}(\tau^1(\theta_2,\phi_2))
\\
&P_{\tilde{l}}\Big(
  \tau(\theta_2,\phi_2)J(x)\Big( \tau(\theta_1,\phi_1)\Big)^T
  \Big)
= {\textstyle \frac{4\pi}{2\,\tilde{l}+1} }
\sum^{\tilde{l}}_{\tilde{k}=-\tilde{l}}
Y^{\tilde{q}}_{\tilde{l}}(\cos\theta_1 \, \kappa_x + \sin\theta_1 \, \kappa_x^{\bot}(\phi_1))\,
Y^{-\tilde{q}}_{\tilde{l}}(\tau(\theta_2,\phi_2)).
\end{align*}
Since
\begin{align*}
&Y^{-q}_{l}(\tau^1(\theta_2,\phi_2)) =
\left(\frac{2l+1}{4\pi}\right)^{\frac12}
\left(\frac{(l-|q|)!}{(l+|q|)!}\right)^{\frac12}
 P^{|q|}_{l}(\sin\theta_2)
 e^{-iq(\phi_2+\pi)}
\\
&Y^{-\tilde{q}}_{\tilde{l}}(\tau(\theta_2,\phi_2)) =
{\color{black}\left(\frac{2\tilde{l}+1}{4\pi}\right)^{\frac12}}
\left(\frac{(\tilde{l}-|\tilde{q}|)!}
      {(\tilde{l}+|\tilde{q}|)!}\right)^{\frac12}
 P^{|\tilde{q}|}_{\tilde{l}}(\cos\theta_2)
 e^{-i\tilde{q}\phi_2}
\end{align*}
we find that
\begin{align*}
&F_{l,\tilde{l}}(x,\theta_1,\theta_2)\\
&=
\sum_{|q|\leq \min(l,\tilde{l})}
  \left(\frac{4\pi}{2l+1}\right)^{\frac12}
  \left(\frac{4\pi}{2\tilde{l}+1}\right)^{\frac12}
  \left(\frac{(l-|q|)!}{(l+|q|)!}\right)^{\frac12}
  \left(\frac{(\tilde{l}-|q|)!}{(\tilde{l}+|q|)!}\right)^{\frac12}P^{|q|}_{l}(\sin\theta_2) (-1)^q
\\
&\quad
\times P^{|q|}_{\tilde{l}}(\cos\theta_2)
\int^{2\pi}_0
  Y^{q}_{l}(\sin\theta_1 \, \kappa_x - \cos\theta_1 \, \kappa_x^{\bot}(\phi_1))
  Y^{-q}_{\tilde{l}}(\cos\theta_1 \, \kappa_x + \sin\theta_1 \, \kappa_x^{\bot}(\phi_1))
\frac{d\phi_1}{2\pi}.
\end{align*}
We plug the previous relation into \eqref{I=SS SGY^2}
and we get
\begin{align*}
{\bf I}&=
\sum_{|q|\leq \min(l,\tilde{l})}
\int_{|\theta_2|\leq\frac{\pi}{4}}
  \beta(\theta_2)
  (\sin\theta_2)^{2n+l}
  (\cos\theta_2)^{2\tilde{n}+\tilde{l}}
  P^{|q|}_{l}(\sin\theta_2)
  P^{|q|}_{\tilde{l}}(\cos\theta_2)
d\theta_2
\\
&\quad\times(-1)^q
\left(\frac{2l+1}{4\pi}\right)^{\frac12}
\left(\frac{2\tilde{l}+1}{4\pi}\right)^{\frac12}
\left(\frac{(l-|q|)!}{(l+|q|)!}\right)^{\frac12}
\left(\frac{(\tilde{l}-|q|)!}
    {(\tilde{l}+|q|)!}\right)^{\frac12}
\\
&\quad
\times\int_{|\theta_1|\leq\frac{\pi}{4}}
  \beta(\theta_1)
  (\sin\theta_1)^{2n+l}(\cos\theta_1)^{2\tilde{n}+\tilde{l}}
\quad\times\int^{2\pi}_0
  2\pi
  \int^1_{-1}
\\
&Y^{q}_{l}(\sin\theta_1 \, \kappa_x - \cos\theta_1 \, \kappa_x^{\bot}(\phi_1))
  Y^{-q}_{\tilde{l}}(\cos\theta_1 \, \kappa_x + \sin\theta_1 \, \kappa_x^{\bot}(\phi_1))
    P_{l^{\prime}}(x)
  dx
\frac{d\phi_1}{2\pi}
d\theta_1.
\end{align*}
On one hand, from \eqref{G=}
\begin{align*}
&G^{q,-q}_{n,\tilde{n},l,\tilde{l}}(e_1)
=
\int_{|\theta_2|\leq\frac{\pi}{4}}
  \beta(\theta_2)(\sin\theta_2)^{2n+l}
  (\cos\theta_2)^{2\tilde{n}+\tilde{l}} \quad\times
\nonumber\\
&\quad
\int^{2\pi}_{0}
  Y^q_l(e_1\sin\theta_2 - e_1^{\bot}(\phi_2)\cos\theta_2)
  Y^{-q}_{\tilde{l}}
  (e_1\cos\theta_2 + e_1^{\bot}(\phi_2)\sin\theta_2)
\frac{d\phi_2}{2\pi}
d\theta_2
\\
&=
(-1)^q{\textstyle
\left(\frac{2l+1}{4\pi}\right)^{\frac12}
\left(\frac{2\tilde{l}+1}{4\pi}\right)^{\frac12}
\left(\frac{(l-|q|)!}{(l+|q|)!}\right)^{\frac12}
\left(\frac{(\tilde{l}-|q|)!}
    {(\tilde{l}+|q|)!}\right)^{\frac12}
}
\\
&\quad\times\int_{|\theta_2|\leq\frac{\pi}{4}}
  \beta(\theta_2)
  (\sin\theta_2)^{2n+l}
  (\cos\theta_2)^{2\tilde{n}+\tilde{l}}P^{|q|}_{l}(\sin\theta_2)
  P^{|q|}_{\tilde{l}}(\cos\theta_2)
d\theta_2.
\end{align*}
On the other hand, from \eqref{int GY=} and from \eqref{S Yq Y-q}
of the next lemma \ref{lem reduc},
\begin{align*}
&\int_{\mathbb{S}^2_{\kappa}}
  G^{q,-q}_{n,\tilde{n},l,\tilde{l}}
  (\kappa)\overline{Y^{0}_{l'}}(\kappa)
d\kappa
=
\int_{|\theta_1|\leq\frac{\pi}{4}}
  \beta(\theta_1)(\sin\theta_1)^{2n+l}
  (\cos\theta_1)^{2\tilde{n}+\tilde{l}}
\int^{2\pi}_{0}
\int_{\mathbb{S}^2_\kappa}
\\
&
\,\,\,\,  Y^q_l(\kappa\sin\theta_1-\kappa^{\bot}(\phi_1)\cos\theta_1)
  Y^{-q}_{\tilde{l}}
  (\kappa\cos\theta_1+\kappa^{\bot}(\phi_1)\sin\theta_1)
  \overline{Y^{0}_{l'}}(\kappa)
d\kappa
\frac{d\phi_1}{2\pi}
d\theta_1
\\
&=
\left(\frac{2l^{\prime}+1}{4\pi}\right)^{\frac12}
\int_{|\theta_1|\leq\frac{\pi}{4}}
  \beta(\theta_1)
  (\sin\theta_1)^{2n+l}(\cos\theta_1)^{2\tilde{n}+\tilde{l}}
\quad\times\int^{2\pi}_0
  2\pi
  \int^1_{-1}
\\
&
Y^{q}_{l}(\sin\theta_1 \, \kappa_x - \cos\theta_1 \, \kappa_x^{\bot}(\phi_1))
  Y^{-q}_{\tilde{l}}(\cos\theta_1 \, \kappa_x + \sin\theta_1 \, \kappa_x^{\bot}(\phi_1))
    P_{l^{\prime}}(x)
  dx
\frac{d\phi_1}{2\pi}
d\theta_1.
\end{align*}
Combining the three previous relations leads to \eqref{reduc}, and this concludes the proof of the Proposition.
\end{proof}

We now prove the following technical lemma.

\begin{lemma}\label{lem reduc}
For any integers
$l,\tilde{l},l^{\prime}\geq0$ and
$|q|\leq l$, we have
\begin{equation}\label{S Yq Y-q}
\begin{split}
&\left(\frac{4\pi}{2l^{\prime}+1}\right)^{\frac12}\int_{\mathbb{S}^2_\kappa}
  Y^q_l(\kappa\sin\theta_1-\kappa^{\bot}(\phi_1)\cos\theta_1)
  Y^{-q}_{\tilde{l}}
  (\kappa\cos\theta_1+\kappa^{\bot}(\phi_1)\sin\theta_1)
  \overline{Y^{0}_{l'}}(\kappa)
d\kappa
\\
&={\color{black}2\pi
\int_{-1}^1
   Y^{q}_{l}(\sin\theta_1 \, \kappa_x - \cos\theta_1 \, \kappa_x^{\bot}(\phi_1))
  Y^{-q}_{\tilde{l}}(\cos\theta_1 \, \kappa_x + \sin\theta_1 \, \kappa_x^{\bot}(\phi_1))
    P_{l^{ \prime}}(x)
dx}.
\end{split}
\end{equation}
\end{lemma}

\begin{proof}
We consider
\begin{align*}
  {\bf I} &=
  \int_{\mathbb{S}^2_\kappa}
  Y^q_l(\kappa\sin\theta_1-\kappa^{\bot}(\phi_1)\cos\theta_1)
  Y^{-q}_{\tilde{l}}
  (\kappa\cos\theta_1+\kappa^{\bot}(\phi_1)\sin\theta_1)
  \overline{Y^{0}_{l'}}(\kappa)
d\kappa.
\end{align*}
From \eqref{Ylm2} we have
\begin{align*}
Y^q_l(\kappa\sin\theta_1-\kappa^{\bot}(\phi_1)\cos\theta_1)
&= N_{l,q}
   \left( \frac{d^{|q|} P_l}{dx^{|q|}}\right)(\sigma_1^-) \,
  (\sigma_2^- + i\, \operatorname{sgn}(q) \, \sigma_3^-)^{|q|},
\\
Y^{-q}_{\tilde{l}}
  (\kappa\cos\theta_1+\kappa^{\bot}(\phi_1)\sin\theta_1)
&= N_{\tilde{l},q}
   \left( \frac{d^{|q|} P_{\tilde{l}}}{dx^{|q|}}\right)
     (\sigma_1^+) \,
  (\sigma_2^+ - i\, \operatorname{sgn}(q) \, \sigma_3^+)^{|q|},
\end{align*}
where
\begin{align*}
\sigma^-&= \kappa\sin\theta_1-\kappa^{\bot}(\phi_1)\cos\theta_1
= (\sigma_1^-,\sigma_2^-,\sigma_3^-)
\\
\sigma^+
&= \kappa\cos\theta_1+\kappa^{\bot}(\phi_1)\sin\theta_1
=(\sigma_1^+,\sigma_2^+,\sigma_3^+) .
\end{align*}
Noting {\color{black}$\kappa  = (\cos\theta,\sin\theta\cos\phi,\sin\theta\sin\phi)$ with $\theta\in[0,\pi]$ and $\phi\in[0.2\pi]$, and
$$\kappa^1=(-\sin\theta,\cos\theta\cos\phi,\cos\theta\sin\phi),\quad \kappa^2=(0,\sin\phi,-\cos\phi),$$
and using the definition \eqref{k orthog}
of $\kappa^{\bot}(\phi_1)$,
we have
\begin{align*}
\sigma_1^-=&\cos\theta\sin\theta_1
+\sin\theta\cos\theta_1\cos\phi_1,\\
\sigma_2^-=&\sin\theta\cos\phi\sin\theta_1
-\cos\theta\cos\phi\cos\theta_1\cos\phi_1-\sin\phi\cos\theta_1\sin\phi_1,\\
\sigma_3^-=&\sin\theta\sin\phi\sin\theta_1
 -\cos\theta\sin\phi\cos\theta_1\cos\phi_1
 +\cos\phi\cos\theta_1\sin\phi_1,\\
\end{align*}
and
\begin{align*}
\sigma_1^+=&\cos\theta\cos\theta_1-\sin\theta\sin\theta_1\cos\phi_1,\\
\sigma_2^+=&\sin\theta\cos\phi\cos\theta_1+\cos\theta\cos\phi\sin\theta_1\cos\phi_1+\sin\phi\sin\theta_1\sin\phi_1,\\
\sigma_3^+=&\sin\theta\sin\phi\cos\theta_1
  +\cos\theta\sin\phi\sin\theta_1\cos\phi_1
  -\cos\phi\sin\theta_1\sin\phi_1.
\end{align*}
Therefore
\begin{align*}
&(\sigma_2^- + i\,  \sigma_3^-)\\
&=(\sin\theta\cos\phi\sin\theta_1-\cos\theta\cos\phi\cos\theta_1\cos\phi_1-\sin\phi\cos\theta_1\sin\phi_1)\\
&\quad+i(\sin\theta\sin\phi\sin\theta_1-\cos\theta\sin\phi\cos\theta_1\cos\phi_1+\cos\phi\cos\theta_1\sin\phi_1)\\
&=(\sin\theta\sin\theta_1-\cos\theta\cos\theta_1\cos\phi_1)e^{i\phi}+i\cos\theta_1\sin\phi_1e^{i\phi}\\
&=(\sin\theta\sin\theta_1-\cos\theta\cos\theta_1\cos\phi_1+i\cos\theta_1\sin\phi_1)e^{i\phi};\\
&(\sigma_2^+ - i\,  \, \sigma_3^+)\\
&=(\sin\theta\cos\phi\cos\theta_1+\cos\theta\cos\phi\sin\theta_1\cos\phi_1+\sin\phi\sin\theta_1\sin\phi_1)\\
&\quad-i(\sin\theta\sin\phi\cos\theta_1+\cos\theta\sin\phi\sin\theta_1\cos\phi_1-\cos\phi\sin\theta_1\sin\phi_1)\\
&=(\sin\theta\cos\theta_1+\cos\theta\sin\theta_1\cos\phi_1+i\sin\theta_1\sin\phi_1)e^{-i\phi}.
\end{align*}}
Direct computations lead to
\begin{align*}
(\sigma_2^- + i\, \sigma_3^-)
(\sigma_2^+ - i\, \sigma_3^+)
&=(\sin\theta\sin\theta_1-\cos\theta\cos\theta_1\cos\phi_1+i\cos\theta_1\sin\phi_1)\\
&\quad\times(\sin\theta\cos\theta_1+\cos\theta\sin\theta_1\cos\phi_1+i\sin\theta_1\sin\phi_1),
\end{align*}
which does not depend of $\phi$.
Since $\sigma_1^\pm$ do not depend also on $\phi$,
we get with the change of variable $x=\cos\theta$
\begin{align*}
{\bf I}
&=
2\pi
\int_{0}^{\pi}
  Y^q_l(\kappa_{\theta,0}\sin\theta_1-\kappa_{\theta,0}^{\bot}(\phi_1)\cos\theta_1)
\\
&\quad\quad\quad\quad
  Y^{-q}_{\tilde{l}}
  (\kappa_{\theta,0}\cos\theta_1
    +\kappa_{\theta,0}^{\bot}(\phi_1)\sin\theta_1)
  \overline{Y^{0}_{l'}}(\kappa_{\theta,0})
\, \sin\theta \, d\theta
\\
&=
2\pi
\int_{-1}^1
    Y^{q}_{l}(\sin\theta_1 \, \kappa_x - \cos\theta_1 \, \kappa_x^{\bot}(\phi_1))
  Y^{-q}_{\tilde{l}}(\cos\theta_1 \, \kappa_x + \sin\theta_1 \, \kappa_x^{\bot}(\phi_1))
    \overline{Y^{0}_{l^{ \prime}}}(\kappa_x) dx\\
&={\color{black}
2\pi\left(\frac{2l^{\prime}+1}{4\pi}\right)^{\frac12}
\int_{-1}^1
    Y^{q}_{l}(\sin\theta_1 \, \kappa_x - \cos\theta_1 \, \kappa_x^{\bot}(\phi_1))
  Y^{-q}_{\tilde{l}}(\cos\theta_1 \, \kappa_x + \sin\theta_1 \, \kappa_x^{\bot}(\phi_1))
    P_{l^{ \prime}}(x)
dx}.
\end{align*}
This concludes the proof of {\color{black} Lemma} \ref{lem reduc} and Proposition \ref{prop reduc}.
\end{proof}


\setcounter{section}{5}

\section{The estimates of the non linear eigenvalues }\label{S5-2}

In this section, we prove {\color{black} Proposition} \ref{est}, we need the following fundamental result of Gamma function.
It is well known of the stirling's formula (see 12.33 {\color{black}of Chap. XII} in \cite{Whit}, \cite{Rudin}) that,
$$
\Gamma(x+1)=
\sqrt{2\pi x} \,
\Big(\frac{x}{e}\Big)^xe^{\frac{\nu(x)}{12x}},\,\text{for}\,x\geq1,
$$
where $0<\nu(x)<1$.
{\color{black} Therefore we derive directly the following useful estimate}.

Let $a,\,b$ be two fixed constant,  for any $x>0$, with $|b-a|\leq\,x+b$, $x+a\geq1$, $x+b\geq1$,\,we have
\begin{align}\label{Gamma function2}
\frac{\Gamma(x+a+1)}{\Gamma(x+b+1)}\leq\,C_{a,b}(x+a)^{a-b},
\end{align}
where $C_{a,b}$ is dependent only on $a,\,b.$
We also recall the definition of the Beta function
\begin{align}\label{Beta function}
B(x,y) = \int^1_0
 t^{x-1} (1-t)^{y-1} dt
=\frac{\Gamma(x)\Gamma(y)}{\Gamma(x+y)}.
\end{align}

\subsection{The estimate for the radially symmetric terms }
We first give the estimate of $|\lambda^{rad, 1}_{n,\tilde{n},\tilde{l}}|^2$, and $|\lambda^{rad, 2}_{n,\tilde{n},l}|^2$, which is $1),\,2)$ in Proposition \ref{est}.  Recall that
{\color{black}
\begin{align*}
\lambda^{rad, 1}_{n,\tilde{n},\tilde{l}}&=\frac{1}{\sqrt{4\pi}}\frac{A_{\tilde{n},\tilde{l}}
A_{n,0}}{A_{n+\tilde{n},\tilde{l}}}\int_{|\theta|\leq\frac{\pi}{4}}
\beta(\theta)(\sin\theta)^{2n}(\cos\theta)^{2\tilde{n}+\tilde{l}}
P_{\tilde{l}}(\cos\theta)d\theta,\\
\lambda^{rad, 2}_{n,\tilde{n},l}&=\frac{1}{\sqrt{4\pi}}\frac{A_{\tilde{n},0}A_{n,l}}{A_{n+\tilde{n},l}}
\int_{|\theta|\leq\frac{\pi}{4}}\beta(\theta)(\sin\theta)^{2n+l}
(\cos\theta)^{2\tilde{n}}P_{l}(\sin\theta)d\theta
\end{align*}
}
where
$$
A_{n,l}=(-i)^l(2\pi)^\frac{3}{4}\left(\frac{1}{\sqrt{2}n!
\Gamma(n+l+\frac{3}{2})}\right)^\frac{1}{2}.
$$

\begin{lemma}\label{est1-b}
For $n\geq1$, $\tilde{n},\tilde{l}\in\mathbb{N}$,
\begin{equation}\label{estrad1-b}
|\lambda^{rad, 1}_{n,\tilde{n},\tilde{l}}|^2\lesssim\tilde{n}^s(\tilde{n}
+\tilde{l})^sn^{-\frac{5}{2}-2s}.
\end{equation}
For all $\tilde{n}\geq1$, $n,l\in\mathbb{N}$, $n+l\geq2$,
\begin{equation}\label{estrad2-b}
|\lambda^{rad, 2}_{n,\tilde{n},l}|^2\lesssim \frac{\tilde{n}^{2s}}{(n+1)^s(n+l)^{\frac{5}{2}+s}}.
\end{equation}
\end{lemma}
\begin{proof}
We estimate $|\lambda^{rad, 1}_{n,\tilde{n},\tilde{l}}|^2$.\,\,
{\color{black} From the assumption on} $\beta(\theta)$
$$\beta(\theta)\approx|\sin\theta|^{-1-2s},$$
and {\color{black} the inequality} {\color{black}$|P_{\tilde{l}}(x)|\leq1$ for any $|x|\leq1$},  we have,
\begin{align*}
|\lambda^{rad, 1}_{n,\tilde{n},\tilde{l}}|^2
&\lesssim
\frac{(n+\tilde{n})!\Gamma(n+\tilde{n}+\tilde{l}+\frac{3}{2})}
{n!\tilde{n}!\Gamma(n+\frac{3}{2})\Gamma(\tilde{n}+\tilde{l}+\frac{3}{2})}
\left(\int_0^{\frac{\pi}{4}}
(\sin\theta)^{2n-1-2s}(\cos\theta)^{2\tilde{n}+\tilde{l}}
d\theta
\right)^2.
\end{align*}
{\color{black} Using} the Cauchy-Schwarz inequality and the Beta Function
\eqref{Beta function},
we derive that
\begin{align*}
&\left(\int^{\frac{\pi}{4}}_{0}
  (\sin\theta)^{2n-1-2s}(\cos\theta)^{2\tilde{n}+\tilde{l}} \, d\theta\right)^2={\color{black}\frac{1}{4}}\Big(\int^{\frac{1}{2}}_0t^{n-1-s}(1-t)^{\tilde{n}+\frac{\tilde{l}}{2}-\frac{1}{2}}dt\Big)^2
  \\
&{\color{black}\leq
\frac{1}{4}\,
\left(\int^{\frac{1}{2}}_0t^{n-1-s}(1-t)^{\tilde{n}+s}dt\right)\times\left(\int^{\frac{1}{2}}_0t^{n-1-s}(1-t)^{\tilde{n}+\tilde{l}-1-s}dt\right)}
\\
&{\color{black}\leq
\frac{1}{4}
\left(\int^{\frac{1}{2}}_0t^{n-1-s}(1-t)^{\tilde{n}+s}dt\right)\times 2^{\frac32+2s}}\,\left(\int^{\frac{1}{2}}_0t^{n-1-s}(1-t)^{\tilde{n}+\tilde{l}+\frac{1}{2}+s}dt\right)
\\
&\lesssim\frac{
  (\Gamma(n-s))^2
  \Gamma(\tilde{n}+1+s)\Gamma(\tilde n + \tilde l+\frac{3}{2}+s)}
{(n+\tilde{n})!\Gamma(n+\tilde n + \tilde l+\frac{3}{2})}.
\end{align*}
Then,
\begin{align}\label{rad1}
\begin{split}
|\lambda^{rad, 1}_{n,\tilde{n},\tilde{l}}|^2&\lesssim
\frac{(n+\tilde{n})!\Gamma(n+\tilde{n}+\tilde{l}+\frac{3}{2})}{n!\tilde{n}!
\Gamma(n+\frac{3}{2})\Gamma(\tilde{n}+\tilde{l}+\frac{3}{2})}\\
&\qquad\times\frac{
  (\Gamma(n-s))^2
  \Gamma(\tilde{n}+1+s)\Gamma(\tilde n + \tilde l+\frac{3}{2}+s)}
{(n+\tilde{n})!\Gamma(n+\tilde n + \tilde l+\frac{3}{2})}\\
&\lesssim\frac{\Gamma(n-s)\Gamma(n-s)\Gamma(\tilde{n}+1+s)\Gamma(\tilde{n}+\tilde{l}+\frac{3}{2}+s)}{n!\Gamma(n+\frac{3}{2})\tilde{n}!\Gamma(\tilde{n}+\tilde{l}+\frac{3}{2})}.
\end{split}
\end{align}
We deduce from the formula \eqref{Gamma function2} with $x=n$, $a=-s$, $b=0$,
$$
\frac{\Gamma(n-s+1)}{n!}\lesssim\frac{1}{(n-s)^s}.
$$
{\color{black} From the following} recurrence formula
for {\color{black}  the} Gamma function
$$
\Gamma(n+1-s)=(n-s)\Gamma(n-s),
$$
we obtain
\begin{align*}
\frac{\Gamma(n-s)}{n!}=\frac{1}{(n-s)}\frac{\Gamma(n+1-s)}
{\Gamma(n+1)}\lesssim\frac{1}{(n-s)^{s+1}}\lesssim\frac{1}{n^{1+s}}.
\end{align*}
Using $x=n$, $a=-s$, $b=\frac{1}{2}$ in \eqref{Gamma function2},
$$\frac{\Gamma(n+1-s)}{\Gamma(n+\frac{3}{2})}\lesssim\frac{1}{n^{\frac{1}{2}+s}},$$
and {\color{black} the} recurrence formula $\Gamma(n+1-s)=(n-s)\Gamma(n-s)$,
\begin{align*}
\frac{\Gamma(n-s)}{\Gamma(n+\frac{3}{2})}=\frac{1}{n-s}\frac{\Gamma(n+1-s)}
{\Gamma(n+\frac{3}{2})}\lesssim\frac{1}{(n-s)n^{\frac{1}{2}+s}}\lesssim\frac{1}{n^{\frac{3}{2}+s}}.
\end{align*}
Using $x=\tilde{n}+1$, $a=s$,\,$b=0$ in \eqref{Gamma function2}, we have
\begin{align*}
&\frac{\Gamma(\tilde{n}+1+s)}{\tilde{n}!}=\frac{\tilde{n}+1}{\tilde{n}+1+s}\frac{\Gamma(\tilde{n}+2+s)}{(\tilde{n}+1)!}\lesssim\tilde{n}^s.
\end{align*}
Using $x=\tilde{n}+\tilde{l}+\frac{1}{2}$,\,$a=s$,\,$b=0$,
\begin{align*}
\frac{\Gamma(\tilde{n}+\tilde{l}+\frac{3}{2}+s)}{\Gamma(\tilde{n}+\tilde{l}+\frac{3}{2})}\lesssim(\tilde{n}+\tilde{l}+\frac{1}{2}+s)^s\lesssim(\tilde{n}+\tilde{l})^s.
\end{align*}
Substitute {\color{black} the previous estimates} into \eqref{rad1}, we obtain
$$
|\lambda^{rad, 1}_{n,\tilde{n},\tilde{l}}|^2\lesssim\tilde{n}^s(\tilde{n}+\tilde{l})^sn^{-\frac{5}{2}-2s}.
$$
This is the formula of \eqref{estrad1-b}

Analogously, for the term $|\lambda^{rad, 2}_{n,\tilde{n},l}|^2$, we use the Cauchy-Schwarz inequality
\begin{align*}
|\lambda^{rad, 2}_{n,\tilde{n},l}|^2&\approx
\frac{(n+\tilde{n})!\Gamma(n+\tilde{n}+l+\frac{3}{2})}{n!\tilde{n}!
\Gamma(n+l+\frac{3}{2})\Gamma(\tilde{n}+\frac{3}{2})}{\color{black}\left(\int^{\frac{\pi}{4}}_{0}(\sin\theta)^{2n+l-1-2s}(\cos\theta)^{2\tilde{n}}d\theta\right)^2}\\
&=\frac{1}{4}\frac{(n+\tilde{n})!\Gamma(n+\tilde{n}+l+\frac{3}{2})}{n!\tilde{n}!
\Gamma(n+l+\frac{3}{2})\Gamma(\tilde{n}+\frac{3}{2})}{\color{black}\left(\int^{\frac{1}{2}}_{0}t^{n+\frac{l}{2}-1-s}(1-t)^{\tilde{n}-\frac{1}{2}}dt\right)^2}\\
&\lesssim\frac{(n+\tilde{n})!\Gamma(n+\tilde{n}+l+\frac{3}{2})}{n!\tilde{n}!
\Gamma(n+l+\frac{3}{2})\Gamma(\tilde{n}+\frac{3}{2})}\\
&\quad\times{\color{black}\left(\int^{\frac{1}{2}}_0t^{n+l-2-s}(1-t)^{\tilde{n}+\frac{3}{2}+s}dt\right)\times\left(\int^{\frac{1}{2}}_0t^{n-s}(1-t)^{\tilde{n}+s}dt\right)}\\
&\leq \frac{(n+\tilde{n})!\Gamma(n+\tilde{n}+l+\frac{3}{2})}{n!\tilde{n}!
\Gamma(n+l+\frac{3}{2})\Gamma(\tilde{n}+\frac{3}{2})}\\
&\quad\times{\color{black}
\left(\frac{\Gamma(n+l-1-s)\Gamma(\tilde{n}+\frac{5}{2}+s)}{\Gamma(n+\tilde{n}+l+\frac{3}{2})}\right)\times
\left(\frac{\Gamma(n+1-s)\Gamma(\tilde{n}+1+s)}{\Gamma(\tilde{n}+n+2)}\right)}\\
&\lesssim\frac{\Gamma(n+1-s)\Gamma(n+l-1-s)\Gamma(\tilde{n}+1+s)\Gamma(\tilde{n}+\frac{5}{2}+s)}{(n+\tilde{n}+1)n!\Gamma(n+l+\frac{3}{2})\tilde{n}!\Gamma(\tilde{n}+\frac{3}{2})}.
\end{align*}
{\color{black}
We deduce from the formula \eqref{Gamma function2} with {\color{red}$x=n+1\geq2$, $a=-s, b=0,$}
$$\frac{\Gamma(n+2-s)}{(n+1)!}\lesssim \frac{1}{(n+1-s)^s}\lesssim \frac{1}{(n+1)^s}.$$
This inequality is also right for $n=0$ $(\text{indeed}, \,\frac{\Gamma(n+2-s)}{(n+1)!}\mid_{n=0}\lesssim1)$.
{\color{red}By using the recurrence formula
$\Gamma(n+2-s)=(n+1-s)\Gamma(n+1-s),$ we conclude that
 $$\frac{\Gamma(n+1-s)}{n!}=\frac{n+1}{n+1-s}\frac{\Gamma(n+2-s)}{(n+1)!}\lesssim \frac{1}{(n+1)^s}.$$
 }Consider the assumption that $n,l\in \mathbb{N}, n+l\geq2,$ using the formula \eqref{Gamma function2} with {\color{red}$x=n+l, a=-s, b=\frac{1}{2}$,
$$\frac{\Gamma(n+l+1-s)}{\Gamma(n+l+\frac{3}{2})}\lesssim\frac{1}{(n+l-s)^{\frac{1}{2}+s}}.$$
By using the recurrence formula
$$\Gamma(n+l+1-s)=(n+l-s)(n+l-1-s)\Gamma(n+l-1-s),$$
we obtain,
$$\frac{\Gamma(n+l-1-s)}{\Gamma(n+l+\frac{3}{2})}=\frac{1}{(n+l-s)(n+l-1-s)}\frac{\Gamma(n+l+1-s)}{\Gamma(n+l+\frac{3}{2})}\lesssim\frac{1}{(n+l)^{\frac{5}{2}+s}}.$$
}Finally, under the assumption of $\tilde{n}\geq1$, we use again \eqref{Gamma function2} with $x=\tilde{n}, a=s, b=0$, then
$$\frac{\Gamma(\tilde{n}+1+s)}{\tilde{n}!}\lesssim(\tilde{n}+s)^s\lesssim\tilde{n}^s.$$
Using  the formula \eqref{Gamma function2} with $x=\tilde{n}, a=\frac{3}{2}+s, b=\frac{1}{2}$, we obtain,
$$\frac{\Gamma(\tilde{n}+\frac{5}{2}+s)}{\Gamma(\tilde{n}+\frac{3}{2})}\lesssim(\tilde{n}+\frac{3}{2}+s)^{1+s}\lesssim \tilde{n}^{1+s}.$$
}
{\color{black}Therefore, under the assumption of $\tilde{n}\geq1$, $n,l\in\mathbb{N}$, $n+l\geq2$, we conclude}
$$
|\lambda^{rad, 2}_{n,\tilde{n},l}|^2\lesssim\frac{\tilde{n}^{1+2s}}{(n+1)^s(n+l)^{\frac{5}{2}+s}(n+\tilde{n}+1)}\leq\frac{\tilde{n}^{2s}}{(n+1)^s(n+l)^{\frac{5}{2}+s}}.
$$
This ends the proof of \eqref{estrad2-b}.
\end{proof}

\subsection{The estimate for the general terms}
In the proof of $3)$ in Proposition \ref{est}, we need the following technical Lemma.
Recall that
$$
A_{n,l}=(-i)^l(2\pi)^\frac{3}{4}\left(\frac{1}{\sqrt{2}n!
\Gamma(n+l+\frac{3}{2})}\right)^\frac{1}{2},
$$
and
\begin{align*}
G^{m,\tilde{m}}_{n,\tilde{n},l,\tilde{l}}(\kappa)&=
\int_{\mathbb{S}^2}b(\kappa\cdot\sigma)\Big(|\kappa-\sigma|/2\Big)^{2n+l}
\Big(|\kappa+\sigma|/2\Big)^{2\tilde{n}+\tilde{l}}\\
&\quad\qquad\times
Y^{m}_l\Big(\frac{\kappa-\sigma}
{|\kappa-\sigma|}\Big) Y^{\tilde{m}}_{\tilde{l}}\Big(\frac{\kappa+\sigma}{|\kappa+\sigma|}\Big)\, d\sigma.
\end{align*}
Then {\color{black} recalling} the notation in Proposition \ref{expansion}, we have
$$
\mu^{m,\tilde{m},m^{\prime}}_{n,\tilde{n},l,\tilde{l},k}=
\frac{A_{\tilde{n},\tilde{l}}A_{n,l}}{A_{n+\tilde{n}+k,l+\tilde{l}-2k}}
\left(\int_{S^2_{\kappa}}G^{m,\tilde{m}}_{n,\tilde{n},l,\tilde{l}}(\kappa)
\overline{Y^{m^{\prime}}_{l+\tilde{l}-2k}}
(\kappa)d\kappa\right),
$$
It follows that,
\begin{align}\label{present}
\begin{split}
\sum_{|m|\leq\,l}\sum_{|\tilde{m}|\leq\,\tilde{l}}\left|\mu^{m,\tilde{m},m^{\prime}}_{n,\tilde{n},l,\tilde{l},k}\right|^2
&=\Big|\frac{A_{\tilde{n},\tilde{l}}A_{n,l}}{A_{n+\tilde{n}+k,l+\tilde{l}-2k}}
\Big|^2\\
&\qquad\times\sum_{|m|\leq\,l}\sum_{|\tilde{m}|\leq\,\tilde{l}}
\left|\int_{S^2_{\kappa}}G^{m,\tilde{m}}_{n,\tilde{n},l,\tilde{l}}(\kappa)
\overline{Y^{m^{\prime}}_{l+\tilde{l}-2k}}
(\kappa)d\kappa\right|^2.
\end{split}
\end{align}
In the next Lemma we estimate
$$\sum_{|m|\leq l}\sum_{|\tilde{m}|\leq \tilde{l}}\Big|\int_{\mathbb{S}^2_{\kappa}}G^{m,\tilde{m}}_{n,\tilde{n},l,\tilde{l}}(\kappa)
\overline{Y^{m^{\prime}}_{l+\tilde{l}-2k}}
(\kappa)d\kappa\Big|^2.$$

\begin{lemma}\label{cr}
For $0\leq k\leq\min(l,\tilde{l})$,\,$|m^{\prime}|\leq l+\tilde{l}-2k$,\,we have
\begin{align}\label{cr2}
\begin{split}
&\sum_{|m|\leq l}\sum_{|\tilde{m}|\leq \tilde{l}}\left|\int_{\mathbb{S}^2_{\kappa}}G^{m,\tilde{m}}_{n,\tilde{n},l,\tilde{l}}(\kappa)
\overline{Y^{m^{\prime}}_{l+\tilde{l}-2k}}
(\kappa)d\kappa\right|^2
\\
&\lesssim
\frac{\tilde{l}\sqrt{l}}{l+\tilde{l}-2k+1}
\left(
\int_0^{\frac{\pi}{4}}
  \beta(\theta)(\sin\theta)^{2n+l}(\cos\theta)^{2\tilde{n}+\tilde{l}}
d\theta
\right)^2.
\end{split}
\end{align}
\end{lemma}

\begin{proof}
For $0\leq\,k\leq\min(l,\tilde{l})$
and $|m^{\prime}|\leq l+\tilde{l}-2k$,
we deduce from Proposition \ref{prop reduc} that
{\color{black}
\begin{align*}
&\sum_{|m|\leq l}
\sum_{|\tilde{m}|\leq\tilde{l}}
\left|
\int_{\mathbb{S}^2_{\kappa}}
  G^{m,\tilde{m}}_{n,\tilde{n},l,\tilde{l}}(\kappa)
  \overline{Y^{m^{\prime}}_{l+\tilde{l}-2k}}(\kappa)
d\kappa
\right|^2
\\
&=\sum_{|q|\leq\min(l,\tilde{l})}
  \left(
  \left(\frac{4\pi}{2(l+\tilde{l}-2k)+1}\right)^{\frac12}
  G^{q,-q}_{n,\tilde{n},l,\tilde{l}}(e_1)
  \right)
  \,
  \left(\int_{\mathbb{S}^2_{\kappa}}
    G^{q,-q}_{n,\tilde{n},l,\tilde{l}}(\kappa)
    \overline{Y^{0}_{l+\tilde{l}-2k}}(\kappa)
  d\kappa\right),
\end{align*}
where $e_1=(1,0,0)$ and
\begin{align*}
G^{q,-q}_{n,\tilde{n},l,\tilde{l}}(e_1)
&=(-1)^q
{
\left(\frac{2l+1}
           {4\pi}
\right)^{\frac12}
\left(\frac{2\tilde{l}+1}
           {4\pi}
\right)^{\frac12}
\left(\frac{(l-|q|)!}{(l+|q|)!}\right)^{\frac12}
\left(\frac{(\tilde{l}-|q|)!}
    {(\tilde{l}+|q|)!}\right)^{\frac12}
} \,\,
\\
&\quad\times\int_{|\theta|\leq\frac{\pi}{4}}
  \beta(\theta)
  (\sin\theta)^{2n+l}
  (\cos\theta)^{2\tilde{n}+\tilde{l}}
  P^{|q|}_{l}(\sin\theta)
  P^{|q|}_{\tilde{l}}(\cos\theta)
d\theta.
\end{align*}
It follows from the Cauchy-Schwarz inequality that
\begin{align*}
&\sum_{|m|\leq l}\sum_{|\tilde{m}|\leq\tilde{l}}
\left|
\int_{\mathbb{S}^2_{\kappa}}
  G^{m,\tilde{m}}_{n,\tilde{n},l,\tilde{l}}(\kappa)
  \overline{Y^{m^{\prime}}_{l+\tilde{l}-2k}}
    (\kappa)d\kappa
\right|^2\nonumber\\
&\leq\Bigg(\sum_{|q|\leq\min(l,\tilde{l})}
\frac{4\pi}{2(l+\tilde{l}-2k)+1}
  \Big|G^{q,-q}_{n,\tilde{n},l,\tilde{l}}(e_1)
  \Big|^2\Bigg)^{\frac{1}{2}}\,\Bigg(\sum_{|q|\leq\min(l,\tilde{l})}\Big|\int_{\mathbb{S}^2_{\kappa}}
    G^{q,-q}_{n,\tilde{n},l,\tilde{l}}(\kappa)
    \overline{Y^{0}_{l+\tilde{l}-2k}}(\kappa)
  d\kappa\Big|^2\Bigg)^{\frac{1}{2}}.
\end{align*}
We observe from \eqref{GYm=Gm0} that, for any $|m'|\leq l+\tilde{l}-2k$,
$$
\sum_{|m|\leq l}
\sum_{|\tilde{m}|\leq\tilde{l}}
\left|
\int_{\mathbb{S}^2_{\kappa}}
  G^{m,\tilde{m}}_{n,\tilde{n},l,\tilde{l}}(\kappa)
  \overline{Y^{m^{\prime}}_{l+\tilde{l}-2k}}(\kappa)
d\kappa
\right|^2
=\sum_{|q|\leq \min(l,\tilde{l})}
\left|
\int_{\mathbb{S}^2_{\kappa}}
  G^{q,-q}_{n,\tilde{n},l,\tilde{l}}(\kappa)
  \overline{Y^{0}_{l+\tilde{l}-2k}}(\kappa)
d\kappa
\right|^2,
$$
then
\begin{align*}
&\sum_{|m|\leq l}\sum_{|\tilde{m}|\leq\tilde{l}}
\left|
\int_{\mathbb{S}^2_{\kappa}}
  G^{m,\tilde{m}}_{n,\tilde{n},l,\tilde{l}}(\kappa)
  \overline{Y^{m^{\prime}}_{l+\tilde{l}-2k}}
    (\kappa)d\kappa
\right|^2\nonumber\\
&\leq\sum_{|q|\leq\min(l,\tilde{l})}
\frac{4\pi}{2(l+\tilde{l}-2k)+1}
  \Bigg|G^{q,-q}_{n,\tilde{n},l,\tilde{l}}(e_1)
  \Bigg|^2\\
  &=\frac{(2l+1)(2\tilde{l}+1)}{4\pi[2(l+\tilde{l}-2k)+1]}
\sum_{|q|\leq\min(l,\tilde{l})}
\frac{(l-|q|)!}{(l+|q|)!}
\frac{(\tilde{l}-|q|)!}
    {(\tilde{l}+|q|)!}\\
&\quad\times\left(\int_{|\theta|\leq\frac{\pi}{4}}
  \beta(\theta)
  (\sin\theta)^{2n+l}
  (\cos\theta)^{2\tilde{n}+\tilde{l}}
  P^{|q|}_{l}(\sin\theta)
  P^{|q|}_{\tilde{l}}(\cos\theta)
d\theta\right)^2
\end{align*}
Using the Cauchy-Schwarz inequality, we have
\begin{align*}
&\sum_{|q|\leq\min(l,\tilde{l})}
\frac{(l-|q|)!}{(l+|q|)!}
\frac{(\tilde{l}-|q|)!}
    {(\tilde{l}+|q|)!}\left(\int_{|\theta|\leq\frac{\pi}{4}}
  \beta(\theta)
  (\sin\theta)^{2n+l}
  (\cos\theta)^{2\tilde{n}+\tilde{l}}
  P^{|q|}_{l}(\sin\theta)
  P^{|q|}_{\tilde{l}}(\cos\theta)
d\theta\right)^2\\
&\leq\sum_{|q|\leq\min(l,\tilde{l})}
  \frac{(l-|q|)!}{(l+|q|)!}\frac{(\tilde{l}-|q|)!}{(\tilde{l}+|q|)!}
\left(
\int_{|\theta|\leq\frac{\pi}{4}}
 \Big| \beta(\theta)(\sin\theta)^{2n+l}
  (\cos\theta)^{2\tilde{n}+\tilde{l}}\Big|
d\theta
\right)
\\
&\qquad\times
  \Bigg(
\int_{|\theta|\leq\frac{\pi}{4}}
  \Big|\beta(\theta)(\sin\theta)^{2n+l}(\cos\theta)^{2\tilde{n}+\tilde{l}}\Big|
\Bigg[\Big|P^{|q|}_l(\sin\theta)\Big|^2\Big|P^{|q|}_{\tilde{l}}(\cos\theta)\Big|^2
\Bigg]
d\theta
\Bigg)\\
&=4
\left(
\int^{\frac{\pi}{4}}_0
  \beta(\theta)(\sin\theta)^{2n+l}
  (\cos\theta)^{2\tilde{n}+\tilde{l}}
d\theta
\right)\times\Bigg(
\int^{\frac{\pi}{4}}_0
\beta(\theta)(\sin\theta)^{2n+l}(\cos\theta)^{2\tilde{n}+\tilde{l}}
\\
&\qquad\times
\Bigg[\sum_{|q|\leq\min(l,\tilde{l})}
  \frac{(l-|q|)!}{(l+|q|)!}\frac{(\tilde{l}-|q|)!}{(\tilde{l}+|q|)!}\Big|P^{|q|}_l(\sin\theta)\Big|^2\Big|P^{|q|}_{\tilde{l}}(\cos\theta)\Big|^2
\Bigg]
d\theta
\Bigg)
\end{align*}
For $\theta\in[0,\frac{\pi}{4}]$, by using the addition theorem \eqref{addition theorem} twice times,
\begin{align*}
  &P_l\big((\sin\theta)^2+(\cos\theta)^2\cos\phi\big)=\sum_{|q|\leq l}\frac{(l-|q|)!}{(l+|q|)!}
  P^{|q|}_l(\sin\theta)P^{|q|}_l(\sin\theta)e^{iq\phi};\\
  &P_{\tilde{l}}\big((\cos\theta)^2+(\sin\theta)^2\cos\phi\big)
  =\sum_{|\tilde{q}|\leq \tilde{l}}\frac{(\tilde{l}-|\tilde{q}|)!}{(\tilde{l}+|\tilde{q}|)!}
  P^{|\tilde{q}|}_{\tilde{l}}(\cos\theta)P^{|\tilde{q}|}_{\tilde{l}}(\cos\theta)e^{i\tilde{q}\phi},
\end{align*}
we obtain that,
\begin{align*}
  &\int^{2\pi}_{0}P_l\big((\sin\theta)^2+(\cos\theta)^2\cos\phi\big)P_{\tilde{l}}\big((\cos\theta)^2
  +(\sin\theta)^2\cos\phi\big)\frac{ d\phi}{2\pi}\\
  &=\sum_{|q|\leq\min(l,\tilde{l})}
\frac{(l-|q|)!}{(l+|q|)!}
\frac{(\tilde{l}-|q|)!}
    {(\tilde{l}+|q|)!}\Big|P^{|q|}_{l}(\sin\theta)\Big|^2
  \Big|P^{|q|}_{\tilde{l}}(\cos\theta)\Big|^2.
\end{align*}
It follows that
  \begin{align}\label{bb}
&\sum_{|m|\leq l}\sum_{|\tilde{m}|\leq\tilde{l}}
\left|
\int_{\mathbb{S}^2_{\kappa}}
  G^{m,\tilde{m}}_{n,\tilde{n},l,\tilde{l}}(\kappa)
  \overline{Y^{m^{\prime}}_{l+\tilde{l}-2k}}
    (\kappa)d\kappa
\right|^2\nonumber\\
&\leq\frac{(2l+1)(2\tilde{l}+1)}{\pi[2(l+\tilde{l}-2k)+1]}\nonumber\\
&\quad\times
\left(
\int^{\frac{\pi}{4}}_0
  \beta(\theta)(\sin\theta)^{2n+l}
  (\cos\theta)^{2\tilde{n}+\tilde{l}}
d\theta
\right)\times
\Bigg(
\int^{\frac{\pi}{4}}_0
  \beta(\theta)(\sin\theta)^{2n+l}
  (\cos\theta)^{2\tilde{n}+\tilde{l}}
\nonumber\\
&\quad\times\,
\left[\int^{2\pi}_0
  P_l\big((\sin\theta)^2+(\cos\theta)^2\cos\phi\big)
  P_{\tilde{l}}\big((\cos\theta)^2
  +(\sin\theta)^2\cos\phi\big)
\frac{d\phi}{2\pi}
\right]
d\theta
\Bigg).
\end{align}
From the formula (14) of Sec.10.3 in Chap.III in \cite{San}
$$
|\sqrt{l}\sqrt[4]{1-x^2}P_l(x)|
\leq
4\sqrt{\frac{2}{\pi}},\quad\forall\,-1\leq~x\leq1,
$$
then for $l\geq1$, we have
\begin{align*}
\Big|P_l\big((\sin\theta)^2+(\cos\theta)^2\cos\phi\big)\Big|\leq 4\sqrt{\frac{2}{\pi l}}\frac{1}{\sqrt[4]{1-\big((\sin\theta)^2+(\cos\theta)^2\cos\phi\big)^2}}.
\end{align*}
Since
\begin{align*}
&1-\big((\sin\theta)^2+(\cos\theta)^2\cos\phi\big)^2
\\
&=(1+(\sin\theta)^2+(\cos\theta)^2\cos\phi)\Big(1-\big((\sin\theta)^2+(\cos\theta)^2\cos\phi\big)\Big)\\
&\geq\Big((\cos\theta)^2(1+\cos\phi)\Big)\Big((\cos\theta)^2(1-\cos\phi)\Big)\\
&=(\cos\theta)^4(1-(\cos\phi)^2),
\end{align*}
we can estimate
\begin{align*}
\Big|P_l\big((\sin\theta)^2+(\cos\theta)^2\cos\phi\big)\Big|
\leq 4\sqrt{\frac{2}{\pi l}}\frac{1}{|\cos\theta|}\frac{1}{\sqrt[4]{1-(\cos\phi)^2}}.
\end{align*}
}
Recall that $|P_{\tilde{l}}(x)|\leq1$ for $|x|\leq1$,\,we can derive
\begin{align*}
&
\int^{2\pi}_0P_l
\Big(
  (\sin\theta)^2+(\cos\theta)^2\cos\phi\Big)
  P_{\tilde{l}}\Big((\cos\theta)^2+(\sin\theta)^2
  \cos\phi\Big)\frac{d\phi }{2\pi}\\
&\lesssim
\frac{1}{\sqrt{l}}\frac{1}{|\cos\theta|}
\int^{2\pi}_0\frac{1}{\sqrt[4]{1-(\cos\phi)^2}}d\phi\\
&\lesssim
\frac{1}{\sqrt{l}}\frac{1}{|\cos\theta|}.
\end{align*}
For $0\leq\theta\leq\frac{\pi}{4}$,
we have $\cos\theta\geq\frac{\sqrt{2}}{2}$,
substitute the above result into the formula \eqref{bb}, we conclude that,\,for $l\geq1,\tilde{l}\geq1$ with $0\leq\,k\leq\min(l,\tilde{l})$,
\begin{align*}
&\sum_{|m|\leq l}\sum_{|\tilde{m}|
\leq \tilde{l}}
\left|
\int_{\mathbb{S}^2_{\kappa}}
  G^{m,\tilde{m}}_{n,\tilde{n},l,\tilde{l}}(\kappa)
  \overline{Y^{m^{\prime}}_{l+\tilde{l}-2k}}(\kappa)
d\kappa
\right|^2\\
&\lesssim{\color{black}\frac{(2l+1)(2\tilde{l}+1)}{[2(l+\tilde{l}-2k)+1]}\frac{1}{\sqrt{l}}\left(
\int_0^{\frac{\pi}{4}}
  \beta(\theta)(\sin\theta)^{2n+l}(\cos\theta)^{2\tilde{n}+\tilde{l}}
d\theta
\right)^2}\\
&\lesssim\frac{\tilde{l}\sqrt{l}}{l+\tilde{l}-2k+1}
\left(
\int_0^{\frac{\pi}{4}}
  \beta(\theta)(\sin\theta)^{2n+l}(\cos\theta)^{2\tilde{n}+\tilde{l}}
d\theta
\right)^2.
\end{align*}
This ends the proof of \eqref{cr2}.
\end{proof}

For $l\geq1$,\,$\tilde{l}\geq1$ with $0\leq\,k\leq\min(l,\tilde{l})$,
we denote $\lambda^k_{n,\tilde{n},\tilde{l},l}$
\begin{align}\label{lambdadef}
\begin{split}
\lambda^k_{n,\tilde{n},\tilde{l},l}
&=
\frac{\tilde{l}\sqrt{l}}{l+\tilde{l}-2k+1}
\left(
  \frac{A_{\tilde{n},\tilde{l}}A_{n,l}}
       {A_{n+\tilde{n}+k,l+\tilde{l}-2k}}
\right)^2\\
&\qquad\times
\left(
\int_0^{\frac{\pi}{4}}
  \beta(\theta)(\sin\theta)^{2n+l}(\cos\theta)^{2\tilde{n}+\tilde{l}}
d\theta
\right)^2.
\end{split}
\end{align}
It follows from \eqref{present} and \eqref{cr2} that,
for $0\leq\,k\leq\min(l,\tilde{l})$,\,with $|m^{\prime}|\leq l+\tilde{l}-2k$,
\begin{align*}
\sum_{|m|\leq~l}\sum_{|\tilde{m}|\leq~\tilde{l}} \left|\mu^{m,\tilde{m},m^{\prime}}_{n,\tilde{n},l,\tilde{l},k}\right|^2\lesssim\lambda^k_{n,\tilde{n},\tilde{l},l}.
\end{align*}
Then we obtain
\begin{align*}
\sum_{\substack{n+\tilde{n}+k=n^{\star}\\n+l\geq2,\tilde{n}+\tilde{l}\geq2\\n\geq0,\tilde{n}\geq0}}
&\sum_{\substack{l+l-2k=l^{\star}\\l\geq1,\tilde{l}\geq1\\0\leq\,k\,\leq\min(l,\tilde{l})}}
\left(\sum_{|m|\leq~l}\sum_{|\tilde{m}|\leq~\tilde{l}} \frac{|\mu^{m,\tilde{m},m^{\star}}_{n,\tilde{n},l,\tilde{l},k}|^2}{\lambda_{\tilde{n},\tilde{l}}}\right)
 \\
&\lesssim\sum_{\substack{n+\tilde{n}+k=n^{\star}\\n+l\geq2,\tilde{n}+\tilde{l}\geq2\\n\geq0,\tilde{n}\geq0}}
\sum_{\substack{l+l-2k=l^{\star}\\l\geq1,\tilde{l}\geq1\\0\leq\,k\,\leq\min(l,\tilde{l})}}
\frac{\lambda^k_{n,\tilde{n},\tilde{l},l}}{\lambda_{\tilde{n},\tilde{l}}}.
\end{align*}
The proof of $3)$ in Proposition \ref{est} is reduced to prove
\begin{align}\label{es2}
\sum_{\substack{n+\tilde{n}+k=n^{\star}\\n+l\geq2,\tilde{n}+\tilde{l}\geq2\\n\geq0,\tilde{n}\geq0}}
\sum_{\substack{l+l-2k=l^{\star}\\l\geq1,\tilde{l}\geq1\\0\leq\,k\,\leq\min(l,\tilde{l})}}
\frac{\lambda^k_{n,\tilde{n},\tilde{l},l}}{\lambda_{\tilde{n},\tilde{l}}}\leq\,C\lambda_{n^{\star},l^{\star}}.
\end{align}

\begin{lemma}\label{est-b}
For $n,\tilde{n},\tilde{l},l\in\mathbb{N}$\,with $n+l\geq2$ and $\tilde{n}+\tilde{l}\geq2$,\,let $s_0=\min(1-s,s)$, we have
\begin{equation}\label{estlambda0-b}
\lambda^{0}_{n,\tilde{n},l,\tilde{l}}\lesssim
\frac{\tilde{l}\sqrt{l}}{l+\tilde{l}+1}\frac{(\tilde{n}+\tilde{l})^{2s+s_0}}{(\tilde{n}+1)^{s_0}
(n+1)^{1-s_0}(n+l)^{\frac{3}{2}+2s+s_0}}
.\,\
\end{equation}
In addition, for $1\leq\,k\leq~\min(l,\tilde{l})$, we have the following estimate
\begin{equation}\label{estlambdak-b}
\lambda^{k}_{n,\tilde{n},l,\tilde{l}}\lesssim\frac{\tilde{l}\sqrt{l}}{l+\tilde{l}-2k+1}\frac{\tilde{n}^s(\tilde{n}+
\tilde{l})^s}{(n+1)^s(n+l+1)^{\frac{5}{2}+s}}.
\end{equation}
\end{lemma}
\begin{remark}
We divide the proof of the estimate of $\lambda^{k}_{n,\tilde{n},l,\tilde{l}}$ into two cases: $k=0$ and $k\geq1$.
{\color{black}The reason is that, when we estimate $\lambda^{k}_{n,\tilde{n},l,\tilde{l}}$ in case $k\geq1$, there is a term
$$\frac{(n+\tilde{n}+k)!\Gamma(n+\tilde{n}+l+\tilde{l}-k+\frac{3}{2})}
{(n+\tilde{n}+1)!\Gamma(n+\tilde{n}+l+\tilde{l}+\frac{1}{2})}\leq1;$$
when $k=0$ and $l+\tilde{l}\gg\,n+\tilde{n}$, by using the recurrence formula $\Gamma(x+1)=x\Gamma(x)$, this term satisfies
$$\frac{(n+\tilde{n}+k)!\Gamma(n+\tilde{n}+l+\tilde{l}-k+\frac{3}{2})}
{(n+\tilde{n}+1)!\Gamma(n+\tilde{n}+l+\tilde{l}+\frac{1}{2})}\Bigg|_{k=0}=\frac{n+\tilde{n}+l+\tilde{l}+\frac{1}{2}}{n+\tilde{n}+1}\gg1.$$
}
\end{remark}
\begin{proof} Firstly, we consider the case $k\geq1$.
By using the Cauchy-Schwarz inequality and the Beta Function
\eqref{Beta function}
we derive that, for $n+l\geq2,$
\begin{align*}
&|\int^{\frac{\pi}{4}}_{0}\beta(\theta)(\sin\theta)^{2n+l}(\cos\theta)^{2\tilde{n}+\tilde{l}}d\theta|^2\approx|\int^{\frac{1}{2}}_0t^{n+\frac{l}{2}-1-s}(1-t)^{\tilde{n}+\frac{\tilde{l}}{2}+s}dt|^2\\
&\leq\frac{\Gamma(n+1-s)\Gamma(\tilde{n}+1+s)}{(n+\tilde{n}+1)!}
\frac{\Gamma(n+l-1-s)\Gamma(\tilde{n}+\tilde{l}+\frac{3}{2}+s)}{\Gamma(n+\tilde{n}+l+\tilde{l}+\frac{1}{2})}.
\end{align*}
Then, we can express
\begin{align*}
\lambda^k_{n,\tilde{n},l,\tilde{l}}
&\lesssim\frac{\tilde{l}\sqrt{l}}{l+\tilde{l}-2k+1}\frac{(n+\tilde{n}+k)!\Gamma(n+\tilde{n}+l+\tilde{l}-k+\frac{3}{2})}
{n!\Gamma(n+l+\frac{3}{2})\tilde{n}!\Gamma(\tilde{n}+\tilde{l}+\frac{3}{2})}\\
&\qquad\times\frac{\Gamma(n+1-s)\Gamma(\tilde{n}+1+s)}{(n+\tilde{n}+1)!}
\frac{\Gamma(n+l-1-s)\Gamma(\tilde{n}+\tilde{l}+\frac{3}{2}+s)}{\Gamma(n+\tilde{n}+l+\tilde{l}+\frac{1}{2})}\\
&=\frac{\tilde{l}\sqrt{l}}{l+\tilde{l}-2k+1}\frac{(n+\tilde{n}+k)!\Gamma(n+\tilde{n}+l+\tilde{l}-k+\frac{3}{2})}{(n+\tilde{n}+1)!\Gamma(n+\tilde{n}+l+\tilde{l}+\frac{1}{2})}\\
&\qquad\times
\frac{\Gamma(n+1-s)\Gamma(\tilde{n}+1+s)\Gamma(n+l-1-s)\Gamma(\tilde{n}+\tilde{l}+\frac{3}{2}+s)
}
{n!\Gamma(n+l+\frac{3}{2})\tilde{n}!\Gamma(\tilde{n}+\tilde{l}+\frac{3}{2})}.
\end{align*}
We deduce from the formula \eqref{Gamma function2} with $x=n+1$, $a=-s$, $b=0$,
$$
\frac{\Gamma(n+1-s+1)}{(n+1)!}\lesssim\frac{1}{(n+1-s)^s},
$$
and the recurrence formula $\Gamma(n+1-s)=\frac{1}{n+1-s}\Gamma(n+2-s)$,
\begin{align*}
\frac{\Gamma(n+1-s)}{n!}=\frac{n+1}{(n+1-s)}\frac{\Gamma(n+1-s+1)}
{(n+1)!}\lesssim\frac{1}{(n+1)^s}.
\end{align*}
Using $x=n+l-1$, $a=-s$, $b=\frac{3}{2}$ in \eqref{Gamma function2},
$$\frac{\Gamma(n+l-s)}{\Gamma(n+l+\frac{3}{2})}=\frac{\Gamma(n+l-1-s+1)}{\Gamma(n+l-1+\frac{3}{2}+1)}\lesssim\frac{1}{(n+l-1)^{\frac{3}{2}+s}},$$
and {\color{black} the} recurrence formula $\Gamma(n+l-1-s)=\frac{\Gamma(n+l-s)}{n+l-1-s}$,
\begin{align*}
\frac{\Gamma(n+l-1-s)}{\Gamma(n+l+\frac{3}{2})}=\frac{1}{n+l-1-s}\frac{\Gamma(n+l-s)}
{\Gamma(n+l+\frac{3}{2})}\lesssim\frac{1}{(n+l+1)^{\frac{5}{2}+s}}.
\end{align*}
Using $x=\tilde{n}+1$, $a=s$,\,$b=0$ in \eqref{Gamma function2}, we have
\begin{align*}
\frac{\Gamma(\tilde{n}+2+s)}{(\tilde{n}+1)!}\lesssim(\tilde{n}+1+s)^s.
\end{align*}
{\color{black}
By using the recurrence formula $\Gamma(x+1)=x\Gamma(x)$,
$$\frac{\Gamma(\tilde{n}+1+s)}{\tilde{n}!}=\frac{\tilde{n}+1}{\tilde{n}+1+s}\frac{\Gamma(\tilde{n}+2+s)}{(\tilde{n}+1)!}\lesssim \tilde{n}^s.$$
}
Using $x=\tilde{n}+\tilde{l}+\frac{1}{2}$,\,$a=s$,\,$b=0$ in \eqref{Gamma function2}, and $\tilde{n}+\tilde{l}\geq2$,
\begin{align*}
\frac{\Gamma(\tilde{n}+\tilde{l}+\frac{3}{2}+s)}{\Gamma(\tilde{n}+\tilde{l}+\frac{3}{2})}
\lesssim(\tilde{n}+\tilde{l}+\frac{1}{2}+s)^s\lesssim(\tilde{n}+\tilde{l})^s.
\end{align*}
Therefore, we obtain that, $n+l\geq2$, $\tilde{n}+\tilde{l}\geq2$
\begin{align*}
\lambda^k_{n,\tilde{n},l,\tilde{l}}&\lesssim\frac{\tilde{l}\sqrt{l}}{l+\tilde{l}-2k+1}\frac{\tilde{n}^s(\tilde{n}+\tilde{l})^s}{(n+1)^s(n+l+1)^{\frac{5}{2}+s}}
\\
&\qquad\times\frac{(n+\tilde{n}+k)!\Gamma(n+\tilde{n}+l+\tilde{l}-k+\frac{3}{2})}
{(n+\tilde{n}+1)!\Gamma(n+\tilde{n}+l+\tilde{l}+\frac{1}{2})}.
\end{align*}
{\color{black}
When $k=1$, we observe that,
$$\frac{(n+\tilde{n}+k)!\Gamma(n+\tilde{n}+l+\tilde{l}-k+\frac{3}{2})}
{(n+\tilde{n}+1)!\Gamma(n+\tilde{n}+l+\tilde{l}+\frac{1}{2})}=1.$$
Now consider the case $k\geq2$, we use again the recurrence formula $\Gamma(x+1)=x\Gamma(x)$,
\begin{align*}
&
\frac{(n+\tilde{n}+k)!\Gamma(n+\tilde{n}+l+\tilde{l}-k+\frac{3}{2})}
{(n+\tilde{n}+1)!\Gamma(n+\tilde{n}+l+\tilde{l}+\frac{1}{2})}\\
&=\frac{(n+\tilde{n}+k)(n+\tilde{n}+k-1)\times\cdots\times(n+\tilde{n}+2)}
{(n+\tilde{n}+l+\tilde{l}-\frac{1}{2})(n+\tilde{n}+l+\tilde{l}-\frac{3}{2})\times\cdots\times(n+\tilde{n}+l+\tilde{l}-k+\frac{3}{2})}\\
&=\prod^{k}_{j=2}\frac{n+\tilde{n}+j}{n+\tilde{n}+l+\tilde{l}-k+j-\frac{1}{2}}.
\end{align*}
Since $2\leq k\leq min(l,\tilde{l})$, one has,
$$l+\tilde{l}-k-\frac{1}{2}\geq k-\frac{1}{2}\geq\frac{3}{2}.$$
This shows that, for any $2\leq j\leq k$,
$$\frac{n+\tilde{n}+j}{n+\tilde{n}+l+\tilde{l}-k+j-\frac{1}{2}}\leq \frac{n+\tilde{n}+j}{n+\tilde{n}+j+\frac{3}{2}}<1.$$
Therefore, we conclude, for $k\geq1$,
$$\frac{(n+\tilde{n}+k)!\Gamma(n+\tilde{n}+l+\tilde{l}-k+\frac{3}{2})}
{(n+\tilde{n}+1)!\Gamma(n+\tilde{n}+l+\tilde{l}+\frac{1}{2})}\leq1.$$}
We obtain the formula \eqref{estlambdak-b}.

For the estimate \eqref{estlambda0-b},\,we assume that $n+l\geq2$,\,$\tilde{n}+\tilde{l}\geq2$\,and \,$s_0=\min(1-s,s)$.\,\,
By using the Cauchy-Schwarz inequality and the Beta Function \eqref{Beta function},
we obtain
\begin{align*}
&|\int^{\frac{\pi}{4}}_{0}\beta(\theta)(\sin\theta)^{2n+l}(\cos\theta)^{2\tilde{n}+\tilde{l}}d\theta|^2
\approx|\int^{\frac{1}{2}}_{0}t^{n+\frac{l}{2}-1-s}(1-t)^{\tilde{n}+\frac{\tilde{l}}{2}+\frac{1}{4}+s}dt|^2\\
&\leq\int^{\frac{1}{2}}_0t^{n-1+s_0}(1-t)^{\tilde{n}-s_0}dt\times\int^{\frac{1}{2}}_0t^{n+l-2s-1-s_0}
(1-t)^{\tilde{n}+\tilde{l}+\frac{1}{2}+2s+s_0}dt\\
&\leq\frac{\Gamma(n+s_0)\Gamma(\tilde{n}+1-s_0)}{(n+\tilde{n})!}\frac{\Gamma(n+l-2s-s_0)\Gamma(\tilde{n}+\tilde{l}+\frac{3}{2}+2s+s_0)}{\Gamma(n+\tilde{n}+l+\tilde{l}+\frac{3}{2})}.
\end{align*}
Therefore,
\begin{align*}
&\lambda^0_{n,\tilde{n},l,\tilde{l}}=\frac{\tilde{l}\sqrt{l}}{l+\tilde{l}+1}
\frac{(n+\tilde{n})!\Gamma(n+\tilde{n}+l+\tilde{l}+\frac{3}{2})}{n!\tilde{n}!\Gamma(n+l+\frac{3}{2})\Gamma(\tilde{n}+\tilde{l}+\frac{3}{2})}\\
&\qquad\times\frac{\Gamma(n+s_0)\Gamma(\tilde{n}+1-s_0)}{(n+\tilde{n})!}\frac{\Gamma(n+l-2s-s_0)\Gamma(\tilde{n}+\tilde{l}+\frac{3}{2}+2s+s_0)}{\Gamma(n+\tilde{n}+l+\tilde{l}+\frac{3}{2})}\\
&\lesssim\frac{\tilde{l}\sqrt{l}}{l+\tilde{l}+1}
\frac{\Gamma(n+s_0)\Gamma(\tilde{n}+1-s_0)\Gamma(n+l-2s-s_0)
\Gamma(\tilde{n}+\tilde{l}+\frac{3}{2}+2s+s_0)}{n!\tilde{n}!\Gamma(n+l+\frac{3}{2})\Gamma(\tilde{n}+\tilde{l}+\frac{3}{2})}.
\end{align*}
{\color{black}
Applying the estimate \eqref{Gamma function2} with $x=n+1, a=s_0-1,b=0$,
$$\frac{\Gamma(n+1+s_0)}{(n+1)!}\lesssim\frac{1}{(n+s_0)^{1-s_0}},$$
by using the recurrence formula $\Gamma(x+1)=x\Gamma(x),$
$$\frac{\Gamma(n+s_0)}{n!}=\frac{n+1}{n+s_0}\frac{\Gamma(n+1+s_0)}{(n+1)!}\lesssim\frac{n+1}{n+s_0}\frac{1}{(n+s_0)^{1-s_0}}\lesssim\frac{1}{(n+1)^{1-s_0}}.$$
Applying the estimate \eqref{Gamma function2} with $x=\tilde{n}+1, a=-s_0, b=0$,
$$\frac{\Gamma(\tilde{n}+2-s_0)}{(\tilde{n}+1)!}\lesssim\frac{1}{(\tilde{n}+1-s_0)^{s_0}},$$
by using the recurrence formula $\Gamma(x+1)=x\Gamma(x),$
$$\frac{\Gamma(\tilde{n}+1-s_0)}{\tilde{n}!}
=\frac{\tilde{n}+1}{\tilde{n}+1-s_0}\frac{\Gamma(\tilde{n}+2-s_0)}{(\tilde{n}+1)!}\lesssim\frac{1}{(\tilde{n}+1-s_0)^{s_0}}
\lesssim\frac{1}{(\tilde{n}+1)^{s_0}}.$$
Since $s_0=\min(s,1-s)$, we have $2s+s_0\leq 1-s+2s=1+s\leq2$.   Using the estimate \eqref{Gamma function2} with $x=n+l, a=-2s-s_0+1, b=\frac{1}{2}$, considering that $n+l\geq2$, then $x+a\geq1$, $x+b\geq1$, we can derive that,
$$\frac{\Gamma(n+l+2-2s-s_0)}{\Gamma(n+l+\frac{3}{2})}\lesssim \frac{1}{(n+l+1-2s-s_0)^{2s+s_0-\frac{1}{2}}}.$$
By using the recurrence formula $\Gamma(x+1)=x\Gamma(x)$ that
\begin{align*}
&\frac{\Gamma(n+l-2s-s_0)}{\Gamma(n+l+\frac{3}{2})}\\
&=\frac{1}{(n+l-2s-s_0)(n+l+1-2s-s_0)}\frac{\Gamma(n+l+2-2s-s_0)}{\Gamma(n+l+\frac{3}{2})}\\
&\lesssim \frac{1}{(n+l+1-2s-s_0)^{2s+s_0+\frac{3}{2}}}\lesssim\frac{1}{(n+l)^{2s+s_0+\frac{3}{2}}}.
\end{align*}
Finally, applying the estimate \eqref{Gamma function2} with $x=\tilde{n}+\tilde{l}+\frac{1}{2}, a=2s+s_0, b=0$
$$\frac{\Gamma(\tilde{n}+\tilde{l}+\frac{3}{2}+2s+s_0)}{\Gamma(\tilde{n}+\tilde{l}+\frac{3}{2})}\lesssim(\tilde{n}+\tilde{l})^{2s+s_0}.$$
Combining the above four estimates together,} we derive
$$\lambda^{0}_{n,\tilde{n},l,\tilde{l}}\lesssim
\frac{\tilde{l}\sqrt{l}}{l+\tilde{l}+1}\frac{(\tilde{n}+\tilde{l})^{2s+s_0}}{(\tilde{n}+1)^{s_0}
(n+1)^{1-s_0}(n+l)^{\frac{3}{2}+2s+s_0}}.
$$
This concludes the proof of Lemma \ref{est-b}.
\end{proof}

The estimate in \eqref{estlambdak-b} is not enough
accurate in proof of $3)$
in Proposition \ref{est}.\,\,To this end, we provide a more optimal estimate of $\lambda^k_{n,\tilde{n},l,\tilde{l}}$ for large $k$ in the following Lemma.
\begin{lemma}\label{relambda}
{\color{black}For any $0<\omega<1$ and $k\geq 20$, we have the following estimates
\begin{align*}
&(i) \,\,\,\, \lambda^k_{n,\tilde{n},l,\tilde{l}}
\lesssim
\frac{\tilde{l}\sqrt{l}}{l+\tilde{l}-2k+1}
\frac{(\tilde{n}+\tilde{l})^{2s}}{(n+l+1)^{\frac{5}{2}+2s}}e^{-\frac{1}{8}k^{\omega}}
\,\,\,
\text{when}\,n+l\leq\,k^{1-\omega}l;
\\
&(ii) \,\,\, \lambda^k_{n,\tilde{n},l,\tilde{l}}
\lesssim
\frac{\tilde{l}\sqrt{l}}{l+\tilde{l}-2k+1}
\frac{(\tilde{n}+\tilde{l})^{2s}}{(n+l+1)^{\frac{5}{2}+2s}}
\,\,\,\,\,\,\,\,\,\,\,\,\,\,
\text{when}\,n+l\geq\,k^{1-\omega}l.
\end{align*}
}
\end{lemma}

\begin{proof}
Recall from \eqref{lambdadef} that
\begin{align*}
\lambda^k_{n,\tilde{n},l,\tilde{l}}&\approx\frac{\tilde{l}\sqrt{l}}{l+\tilde{l}-2k+1}\frac{(n+\tilde{n}+k)!\Gamma(n+\tilde{n}+l+\tilde{l}-k+\frac{3}{2})}
{n!\Gamma(n+l+\frac{3}{2})\tilde{n}!\Gamma(\tilde{n}+\tilde{l}+\frac{3}{2})}\\
&\qquad\times\Big(\int^{\frac{1}{2}}_0t^{n+\frac{l}{2}-1-s}(1-t)^{\tilde{n}+\frac{\tilde{l}}{2}}dt\Big)^2.
\end{align*}
By using the Beta Function \eqref{Beta function}
and the Cauchy -Schwarz inequality,
\begin{align*}
&\Big(\int^{\frac{1}{2}}_0t^{n+\frac{l}{2}-1-s}
  (1-t)^{\tilde{n}+\frac{\tilde{l}}{2}}dt\Big)^2
\leq
2^{\frac{1}{2}+2s}
\Big(
\int^{\frac{1}{2}}_0
  t^{n+\frac{l}{2}-1-s}
  (1-t)^{\tilde{n}+\frac{\tilde{l}}{2}+\frac{1}{4}+s}
dt
\Big)^2
\\
&\leq
2^{\frac{1}{2}+2s}
\frac{\Gamma(n+\frac{k}{2})\Gamma(\tilde{n}+\frac{k}{2}+1)}{\Gamma(n+\tilde{n}+k+1)}\times
\frac{\Gamma(n+l-\frac{k}{2}-2s)
\Gamma(\tilde{n}+\tilde{l}-\frac{k}{2}+\frac{3}{2}+2s)}
{\Gamma(n+\tilde{n}+l+\tilde{l}-k+\frac{3}{2})},
\end{align*}
we obtain
\begin{align*}
\lambda^k_{n,\tilde{n},l,\tilde{l}}&\lesssim
\frac{\tilde{l}\sqrt{l}}{l+\tilde{l}-2k+1}
\frac{(n+\tilde{n}+k)!\Gamma(n+\tilde{n}+l+\tilde{l}-k+\frac{3}{2})}
{n!\Gamma(n+l+\frac{3}{2})\tilde{n}!\Gamma(\tilde{n}+\tilde{l}+\frac{3}{2})}
\\
&\qquad\times\frac{\Gamma(n+\frac{k}{2})\Gamma(\tilde{n}+\frac{k}{2}+1)}{\Gamma(n+\tilde{n}+k+1)}\frac{\Gamma(n+l-\frac{k}{2}-2s)\Gamma(\tilde{n}+\tilde{l}-\frac{k}{2}+\frac{3}{2}+2s)}
{\Gamma(n+\tilde{n}+l+\tilde{l}-k+\frac{3}{2})}\\
&=\frac{\sqrt{l}\tilde{l}\Gamma(n+\frac{k}{2})\Gamma(\tilde{n}+\frac{k}{2}+1)\Gamma(n+l-\frac{k}{2}-2s)\Gamma(\tilde{n}+\tilde{l}-\frac{k}{2}+\frac{3}{2}+2s)}
{(l+\tilde{l}-2k+1)n!\Gamma(n+l+\frac{3}{2})\tilde{n}!\Gamma(\tilde{n}+\tilde{l}+\frac{3}{2})}.
\end{align*}
{\color{black} In case $k\ge 20$,  remind that $l\geq\min(l,\tilde{l})\geq\,k$, we have},
$$
n+l-\frac{k}{2}-1\geq\frac{k}{2}-1\geq9,
$$
Let $x=n+l-\frac{k}{2}-1$, $a=-2s$, $b=\frac{5}{2}$ in formula \eqref{Gamma function2}, we then derive
\begin{align*}
\frac{\Gamma(n+l-\frac{k}{2}-2s)}{\Gamma(n+l-\frac{k}{2}+\frac{5}{2})}\lesssim
\frac{1}{(n+l-\frac{k}{2}-1-2s)^{\frac{5}{2}+2s}}\lesssim
\frac{1}{(n+l+1)^{\frac{5}{2}+2s}}.
\end{align*}
When we choose $x=\tilde{n}+\tilde{l}-\frac{k}{2}+\frac{1}{2}$,\,$a=2s$,\,$b=0$\,in formula \eqref{Gamma function2}, then
$$\frac{\Gamma(\tilde{n}+\tilde{l}-\frac{k}{2}+\frac{3}{2}+2s)}{\Gamma(\tilde{n}+\tilde{l}-\frac{k}{2}+\frac{3}{2})}\lesssim(\tilde{n}+\tilde{l})^{2s}.
$$
Therefore,  we can verify that for $k\ge 20$
\begin{align*}
\lambda^k_{n,\tilde{n},l,\tilde{l}}&\lesssim\frac{\tilde{l}\sqrt{l}}{l+\tilde{l}-2k+1}\frac{(\tilde{n}+\tilde{l})^{2s}}{(n+l+1)^{\frac{5}{2}+2s}}
\frac{\Gamma(n+\frac{k}{2})\Gamma(n+l-\frac{k}{2}+\frac{5}{2})}{n!\Gamma(n+l+\frac{3}{2})}\\
&\qquad\times
\frac{\Gamma(\tilde{n}+\frac{k}{2}+1)\Gamma(\tilde{n}+\tilde{l}-\frac{k}{2}+\frac{3}{2})}
{\tilde{n}!\Gamma(\tilde{n}+\tilde{l}+\frac{3}{2})}.
\end{align*}
{\color{black}
We claim that, for $k\ge 20$,
\begin{equation}\label{leq1}
\frac{\Gamma(\tilde{n}+\frac{k}{2}+1)\Gamma(\tilde{n}+\tilde{l}-\frac{k}{2}+\frac{3}{2})}
{\tilde{n}!\Gamma(\tilde{n}+\tilde{l}+\frac{3}{2})}\leq1.
\end{equation}
Indeed, if $k$ is even, by using the recurrence formula $\Gamma(x+1)=x\Gamma(x)$ and $\Gamma(\tilde{n}+\frac{k}{2}+1)=(\tilde{n}+\frac{k}{2})!$, one has,
\begin{align*}
&\frac{\Gamma(\tilde{n}+\frac{k}{2}+1)\Gamma(\tilde{n}+\tilde{l}-\frac{k}{2}+\frac{3}{2})}
{\tilde{n}!\Gamma(\tilde{n}+\tilde{l}+\frac{3}{2})}\\
&=\frac{(\tilde{n}+\frac{k}{2})(\tilde{n}+\frac{k}{2}-1)\times\cdots\times(\tilde{n}+1)}
{(\tilde{n}+\tilde{l}+\frac{1}{2})(\tilde{n}+\tilde{l}-\frac{1}{2})\times\cdots\times(\tilde{n}+\tilde{l}-\frac{k}{2}+\frac{3}{2})}\\
&=\prod^{[\frac{k}{2}]}_{j=1}\frac{\tilde{n}+\frac{k}{2}+1-j}{\tilde{n}+\tilde{l}+\frac{3}{2}-j}.
\end{align*}
For any $1\leq j\leq [\frac{k}{2}]$, recall that $\tilde{l}\geq k>\frac{k}{2}$, we obtain,
$$\frac{\tilde{n}+\frac{k}{2}+1-j}{\tilde{n}+\tilde{l}+\frac{3}{2}-j}\leq1.$$
Then the formula \eqref{leq1} holds.
If $k$ is odd, by using the recurrence formula $\Gamma(x+1)=x\Gamma(x)$ and $\Gamma(\tilde{n}+\tilde{l}-\frac{k}{2}+\frac{3}{2})
=(\tilde{n}+\tilde{l}-\frac{k}{2}+\frac{1}{2})!$,
\begin{align*}
&\frac{\Gamma(\tilde{n}+\frac{k}{2}+1)\Gamma(\tilde{n}+\tilde{l}-\frac{k}{2}+\frac{3}{2})}
{\tilde{n}!\Gamma(\tilde{n}+\tilde{l}+\frac{3}{2})}\\
&=\frac{(\tilde{n}+\tilde{l}-\frac{k}{2}+\frac{1}{2})(\tilde{n}+\tilde{l}-\frac{k}{2}-\frac{1}{2})\times\cdots\times(\tilde{n}+1)}
{(\tilde{n}+\tilde{l}+\frac{1}{2})(\tilde{n}+\tilde{l}-\frac{1}{2})\times\cdots\times(\tilde{n}+\frac{k}{2}+1)}\\
&=\prod^{\tilde{l}-\frac{k}{2}+\frac{1}{2}}_{j=1}\frac{\tilde{n}+\tilde{l}-\frac{k}{2}+\frac{3}{2}-j}{\tilde{n}+\tilde{l}+\frac{3}{2}-j}.
\end{align*}
We can observe directly that
$$\frac{\tilde{n}+\tilde{l}-\frac{k}{2}+\frac{3}{2}-j}{\tilde{n}+\tilde{l}+\frac{3}{2}-j}\leq1,$$
Therefore, we conclude that the formula \eqref{leq1} holds for all $k\ge 20$.
}
Thus, we  derive
\begin{align}\label{mid}
\lambda^k_{n,\tilde{n},l,\tilde{l}}\lesssim\frac{\tilde{l}\sqrt{l}}{l+\tilde{l}-2k+1}\frac{(\tilde{n}+\tilde{l})^{2s}}{(n+l+1)^{\frac{5}{2}+2s}}
\frac{\Gamma(n+\frac{k}{2})\Gamma(n+l-\frac{k}{2}+\frac{5}{2})}{n!\Gamma(n+l+\frac{3}{2})}.
\end{align}
Now consider the formula
$$\frac{\Gamma(n+\frac{k}{2})\Gamma(n+l-\frac{k}{2}+\frac{5}{2})}{n!\Gamma(n+l+\frac{3}{2})},$$
{\color{black}we claim that, for $k\ge20$,
\begin{align}\label{important}
&\frac{\Gamma(n+\frac{k}{2})\Gamma(n+l-\frac{k}{2}+\frac{5}{2})}{n!\Gamma(n+l+\frac{3}{2})}\leq\,e^{-\frac{lk}{8(n+l)}}.
\end{align}
}{\color{black}Firstly, we assume $k$ is even.
{\color{black} Reminding} that $k\leq\min(l,\tilde{l})$,  we can set $l=l_1+k$ with $l_1\geq0$.}
Using the recurrence formula $\Gamma(x+1)=x\Gamma(x)\,$ for $x>0$ and
$$\Gamma(n+\frac{k}{2})=(n+\frac{k}{2}-1)!$$
we obtain
\begin{align*}
&\frac{\Gamma(n+\frac{k}{2})\Gamma(n+l-\frac{k}{2}+\frac{5}{2})}{n!\Gamma(n+l+\frac{3}{2})}\\
&=\frac{(n+\frac{k}{2}-1)!\Gamma(n+l_1+\frac{k}{2}+\frac{5}{2})}{n!\Gamma(n+l_1+k+\frac{3}{2})}\\
&=\frac{(n+\frac{k}{2}-1)(n+\frac{k}{2}-2)\times\cdots\times(n+1)}{(n+l_1+k+\frac{1}{2})(n+l_1+k-\frac{1}{2})\times\cdots\times(n+l_1+\frac{k}{2}+\frac{5}{2})}\\
&=\prod^{\frac{k}{2}-1}_{j=1}\frac{n+j}{n+l_1+\frac{k}{2}+\frac{3}{2}+j}.
\end{align*}
Applying the elementary inequality, for any $1\leq j\leq \frac{k}{2}-1$,
$$\frac{n+j}{n+l_1+\frac{k}{2}+\frac{3}{2}+j}\leq \frac{n+\frac{k}{2}}{n+l_1+k+\frac{3}{2}}\leq\frac{n+\frac{k}{2}}{n+l_1+k}$$
and
$$
\frac{k}{2}-1\geq\frac{k}{4}\,\, \text{for}\,\, k\ge 20,
$$
{\color{black} we get}
\begin{align*}
\prod^{\frac{k}{2}-1}_{j=1}\frac{n+j}{n+l_1+\frac{k}{2}+{\color{black}\frac{3}{2}}+j}\leq
\Big(\frac{n+\frac{k}{2}}{n+l}\Big)^{\frac{k}{2}-1}\leq\Big(1-\frac{\frac{k}{2}+l_1}{n+l}\Big)^{\frac{k}{4}}=e^{\frac{k}{4}\log(1-\frac{\frac{k}{2}+l_1}{n+l})}.
\end{align*}
{\color{black} Using} the inequality,
$$\log(1+x)\leq x,\,\forall\,x>-1,$$
and {\color{black} recalling} that $l=l_1+k$ with $l_1\geq0$,
$$\frac{l}{2}\leq l_1+\frac{k}{2}<l\leq n+l,$$
we have
$$\frac{k}{4}\log(1-\frac{\frac{k}{2}+l_1}{n+l})\leq -\frac{k}{4}\frac{l_1+\frac{k}{2}}{n+l}\leq-\frac{kl}{8(n+l)}.$$
It follows that
\begin{align*}
\prod^{\frac{k}{2}-1}_{j=1}\frac{n+j}{n+l_1+\frac{k}{2}+{\color{black}\frac{3}{2}}+j}\leq e^{\frac{k}{4}\log(1-\frac{\frac{k}{2}+l_1}{n+l})}\leq\,e^{-\frac{lk}{8(n+l)}}.
\end{align*}
Therefore, {\color{black}the formula \eqref{important} follows when $k$ is even.}

{\color{black}Analogously, when $k$ is odd, we set $l=l_1+k$ with $l_1\geq0$.
Using the recurrence formula $\Gamma(x+1)=x\Gamma(x)\,$ for $x>0$ and
$$\Gamma(n+l-\frac{k}{2}+\frac{5}{2})=(n+l-\frac{k}{2}+\frac{3}{2})!=(n+l_1+\frac{k}{2}+\frac{3}{2})!$$
{\color{black} we obtain}
\begin{align*}
&\frac{\Gamma(n+\frac{k}{2})\Gamma(n+l-\frac{k}{2}+\frac{5}{2})}{n!\Gamma(n+l+\frac{3}{2})}\\
&=\frac{\Gamma(n+\frac{k}{2})(n+l_1+\frac{k}{2}+\frac{3}{2})!}{n!\Gamma(n+l_1+k+\frac{3}{2})}\\
&=\frac{(n+l_1+\frac{k}{2}+\frac{3}{2})(n+l_1+\frac{k}{2}+\frac{1}{2})\times\cdots\times(n+1)}{(n+l_1+k+\frac{1}{2})(n+l_1+k-\frac{1}{2})\times\cdots\times(n+\frac{k}{2})}\\
&=\prod^{l_1+\frac{k}{2}+\frac{3}{2}}_{j=1}\frac{n+j}{n+\frac{k}{2}-1+j}.
\end{align*}
{\color{black} Since $k\ge 20$, we have the elementary inequality : for any $1\leq j\leq l_1+\frac{k}{2}+\frac{3}{2}$, }
$$\frac{n+j}{n+\frac{k}{2}-1+j}\leq \frac{n+l_1+\frac{k}{2}+\frac{3}{2}}{n+l_1+k+\frac{1}{2}}\leq \frac{n+l_1+\frac{3k}{4}}{n+l_1+k} .$$
Therefore
\begin{align*}
\prod^{l_1+\frac{k}{2}+\frac{3}{2}}_{j=1}\frac{n+j}{n+\frac{k}{2}-1+j}\leq
\Big(\frac{n+l_1+\frac{3k}{4}}{n+l_1+k}\Big)^{{\color{red}l_1+\frac{k}{2}+\frac{3}{2}}}\leq\Big(1-\frac{\frac{k}{4}}{n+l}\Big)^{{\color{red}l_1+\frac{k}{2}+\frac{3}{2}}}=e^{({\color{red}l_1+\frac{k}{2}+\frac{3}{2}})\log(1-\frac{k}{4(n+l)})}.
\end{align*}
{\color{black} Using} the inequality,
$$\log(1+x)\leq x,\,\forall\,x>-1,$$
and {\color{black} recalling} that $l=l_1+k$ with $l_1\geq0$, $k\ge 20$
$$\frac{l}{2}\leq {\color{red}l_1+\frac{k}{2}+\frac{3}{2}}<l\leq n+l,$$
we have
$$({\color{red}l_1+\frac{k}{2}+\frac{3}{2}})\log(1-\frac{k}{4(n+l)})\leq -\frac{k}{4(n+l)}({\color{red}l_1+\frac{k}{2}+\frac{3}{2}})\leq-\frac{kl}{8(n+l)}.$$
Therefore, the formula \eqref{important} holds true for $k$ is odd.
This ends the proof of the formula \eqref{important}.
}For $0<\omega<1$, if
$$\frac{k^{1-\omega}l}{n+l}>1,$$
then
$$\frac{\Gamma(n+\frac{k}{2})\Gamma(n+l-\frac{k}{2}+\frac{5}{2})}{n!\Gamma(n+l+\frac{3}{2})}\lesssim\,e^{-\frac{1}{8}k^{\omega}}.$$
This implies that, when $n+l<k^{1-\omega}l$,
\begin{align*}
\lambda^k_{n,\tilde{n},l,\tilde{l}}&\lesssim\frac{\tilde{l}\sqrt{l}}{l+\tilde{l}-2k+1}
\frac{(\tilde{n}+\tilde{l})^{2s}}{(n+l+1)^{\frac{5}{2}+2s}}e^{-\frac{1}{8}k^{\omega}}.
\end{align*}
This is the result of the estimate $(i)$.\,\,\\
If $n+l\geq\,k^{1-\omega}l$, we deduce from the same estimate as \eqref{leq1} that,
$$\frac{\Gamma(n+\frac{k}{2})\Gamma(n+l-\frac{k}{2}+\frac{5}{2})}{n!\Gamma(n+l+\frac{3}{2})}\leq1.$$
The estimate $(ii)$ follows from the estimate \eqref{mid}.\,\, This ends the proof of Lemma \ref{relambda}.
\end{proof}

\subsection{The proof of $3)$ in {\color{black} Proposition} \ref{est}}

\begin{proof}
For $\lambda^k_{n,\tilde{n},l,\tilde{l}}$ defined in \eqref{lambdadef}, by the analysis in \eqref{es2}, we {\color{black} only} need to prove
$$\sum_{\substack{n+\tilde{n}+k=n^{\star}\\n+l\geq2,\tilde{n}+\tilde{l}\geq2\\n\geq0,\tilde{n}\geq0}}
\sum_{\substack{l+l-2k=l^{\star}\\l\geq1,\tilde{l}\geq1\\0\leq\,k\,\leq\min(l,\tilde{l})}}
\frac{\lambda^k_{n,\tilde{n},\tilde{l},l}}{\lambda_{\tilde{n},\tilde{l}}}\leq\,C\lambda_{n^{\star},l^{\star}}.$$

The constraint of the above summation is
\begin{align*}
\Lambda_{n^{\star},l^{\star}}
&=\Big\{
(n,\tilde{n},l,\tilde{l},k)\in\mathbb{N}^5;\,n+\tilde{n}+k=n^{\star},\,l+\tilde{l}-2k=l^{\star},\\
&\qquad l\geq1,\tilde{l}\geq1,0\leq k\leq \min(l,\tilde{l}),\,\,n+l\geq2,\,\tilde{n}+\tilde{l}\geq2
\Big\},
\end{align*}
{\color{black}
which is a subset of a hyperplane of dimension 3.
By using Lemma \ref{relambda} with $\omega=\frac{1}{4}$, we divide the sets into four parts,
\begin{align*}
\Lambda_{n^{\star},l^{\star}}&=\Lambda^1_{n^{\star},l^{\star}}\bigcup\Lambda^2_{n^{\star},l^{\star}}\bigcup\Lambda^3_{n^{\star},l^{\star}}\bigcup\Lambda^4_{n^{\star},l^{\star}}
\end{align*}
with the sets
\begin{align*}
\Lambda^1_{n^{\star},l^{\star}}=&\Big\{
(n,\tilde{n},l,\tilde{l},k)\in\mathbb{N}^5;\,n+\tilde{n}=n^{\star},\,l+\tilde{l}=l^{\star},k=0\\
&\qquad\qquad\qquad\qquad l\geq1,\tilde{l}\geq1,\,n+l\geq2,\,\tilde{n}+\tilde{l}\geq2
\Big\};\\
\Lambda^2_{n^{\star},l^{\star}}=&\Big\{
(n,\tilde{n},l,\tilde{l},k)\in\mathbb{N}^5;\,n+\tilde{n}+k=n^{\star},\,l+\tilde{l}-2k=l^{\star},\\
&\qquad\qquad l\geq1,\tilde{l}\geq1,1\leq {\color{red}k\leq \min(19,l,\tilde{l})},\,n+l\geq2,\,\tilde{n}+\tilde{l}\geq2
\Big\};\\
\Lambda^3_{n^{\star},l^{\star}}
=&\Big\{
(n,\tilde{n},l,\tilde{l},k)\in\mathbb{N}^5;\,n+\tilde{n}+k=n^{\star},\,l+\tilde{l}-2k=l^{\star},n+l<k^{\frac{3}{4}}l,\\
&\qquad\qquad l\geq1,\tilde{l}\geq1,20\leq k\leq\min(l,\tilde{l}),\,n+l\geq2,\,\tilde{n}+\tilde{l}\geq2
\Big\};\\
\Lambda^4_{n^{\star},l^{\star}}
=&\Big\{
(n,\tilde{n},l,\tilde{l},k)\in\mathbb{N}^5;\,n+\tilde{n}+k=n^{\star},\,l+\tilde{l}-2k=l^{\star},n+l\geq k^{\frac{3}{4}}l,\\
&\qquad\qquad l\geq1,\tilde{l}\geq1,20\leq k\leq\min(l,\tilde{l}),\,n+l\geq2,\,\tilde{n}+\tilde{l}\geq2
\Big\}.
\end{align*}
Then the summation can be divided into four terms corresponding to the sets.
\begin{align*}
&\sum_{(n,\tilde{n},l,\tilde{l},k)\in\Lambda_{n^{\star},l^{\star}}}
\frac{\lambda^k_{n,\tilde{n},\tilde{l},l}}{\lambda_{\tilde{n},\tilde{l}}}\\
&=\sum_{(n,\tilde{n},l,\tilde{l})\in\Lambda^1_{n^{\star},l^{\star}}}
\frac{\lambda^0_{n,\tilde{n},\tilde{l},l}}{\lambda_{\tilde{n},\tilde{l}}}
+\sum_{(n,\tilde{n},l,\tilde{l},k)\in\Lambda^2_{n^{\star},l^{\star}}}
\frac{\lambda^k_{n,\tilde{n},\tilde{l},l}}{\lambda_{\tilde{n},\tilde{l}}}
+\sum_{(n,\tilde{n},l,\tilde{l},k)\in\Lambda^3_{n^{\star},l^{\star}}}
\frac{\lambda^k_{n,\tilde{n},\tilde{l},l}}{\lambda_{\tilde{n},\tilde{l}}}
+\sum_{(n,\tilde{n},l,\tilde{l},k)\in\Lambda^4_{n^{\star},l^{\star}}}
\frac{\lambda^k_{n,\tilde{n},\tilde{l},l}}{\lambda_{\tilde{n},\tilde{l}}}\\
&=\bf{K_1}+\bf{K_2}+\bf{K_3}+\bf{K_4}.
\end{align*}
}
{\color{black}By using \eqref{estlambda0-b} in Lemma \ref{est-b}} with $s_0=\min(s,1-s)$, one can estimate $\bf{K_1}$ as follows,
\begin{align*}
{\bf{K_1}}\lesssim\sum_{\substack{n+\tilde{n}=n^{\star}\\n+l\geq2,\tilde{n}+\tilde{l}\geq2\\n\geq0,\tilde{n}\geq0}}
\sum_{\substack{l+l=l^{\star}\\l\geq1,\tilde{l}\geq1}}\frac{\tilde{l}\sqrt{l}}{l+\tilde{l}+1}\frac{(\tilde{n}+\tilde{l})^{2s+s_0}}{(\tilde{n}+1)^{s_0}
(n+1)^{1-s_0}(n+l)^{\frac{3}{2}+2s+s_0}\lambda_{\tilde{n},\tilde{l}}}.
\end{align*}
We claim that, for $(\tilde{n},\tilde{l})\in \mathbb{N}^2$, $\tilde{n}+\tilde{l}\geq2$,
\begin{equation}\label{leq2}
(\tilde{n}+\tilde{l})^{2s}\leq{\color{black}2(\tilde{n}+1)^s\big(\tilde{n}^s+\tilde{l}^{2s}\big)}.
\end{equation}
{\color{black}
Indeed, the formula \eqref{leq2} holds for $\tilde{n}=0$ or $\tilde{l}=0$ under the assumption of $\tilde{n}+\tilde{l}\geq2$.
Now we assume $\tilde{n}\geq1$ and $\tilde{l}\geq1$.  In fact, for $0<s<1$,
$$(\tilde{n}+\tilde{l})^{2s}\leq{\color{black}2(\tilde{n}^{2s}+\tilde{l}^{2s})\leq 2(\tilde{n}+1)^s\big(\tilde{n}^s+\tilde{l}^{2s}\big)}.$$
This ends the proof of the formula \eqref{leq2}.
By using the inequality $\lambda_{\tilde{n},\tilde{l}}\gtrsim \tilde{n}^s+\tilde{l}^{2s}$ in \eqref{eq:3.111} and \eqref{leq2}, we have
{\color{black}$$\frac{(\tilde{n}+\tilde{l})^{2s}}{\lambda_{\tilde{n},\tilde{l}}}\lesssim(\tilde{n}+1)^s.$$}
Then
\begin{align*}
\bf{K_1}&\lesssim\sum_{\substack{n+\tilde{n}=n^{\star}\\n+l\geq2,\tilde{n}+\tilde{l}\geq2\\n\geq0,\tilde{n}\geq0}}
\sum_{\substack{l+l=l^{\star}\\l\geq1,\tilde{l}\geq1}}
\frac{\tilde{l}\sqrt{l}}{l+\tilde{l}+1}
\frac{(\tilde{n}+1)^s(\tilde{n}+\tilde{l})^{s_0}}{(\tilde{n}+1)^{s_0}
(n+1)^{1-s_0}(n+l)^{\frac{3}{2}+2s+s_0}}
\\
&\lesssim
\sum_{\substack{n+\tilde{n}=n^{\star}\\n\geq0,\tilde{n}\geq0}}
\sum_{\substack{l+\tilde{l}=l^{\star}\\l\geq1,\tilde{l}\geq1}}
\frac{(\tilde{n}+1)^s}{(n+1)^{1-
s_0}(n+l)^{1+2s+s_0}}\Big(1+\frac{\tilde{l}}{\tilde{n}+1}\Big)^{s_0}\\
&\lesssim
\sum_{\substack{n+\tilde{n}=n^{\star}\\n\geq0,\tilde{n}\geq0}}
\sum_{\substack{l+\tilde{l}=l^{\star}\\l\geq1,\tilde{l}\geq1}}
\frac{(\tilde{n}+1)^s}{(n+1)^{1+s
}(n+l)^{1+s}}\Big(1+\frac{\tilde{l}}{\tilde{n}+1}\Big)^{s_0}.
\end{align*}
 Considering $0<s_0=\min(1-s,s)<1$ and
using the elementary inequality,
$$\Big(1+\frac{\tilde{l}}{\tilde{n}+1}\Big)^{s_0}\leq1+\Big(\frac{\tilde{l}}{\tilde{n}+1}\Big)^{s_0},$$
we have
\begin{align*}
\bf{K_1}
&\lesssim\sum_{\substack{n+\tilde{n}=n^{\star}\\n\geq0,\tilde{n}\geq0}}
\sum_{\substack{l+\tilde{l}=l^{\star}\\l\geq1,\tilde{l}\geq1}}
\frac{(\tilde{n}+1)^s}{(n+1)^{1+s
}(n+l)^{1+s}}\Bigg(1+\Big(\frac{\tilde{l}}{\tilde{n}+1}\Big)^{s_0}\Bigg)
\end{align*}
Recall that $s_0=\min(s,1-s)$, if $s_0=s$, one gets
$$(\tilde{n}+1)^s\Bigg(1+\Big(\frac{\tilde{l}}{\tilde{n}+1}\Big)^{s_0}\Bigg)=(\tilde{n}+1)^s+(\tilde{l})^s;$$
if $s_0<s$, from Young's inequality,
{\color{red}
$$\tilde{l}^{s_0}(\tilde{n}+1)^{s-s_0}\leq \frac{s-s_0}{s}(\tilde{n}+1)^{s}+\frac{s_0}{s}\tilde{l}^{s},$$
}we conclude that,
$$(\tilde{n}+1)^s\Bigg(1+\Big(\frac{\tilde{l}}{\tilde{n}+1}\Big)^{s_0}\Bigg)=(\tilde{n}+1)^s+\tilde{l}^{s_0}(\tilde{n}+1)^{s-s_0}\lesssim(\tilde{n}+1)^{s}+\tilde{l}^{s}.$$
Therefore,
\begin{align*}
\bf{K_1}
&\lesssim\sum_{\substack{n+\tilde{n}=n^{\star}\\n\geq0,\tilde{n}\geq0}}
\sum_{\substack{l+\tilde{l}=l^{\star}\\l\geq1,\tilde{l}\geq1}}
\frac{(\tilde{n}+1)^s+\tilde{l}^s}{(n+1)^{1+s
}(n+l)^{1+s}}\\
&\lesssim
 (n^{\star}+l^{\star})^s
 \Bigg[\sum_{\substack{n+\tilde{n}=n^{\star}\\n\geq0,\tilde{n}\geq0}}\frac{1}{(n+1)^{1+s}}
\sum_{\substack{l+\tilde{l}=l^{\star}\\l\geq1,\tilde{l}\geq1}}
\frac{1}{(n+l)^{1+s}}
\Bigg]\\
&\lesssim
 (n^{\star}+l^{\star})^s
 \Bigg[\sum_{\substack{n+\tilde{n}=n^{\star}\\n\geq0,\tilde{n}\geq0}}\frac{1}{(n+1)^{1+s}}
\sum_{\substack{l+\tilde{l}=l^{\star}\\l\geq1,\tilde{l}\geq1}}
\frac{1}{l^{1+s}}
\Bigg]
\end{align*}
The set $\{(n,\tilde{n})\in\mathbb{N}^2, n+\tilde{n}=n^{\star}\}$ is a subset of hyperplane of dimension 1,
$$\sum_{\substack{n+\tilde{n}=n^{\star}\\n\geq0,\tilde{n}\geq0}}\frac{1}{(n+1)^{1+s}}=\sum^{n^{\star}}_{n=0}\frac{1}{(n+1)^{1+s}}\leq\sum^{+\infty}_{n=1}\frac{1}{n^{1+s}}<+\infty.$$
Similarly,
$$\sum_{\substack{l+\tilde{l}=l^{\star}\\l\geq1,\tilde{l}\geq1}}
\frac{1}{l^{1+s}}=\sum^{l^{\star}-1}_{l=1}\frac{1}{l^{1+s}}\leq{\color{red}\sum^{+\infty}_{l=1}}\frac{1}{l^{1+s}}<+\infty.$$
Therefore,
\begin{align*}
\bf{K_1}&\lesssim (n^{\star}+l^{\star})^s.
\end{align*}
The estimate of the second term $\bf{K_2}$ :\\
{\color{black}By using \eqref{estlambdak-b} in Lemma \ref{est-b}}, we have
\begin{align*}
\bf{K_2}&\lesssim\sum^{\min(19,n^{\star})}_{k=1}
\sum_{\substack{n+\tilde{n}=n^{\star}-k\\n+l\geq2,\tilde{n}+\tilde{l}\geq2\\n\geq0,\tilde{n}\geq0}}
\sum_{\substack{l+\tilde{l}=l^{\star}+2k\\l\geq\,k,\tilde{l}\geq\,k}}
\frac{\tilde{l}\sqrt{l}}{l+\tilde{l}-2k+1}
\frac{\tilde{n}^s(\tilde{n}+\tilde{l})^{s}}{(n+1)^s(n+l+1)^{\frac{5}{2}+s}\lambda_{\tilde{n},\tilde{l}}}
\end{align*}
Applying the inequality \eqref{eq:3.111} that $\lambda_{\tilde{n},\tilde{l}}\gtrsim (\tilde{n}+\tilde{l})^{s}$, one can re-estimated $\bf{K_2}$ as
\begin{align*}
\bf{K_2}&\lesssim\sum^{\min(19,n^{\star})}_{k=1}
\sum_{\substack{n+\tilde{n}=n^{\star}-k\\n+l\geq2,\tilde{n}+\tilde{l}\geq2\\n\geq0,\tilde{n}\geq0}}
\sum_{\substack{l+\tilde{l}=l^{\star}+2k\\l\geq\,k,\tilde{l}\geq\,k}}
\frac{\tilde{l}\sqrt{l}}{l+\tilde{l}-2k+1}
\frac{\tilde{n}^s}{(n+1)^s(n+l+1)^{\frac{5}{2}+s}}\\
&\lesssim(n^{\star})^{s}\sum^{19}_{k=1}\sum^{n^{\star}-k}_{n=0}
\sum^{l^{\star}+k}_{l=k}\frac{l^{\star}+2k-l}{l^{\star}+1}
\frac{1}{(n+1)^s(n+l+1)^{2+s}}\\
&\lesssim(n^{\star})^{s}\sum^{19}_{k=1}\sum^{n^{\star}-k}_{n=0}
\sum^{l^{\star}+k}_{l=k}\frac{l^{\star}+19}{l^{\star}+1}
\frac{1}{(n+1)^s(n+l+1)^{2+s}}.
\end{align*}
Since
\begin{align*}
&\sum^{l^{\star}+k}_{l=k}\frac{1}{(n+l+1)^{2+s}}\\
&\leq\sum^{l^{\star}+k}_{l=k}\int^{l}_{l-1}\frac{1}{(n+x+1)^{2+s}}dx=\int^{l^{\star}+k}_{k-1}\frac{1}{(n+x+1)^{2+s}}dx\\
&=\frac{1}{1+s}\left(\frac{1}{(n+k)^{1+s}}-\frac{1}{(n+k+l^{\star}+1)^{1+s}}\right)\lesssim\frac{1}{(n+k)^{1+s}},
\end{align*}
we can estimate that
\begin{align*}
\bf{K_2}&\lesssim(n^{\star})^{s}\sum^{19}_{k=1}\sum^{n^{\star}-k}_{n=0}\frac{1}{(n+k)^{1+s}}\\
&\lesssim(n^{\star})^{s}\sum^{n^{\star}}_{n=0}\frac{1}{(n+1)^{1+s}}\lesssim(n^{\star})^{s}.
\end{align*}
Now we consider $k\geq 20$.  In this case, we assume $n^{\star}\geq20.$ Or else, $$\Lambda^3_{n^{\star},l^{\star}}=\varnothing,\Lambda^4_{n^{\star},l^{\star}}=\varnothing.$$
For the third term $\bf{K_3}$, by using $(i)$ of Lemma \ref{relambda} with $\omega=\frac{1}{4}$, we have
\begin{align*}
\bf{K_3}&\lesssim\sum^{n^{\star}}_{k=20}
\sum_{\substack{n+\tilde{n}=n^{\star}-k\\2\leq n+l<k^{\frac{3}{4}}l,\tilde{n}+\tilde{l}\geq2\\n\geq0,\tilde{n}\geq0}}
\sum_{\substack{l+\tilde{l}=l^{\star}+2k\\l\geq\,k,\tilde{l}\geq\,k}}
\frac{\tilde{l}\sqrt{l}}{l+\tilde{l}-2k+1}
\frac{(\tilde{n}+\tilde{l})^{2s}}{(n+l+1)^{\frac{5}{2}+2s}\lambda_{\tilde{n},\tilde{l}}}e^{-\frac{1}{8}k^{\frac{1}{4}}}
\end{align*}
Applying the inequality \eqref{eq:3.111} that $\lambda_{\tilde{n},\tilde{l}}\gtrsim (\tilde{n}+\tilde{l})^{s}$,  we estimate that,
\begin{align*}
\bf{K_3}&\lesssim\sum^{n^{\star}}_{k=20}
\sum_{\substack{n+\tilde{n}=n^{\star}-k\\2\leq n+l<k^{\frac{3}{4}}l,\tilde{n}+\tilde{l}\geq2\\n\geq0,\tilde{n}\geq0}}
\sum_{\substack{l+\tilde{l}=l^{\star}+2k\\l\geq\,k,\tilde{l}\geq\,k}}
\frac{\tilde{l}\sqrt{l}}{l+\tilde{l}-2k+1}
\frac{(\tilde{n}+\tilde{l})^{s}}{(n+l+1)^{\frac{5}{2}+2s}}e^{-\frac{1}{8}k^{\frac{1}{4}}}\\
&\lesssim\sum^{n^{\star}}_{k=20}
\sum^{n^{\star}-k}_{n=0}\sum^{l^{\star}+k}_{l=k}\frac{l^{\star}+k}{l^{\star}+1}
\frac{(n^{\star}+l^{\star})^{s}}{(n+l+1)^{2+2s}}
e^{-\frac{1}{4}k^{\frac{1}{4}}}\\
&=(n^{\star}+l^{\star})^{s}\sum^{n^{\star}}_{k=20}e^{-\frac{1}{4}k^{\frac{1}{4}}}\frac{l^{\star}+k}{l^{\star}+1}
\sum^{n^{\star}-k}_{n=0}\sum^{l^{\star}+k}_{l=k}\frac{1}{(n+l+1)^{2+2s}}.
\end{align*}
Use the estimate
\begin{align*}
&\sum^{l^{\star}+k}_{l=k}\frac{1}{(n+l+1)^{2+2s}}\\
&\leq\sum^{l^{\star}+k}_{l=k}\int^{l}_{l-1}\frac{1}{(n+x+1)^{2+2s}}dx=\int^{l^{\star}+k}_{k-1}\frac{1}{(n+x+1)^{2+2s}}dx\\
&=\frac{1}{1+2s}\left(\frac{1}{(n+k)^{1+2s}}-\frac{1}{(n+k+l^{\star}+1)^{1+2s}}\right)\\
&=\frac{1}{1+2s}\frac{(n+k+l^{\star}+1)^{1+2s}-(n+k)^{1+2s}}{(n+k)^{1+2s}(n+k+l^{\star}+1)^{1+2s}},
\end{align*}
and the mean value theorem for $f(x)=x^{1+2s}$,
$$(n+k+l^{\star}+1)^{1+2s}-(n+k)^{1+2s}=(1+2s)\big(n+k+\tau (l^{\star}+1)\big)^{2s}(l^{\star}+1),\,\,(0<\tau<1),$$
we obtain,
\begin{align*}
\bf{K_3}
&\lesssim(n^{\star}+l^{\star})^{s}\sum^{n^{\star}}_{k=1}e^{-\frac{1}{8}k^{\frac{1}{4}}}
\sum^{n^{\star}-k}_{n=0}\frac{l^{\star}+1}{(n+k)^{1+2s}(n+l^{\star}+k)}\times\frac{l^{\star}+k}{l^{\star}+1}\\
&\lesssim(n^{\star}+l^{\star})^{s}\sum^{n^{\star}}_{k=1}e^{-\frac{1}{4}k^{\frac{1}{4}}}
\sum^{n^{\star}}_{n=0}\frac{1}{(n+1)^{1+2s}}\\
&\lesssim(n^{\star}+l^{\star})^{s}.
\end{align*}
Finally, we estimate the remaining term $\bf{K_4}$.\,\,By using $(ii)$ of Lemma \ref{relambda} with $\omega=\frac{1}{4}$, we have
\begin{align*}
\bf{K_4}&\lesssim\sum^{n^{\star}}_{k=20}
\sum_{\substack{n+\tilde{n}=n^{\star}-k\\n+l\geq k^{\frac{3}{4}}l,\tilde{n}+\tilde{l}\geq2\\n\geq0,\tilde{n}\geq0}}
\sum_{\substack{l+\tilde{l}=l^{\star}+2k\\l\geq\,k,\tilde{l}\geq\,k}}
\frac{\tilde{l}\sqrt{l}}{l+\tilde{l}-2k+1}
\frac{(\tilde{n}+\tilde{l})^{2s}}{(n+l+1)^{\frac{5}{2}+2s}\lambda_{\tilde{n},\tilde{l}}}
\end{align*}
Applying the inequality \eqref{eq:3.111} that $\lambda_{\tilde{n},\tilde{l}}\gtrsim (\tilde{n}+\tilde{l})^{s}$,  we estimate that
\begin{align*}
\bf{K_4}&\lesssim\sum^{n^{\star}}_{k=20}
\sum_{\substack{n+\tilde{n}=n^{\star}-k\\n+l\geq k^{\frac{3}{4}}l,\tilde{n}+\tilde{l}\geq2\\n\geq0,\tilde{n}\geq0}}
\sum_{\substack{l+\tilde{l}=l^{\star}+2k\\l\geq\,k,\tilde{l}\geq\,k}}
\frac{\tilde{l}\sqrt{l}}{l+\tilde{l}-2k+1}
\frac{(\tilde{n}+\tilde{l})^{s}}{(n+l+1)^{\frac{5}{2}+2s}}
\end{align*}
Considering the  condition
$$n+l\geq\,k^{\frac{3}{4}}l,$$
we obtain
$$\frac{1}{(n+l)^{\frac{3}{2}}}\leq\,\frac{1}{k^{\frac{9}{8}}l^{\frac{3}{2}}}.$$
Then $\bf{K_4}$ can be rewritten as
\begin{align*}
\bf{K_4}&\lesssim\sum^{n^{\star}}_{k=20}
\sum_{\substack{n+\tilde{n}=n^{\star}-k\\n+l\geq\,k^{\frac{3}{4}}l,\tilde{n}+\tilde{l}\geq2\\n\geq0,\tilde{n}\geq0}}
\sum_{\substack{l+\tilde{l}=l^{\star}+2k\\l\geq\,k,\tilde{l}\geq\,k}}\frac{\tilde{l}\sqrt{l}}{l+\tilde{l}-2k+1}
\frac{(\tilde{n}+\tilde{l})^{s}}{(n+l)^{1+2s}k^{\frac{9}{8}}l^{\frac{3}{2}}}\\
&\lesssim\Big(n^{\star}+l^{\star}\Big)^s\sum^{n^{\star}}_{k=20}\frac{1}{k^{\frac{9}{8}}}
\sum^{n^{\star}-k}_{n=0}\sum^{l^{\star}+k}_{l=k}\frac{1}{{\color{red}(n+l)^{1+2s}}l}\frac{l^{\star}+k}{l^{\star}+1}\\
&
\lesssim
\Big(n^{\star}+l^{\star}\Big)^s\sum^{n^{\star}}_{k=20}\frac{1}{k^{\frac{9}{8}}}\Big[\sum^{n^{\star}-k}_{n=0}\sum^{l^{\star}+k}_{l=k}\frac{1}{(n+l)^{1+2s}l}
        +\sum^{n^{\star}-k}_{n=0}\sum^{l^{\star}+k}_{l=k}\frac{1}{(n+l)^{1+2s}l}\frac{k}{l^{\star}+1}\Big]
\end{align*}
It is obviously that, for $k\geq1$,
$$\sum^{n^{\star}-k}_{n=0}\sum^{l^{\star}+k}_{l=k}\frac{1}{(n+l)^{1+2s}l}
\leq\sum^{n^{\star}}_{n=0}\frac{1}{(n+1)^{1+s}}\sum^{l^{\star}+k}_{l=k}\frac{1}{l^{1+s}}<+\infty.$$
At the same time,
\begin{align*}
\sum^{n^{\star}-k}_{n=0}\sum^{l^{\star}+k}_{l=k}\frac{1}{(n+l)^{1+2s}l}\frac{k}{l^{\star}+1}&\leq\sum^{n^{\star}-k}_{n=0}\frac{1}{(n+1)^{1+2s}}\sum^{l^{\star}+k}_{l=k}\frac{1}{l^{\star}+1}\\
&\leq\sum^{n^{\star}-k}_{n=0}\frac{1}{(n+1)^{1+2s}}<+\infty.
\end{align*}
Substitute these two estimates into the above inequality of $\bf{K_4}$, we have
\begin{align*}
{\bf{K_4}}
\lesssim\Big(n^{\star}+l^{\star}\Big)^s
\sum^{n^{\star}}_{k=20}\frac{1}{k^{\frac{9}{8}}}\lesssim\Big(n^{\star}+l^{\star}\Big)^s.
\end{align*}
}
Combining the estimates of $\bf{K_1}$, $\bf{K_2}$, $\bf{K_3}$ and $\bf{K_4}$, using again \eqref{eq:3.111}
$$(n^{\star}+l^{\star})^s+(l^{\star})^{2s}\lesssim\lambda_{n^{\star},l^{\star}},$$
we {\color{black} derive}
$$\sum_{\substack{n+\tilde{n}+k=n^{\star}\\n+l\geq2,\tilde{n}+\tilde{l}\geq2\\n\geq0,\tilde{n}\geq0}}
\sum_{\substack{l+l-2k=l^{\star}\\l\geq1,\tilde{l}\geq1\\0\leq\,k\,\leq\min(l,\tilde{l})}}
\frac{\lambda^k_{n,\tilde{n},\tilde{l},l}}{\lambda_{\tilde{n},\tilde{l}}}\leq\,C\lambda_{n^{\star},l^{\star}}.$$
This ends the proof of $3)$ in Proposition \ref{est}.
\end{proof}


\section{Appendix}\label{S6}

The important known results but really needed for this paper are presented in this section. For the self-content of paper,
we will present some proof of those properties.

\subsection{Gelfand-Shilov {\color{black} spaces}}

The symmetric Gelfand-Shilov space $S^{\nu}_{\nu}(\mathbb{R}^3)$ can be characterized through the decomposition
into the Hermite basis $\{H_{\alpha}\}_{\alpha\in\mathbb{N}^3}$ and the harmonic oscillator $\mathcal{H}=-\triangle +\frac{|v|^2}{4}$.
{\color{black} For} more details, see Theorem 2.1 in the book \cite{GPR}
\begin{align*}
f\in S^{\nu}_{\nu}(\mathbb{R}^3)&\Leftrightarrow\,f\in C^\infty (\mathbb{R}^3),\exists\, \tau>0, \|e^{\tau\mathcal{H}^{\frac{1}{2\nu}}}f\|_{L^2}<+\infty;\\
&\Leftrightarrow\, f\in\,L^2(\mathbb{R}^3),\exists\, \epsilon_0>0,\,\,\Big\|\Big(e^{\epsilon_0|\alpha|^{\frac{1}{2\nu}}}(f,\,H_{\alpha})_{L^2}\Big)_{\alpha\in\mathbb{N}^3}\Big\|_{l^2}<+\infty;\\
&\Leftrightarrow\,\exists\,C>0,\,A>0,\,\,\|(-\triangle +\frac{|v|^2}{4})^{\frac{k}{2}}f\|_{L^2(\mathbb{R}^3)}\leq AC^k(k!)^{\nu},\,\,\,k\in\mathbb{N}
\end{align*}
where
$$H_{\alpha}(v)=H_{\alpha_1}(v_1)H_{\alpha_2}(v_2)H_{\alpha_3}(v_3),\,\,\alpha\in\mathbb{N}^3,$$
and for $x\in\mathbb{R}$,
$$H_{n}(x)=\frac{(-1)^n}{\sqrt{2^nn!\pi}}e^{\frac{x^2}{2}}\frac{d^n}{dx^n}(e^{-x^2})
=\frac{1}{\sqrt{2^nn!\pi}}\Big(x-\frac{d}{dx}\Big)^n(e^{-\frac{x^2}{2}}).$$
For the harmonic oscillator $\mathcal{H}=-\triangle +\frac{|v|^2}{4}$ of 3-dimension and $s>0$, we have
$$
\mathcal{H}^{\frac{k}{2}} H_{\alpha} = (\lambda_{\alpha})^{\frac{k}{2}}H_{\alpha},\,\, \lambda_{\alpha}=\sum^3_{j=1}(\alpha_j+\frac{1}{2}),\,\,k\in\mathbb{N},\,\alpha\in\mathbb{N}^3.
$$

\subsection{Fourier Transform of special functions}
For the eigenvector $\varphi_{n,l,{m}}$ given in \eqref{v}, Lerner,  Morimoto, Pravda-Starov and Xu in \cite{NYKC3} shows in Lemma 3.2 the Fourier transform of  $\sqrt{\mu}\varphi_{n,0,0}$.
Following this work, we provide the Fourier transform of $\sqrt{\mu}\varphi_{n,l,{m}}$.
\begin{lemma}\label{lem trans}
Let $\alpha,\,\kappa\in S^2$\,and\,$r>0$,\, then
\begin{equation}\label{siga}
\int_{S_\kappa^2}e^{ir\kappa\cdot\alpha}Y_l^{m}(\kappa)d\kappa=(2\pi)^\frac{3}{2}i^l\Big(\frac{1}{r}\Big)^\frac{1}{2}J_{l+\frac{1}{2}}(r)Y_l^{m}(\alpha),
\end{equation}
where $\kappa\cdot\alpha$ denote the scalar product\,and $J_{l+\frac{1}{2}}$ is the Bessel function of $l+\frac{1}{2}$ order.
\end{lemma}

\begin{proof}
Since for any real r and $|z|\leq1$,
\begin{align}\label{bessel}
\sqrt{2r}e^{irz}=\sqrt{\pi}\sum^{\infty}_{k=0}(2k+1)J_{k+\frac{1}{2}}(r)(i)^kP_k(z)\quad{\color{black}\text{(cf. $(1)$ of Section 11.5 in \cite{Watson})}}.
\end{align}
Substituting $z=\kappa\cdot\alpha$ into \eqref{bessel},
$$e^{ir\kappa\cdot\alpha}=\sqrt{\frac{\pi}{2r}}\sum^{\infty}_{k=0}(2k+1)J_{k+\frac{1}{2}}(r)i^kP_k(\kappa\cdot\alpha).$$
{\color{black}
By using the definition of the Bessel function $J_{k+\frac{1}{2}}(r)$, see $(8)$ of Sec. 3.1, Chap.III in \cite{Watson},
$$J_{k+\frac{1}{2}}(r)=\sum^{\infty}_{m=0}\frac{(-1)^m(r/2)^{2m+k+\frac{1}{2}}}{m!\Gamma(m+k+\frac{3}{2})}.$$
Then
$$|J_{k+\frac{1}{2}}(r)|\leq(|r|/2)^{k+\frac{1}{2}}\sum^{\infty}_{m=0}\frac{(|r|/2)^{2m}}{m!k!}\leq(|r|/2)^{k+\frac{1}{2}}\frac{1}{k!}e^{\frac{|r|^2}{4}}$$
}
Now since $|P_k(\kappa\cdot\alpha)|\leq1$ and for every $r>0$,
$$\sqrt{\frac{1}{r}}\sum^\infty_{k=0}(2k+1)|J_{k+\frac{1}{2}}(r)|\leq\sum^\infty_{k=0}\frac{(\frac{r}{2})^k}{k!}e^{\frac{r^2}{4}}\leq e^{\frac{r}{2}+\frac{r^2}{4}},$$
we obtain that
\begin{equation}\label{1}
\int_{S_\kappa^2}e^{ir\kappa\cdot\alpha}Y_l^{m}(\kappa)d\kappa=\sqrt{\frac{\pi}{2r}}\sum^{\infty}_{k=0}(2k+1)J_{k+\frac{1}{2}}(r)i^k\int_{S_\kappa^2}P_k(\kappa\cdot\alpha)Y_l^{m}(\kappa)d\kappa.
\end{equation}
Consider the addition theorem of spherical harmonics $(7-34)$ in Chapter 7 of \cite{JCSlater} that
$$
P_k(\kappa\cdot\alpha)=\frac{4\pi}{2k+1}\sum^k_{m=-k}[Y_k^{m}(\kappa)]^{\ast}Y_k^{m}(\alpha),
$$
where $[Y_k^{m}(\kappa)]^{\ast}$ is the conjugate of $Y_k^{m}(\kappa)$.\,\,Then
$$
\int_{S_\kappa^2}P_k(\kappa\cdot\alpha)Y_l^{m}(\kappa)d\kappa=\frac{4\pi}{2l+1}Y_l^{m}(\alpha)\delta_{k,l}.
$$
Substitute this addition formula into \eqref{1}, the formula \eqref{siga} follows.
\end{proof}

Lemma \ref{lem trans} is the basis for calculating the Fourier transform of $\sqrt{\mu}\varphi_{n,l,m}.$
\begin{lemma}\label{Fouriertransform}
Let $\varphi_{n,l,{m}}$ be the functions defined in \eqref{v}, then for $n,l\in\mathbb{N},\,|m|\leq l$, we have
\begin{equation}\label{four}
\widehat{\sqrt{\mu}\varphi_{n,l,{m}}}(\xi)=(-i)^l(2\pi)^\frac{3}{4}\Big(\frac{1}{\sqrt{2}n!\Gamma(n+l+\frac{3}{2})}\Big)^\frac{1}{2}\Big(\frac{|\xi|}{\sqrt{2}}\Big)^{2n+l}e^{-\frac{|\xi|^2}{2}}Y_l^{m}\Big(\frac{\xi}{|\xi|}\Big).
\end{equation}
\end{lemma}

\begin{proof}Define $H(\xi)=\Big(\frac{|\xi|}{\sqrt{2}}\Big)^{2n+l}e^{-\frac{|\xi|^2}{2}}Y_l^{m}\Big(\frac{\xi}{|\xi|}\Big)$, and by Lemma \ref{lem trans} with $r=|v||\xi|$, ${\color{black}\alpha=\frac{v}{|v|}, \kappa=\frac{\xi}{|\xi|}}$, we can compute the inverse Fourier transform of $H$,
\begin{align}
\mathcal{F}^{-1}(H)(v)&=\Big(\frac{1}{2\pi}\Big)^3\int_{\mathbb{R}^3}e^{iv\cdot\xi}\Big(\frac{|\xi|}{\sqrt{2}}\Big)^{2n+l}e^{-\frac{|\xi|^2}{2}}Y_l^{m}\Big(\frac{\xi}{|\xi|}\Big)d\xi\nonumber\\
&=\frac{1}{4\pi^3}\int^\infty_{0}\Big(\frac{\rho}{\sqrt{2}}\Big)^{2n+l+2}e^{-\frac{\rho^2}{2}}\left(\int_{S^2_\kappa}e^{i|v|\rho\kappa\cdot\alpha}Y_l^{m}(\kappa)d\kappa\right)d\rho\nonumber\\
&=i^l\Big(\frac{1}{2\pi}\Big)^\frac{3}{2}Y^{m}_l\Big(\frac{v}{|v|}\Big)\Big(\frac{2\sqrt{2}}{|v|}\Big)^\frac{1}{2}
\int^\infty_{0}\Big(\frac{\rho}{\sqrt{2}}\Big)^{2n+l+\frac{3}{2}}e^{-\frac{\rho^2}{2}}J_{l+\frac{1}{2}}(|v|\rho)d\rho.\nonumber
\end{align}
By using the standard formula, see (6.2.15) in \cite{Ge},
$$L^{(l+\frac{1}{2})}_n(x)=\frac{e^xx^{-\frac{l+\frac{1}{2}}{2}}}{n!}\int^{+\infty}_{0}t^{n+\frac{l+\frac{1}{2}}{2}}J_{l+\frac{1}{2}}(2\sqrt{xt})e^{-t}dt,$$ we have
\begin{equation*}
L^{(l+\frac{1}{2})}_n(\frac{|v|^2}{2})=\sqrt{2}\frac{e^{\frac{|v|^2}{2}}(\frac{|v|}{\sqrt{2}})^{-(l+\frac{1}{2})}}{n!}\int^{\infty}_{0}(\frac{\rho}{\sqrt{2}})^{2n+l+\frac{3}{2}}J_{l+\frac{1}{2}}(|v|\rho)e^{-\frac{\rho^2}{2}}d\rho.
\end{equation*}
Therefore,
$$
\mathcal{F}^{-1}(H)(v)=(\frac{1}{2\pi})^\frac{3}{2}n!i^le^{-\frac{|v|^2}{2}}\Big(\frac{|v|}{\sqrt{2}}\Big)^{l}L^{(l+\frac{1}{2})}_n\Big(\frac{|v|^2}{2}\Big)Y^{m}_l\Big(\frac{v}{|v|}\Big).
$$
Recall the expression of \eqref{v},\,one can verify
$$\sqrt{\mu}\varphi_{n,l,m}(v)=(-i)^l(2\pi)^\frac{3}{4}\Big(\frac{1}{\sqrt{2}n!\Gamma(n+l+\frac{3}{2})}\Big)^\frac{1}{2}\mathcal{F}^{-1}(H)(v).$$
Henceforth, \eqref{four} yields.
\end{proof}

\subsection{Spherical Harmonics}
The following results with respect to the sphe\-rical harmonics is significant.\,\,For $l,\,\tilde{l}\in\mathbb{N}$,\,$|m|\leq l$,\,$|\tilde{m}|\leq \tilde{l}$,\,
\begin{equation}\label{represent}
Y_l^{m}Y^{\tilde{m}}_{\tilde{l}}=\sum_{l^{\prime}}\sum_{m^{\prime}}\left(\int_{\mathbb{S}^2}Y_l^{m}(\kappa)Y^{\tilde{m}}_{\tilde{l}}(\kappa)Y^{-m^{\prime}}_{l^{\prime}}(\kappa)d\kappa\right)Y^{m^{\prime}}_{l^{\prime}}
\end{equation}
where $|m^{\prime}|\leq l^{\prime}$ and $\tilde{l}-l\leq l^{\prime}\leq\tilde{l}+l$.\,\, More explicitly, in order to  have a non-vanishing integral, the parameters $m^{\prime},\,l^{\prime}$ satisfy
\begin{equation}\label{con}
m^{\prime}=m+\tilde{m},\,\,l^{\prime}=l+\tilde{l}-2j\,\, \text{with}\,\,0\leq j\leq\min(l,\tilde{l}).
\end{equation}
For more details, {\color{black}see (86) in Chap.\,3 in \cite{Jones} or Theorem 2.1 of Sec.2, Chap.IV in \cite{CM}}.\\

\bigskip
\noindent {\bf Acknowledgements.} The first author would like to express his sincere thanks to School of mathematics and statistics of Wuhan University for the invitations.
The second author would like to thanks the invitation of ``F\'ederation de math\'ematiques de Normandie" and `` laboratoire de math\'ema\-tiques'' of the Universit\'e de Rouen.
The research of the second author and
the last author is supported partially by ``The Fundamental Research
Funds for Central Universities'' and the National Science Foundation
of China No. 11171261.
 Finally, the authors are grateful to the referees
for careful reading and many useful comments
and suggestions for improvement in the article.

\end{document}